\documentclass[11pt,reqno]{amsart}
\usepackage[T1]{fontenc}
\usepackage{lmodern}

\usepackage{amssymb}
\usepackage{enumitem}
\usepackage[dvipsnames]{xcolor}

\usepackage[a4paper,top=2.5 cm,bottom=2.5 cm,left=2.5 cm,right=2.5 cm]{geometry}

\newtheorem{theorem}{Theorem}[section]
\newtheorem{lemma}{Lemma}[section]
\newtheorem{remark}{Remark}[section]
\newtheorem{proposition}{Proposition}[section]

\numberwithin{equation}{section}

\def\to{\rightarrow}

\DeclareMathOperator{\dive}{div}

\newcommand{\ba}{\begin{aligned}}
	\newcommand{\ea}{\end{aligned}}
\newcommand{\be}{\begin{equation}}
	\newcommand{\ee}{\end{equation}}

\allowdisplaybreaks

\usepackage[colorlinks,
linkcolor=red,
citecolor=blue
]{hyperref}
\usepackage{color}

\begin{document}
\title[Fluid-particle models in critical spaces]{Global existence and vanishing viscosity limit of the compressible Navier--Stokes--Vlasov--Fokker--Planck system in critical spaces}

\author{Hai-Liang Li and Ling-Yun Shou}

\begin{abstract}
We study multidimensional compressible fluid--particle systems at critical regularity, in which a carrier fluid and a particle phase with Fokker--Planck diffusion are coupled through a drag force. We prove the existence and uniqueness of strong solutions for the Cauchy problems of the Navier--Stokes--Vlasov--Fokker--Planck and Euler--Vlasov--Fokker--Planck systems near equilibrium in their respective critical Besov spaces. Moreover, we establish regularity estimates for the Navier--Stokes--Vlasov--Fokker--Planck system uniform with respect to the common viscosity parameter $\mu=\lambda=\varepsilon$ and justify the global-in-time vanishing-viscosity limit with the convergence rate $\mathcal O(\varepsilon)$. Finally, under an additional lower-order Besov assumption on the initial data, we obtain optimal time-decay estimates for both systems and derive enhanced decay rates for the relative velocity and the microscopic part of the distribution function.

\end{abstract}

\maketitle
\section{Introduction}

Fluid--particle systems describe the motion of dispersed particles in a carrier fluid and arise in applications such as spray combustion, aerosol transport and sedimentation \cite{boudin2,jabin1,williams1}. In this paper, we consider the following fluid--particle system
\begin{equation}\label{m1}
\left\{
\begin{aligned}
&\partial_t\rho+\dive_x(\rho u)=0,\\
&\partial_t(\rho u)
 +\dive_x(\rho u\otimes u)
 +\nabla_xP(\rho)
 =\mu\Delta_xu
 +(\mu+\lambda)\nabla_x\dive_xu
 -\int_{\mathbb R^d_v}(u-v)F\,dv,\\
&\partial_tF+v\cdot\nabla_xF
 +\dive_v\big((u-v)F-\nabla_vF\big)=0\quad\text{for}\quad
 (t,x,v)\in \mathbb{R}_+\times \mathbb{R}^d\times\mathbb{R}^d,
\end{aligned}
\right.
\end{equation}
where the unknowns $\rho=\rho(x,t)\geq 0$, $u=u(x,t)\in\mathbb{R}^d$ and
$F=F(x,v,t)\geq0$ denote the fluid density, the fluid velocity and the
particle distribution function, respectively. The fluid and particle
phases are coupled through the drag force $(u-v)F$. The pressure function $P$ satisfies
\begin{align*}
    P\in C^{\infty}((0,\infty))\quad\text{and}\quad P'(\rho)>0\quad\text{for}\quad \rho>0,
\end{align*}
and the shear viscosity $\mu$ and the bulk viscosity $\lambda$ are constants satisfying 
 \begin{align*}
 \mu>0\quad\text{and}\quad 2\mu+\lambda>0.
 \end{align*}
We refer to \eqref{m1} as the NS--VFP system. When
$\mu,\lambda\rightarrow 0$, we formally derive the compressible
Euler--Vlasov--Fokker--Planck system,
abbreviated as the Euler--VFP system.

Significant progress has been made in the mathematical theory of
fluid--particle models related to \eqref{m1}. For the incompressible
NS--VFP system, Goudon et al.\ \cite{goudon1} established global classical solutions near
equilibrium and studied their large-time behavior. Chae, Kang
and Lee \cite{chae1} proved global weak solutions in two and three
dimensions and global smooth solutions in two dimensions, while He \cite{he1} investigated the smoothing effect for arbitrary initial data in two
dimensions. In the
compressible case, Mellet and Vasseur \cite{mellet2} constructed
global weak solutions in a three-dimensional bounded domain. Global
classical solutions near equilibrium and their decay properties were
later obtained on the torus in \cite{chae2} and in the whole space in
\cite{lf1}. Li et al.\ \cite{lhl1} carried out the
spectral analysis, derived pointwise estimates and established the
optimal time-decay rates for the compressible system. Hydrodynamic
limits for this system and related kinetic--fluid models have also
been investigated under various
asymptotic regimes; refer to \cite{carrillo2,choi4,goudon2,goudon3,mellet1} and the references therein.

Although the fluid equation in the Euler--VFP system has no direct dissipation, the Fokker--Planck operator and the drag force can stabilize the inviscid system. For the
incompressible Euler--VFP system,
Carrillo, Duan and Moussa \cite{carrillo4} proved the global
existence of classical solutions near equilibrium and derived their
algebraic and exponential decay rates in the whole space and on the
torus, respectively. For the compressible Euler equations, Duan and
Liu \cite{duan3} established global well-posedness near equilibrium.

For both the compressible NS--VFP and Euler--VFP systems, Li et al.\ \cite{lhl1} carried out the
spectral analysis, derived pointwise estimates and established the
optimal time-decay rates. Li, Wang and Wang \cite{LWW} established the nonlinear stability of planar rarefaction
waves for both systems.
They also proved the zero-viscosity limit on every finite time interval
in Sobolev spaces. More precisely, when $\mu=\lambda=\varepsilon$,
their estimate gives a convergence rate
$\mathcal O(\sqrt{\varepsilon})$ on a finite time interval.
We also refer to \cite{LiNiShouWang2025,LiNiWang2026}
for recent results on the inviscid limit for related systems.

If the velocity diffusion is removed, the kinetic equation reduces to
a Vlasov equation. The drag force then drives the particle velocities
toward the fluid velocity, and the particle distribution is
expected to concentrate on a monokinetic state of the form
$n(t,x)\delta_{u(t,x)}(v)$. Hamdache
\cite{hamdache1} first established global weak solutions for the
Vlasov--Stokes system, and global weak solutions to the incompressible
Navier--Stokes--Vlasov system were later constructed in
\cite{boudinNSV1,wangd1,yuNSV1}. Baranger and Desvillettes
\cite{baranger2} also treated local classical solutions to a related
Euler--Vlasov spray model. For the one-dimensional Euler--Vlasov
system, global bounded weak entropy solutions for arbitrarily large
initial data were constructed in \cite{CaoJiang2021}. Bae et al.\ \cite{bae1}
studied conditional large-time behavior with an
additional nonlocal alignment force, and Choi \cite{choiNSV1} later
revisited the large-time behavior. Finite-time blow-up of classical
solutions to the Vlasov/compressible Navier--Stokes system and related
models was also studied by Choi \cite{choiblowup}. For the incompressible
Navier--Stokes--Vlasov system on the three-dimensional torus,
Han-Kwan, Moussa and Moyano \cite{han1} rigorously proved that the
particle distribution converges unconditionally to a Dirac mass in
velocity when the initial modulated energy is sufficiently small.
Recent results also include the exponential stability of the inhomogeneous
Navier--Stokes--Vlasov system with vacuum \cite{LiShouZhang2025} and the large-time
asymptotics of finite-energy weak solutions on the two-dimensional torus
\cite{DanchinShou2026}. For the compressible models in one dimension,
Li and Shou \cite{lishouNSV1,lishouNSV2} established global
well-posedness for general initial data and proved an exponential time convergence rate to equilibrium.

The previous results on global strong solutions for the fluid--particle systems discussed above were established in high-order Sobolev spaces. It is therefore natural to ask whether global well-posedness can be established at the critical regularity. Function spaces whose norms are invariant under the natural scaling of the equations are called critical spaces. Refer to \cite{crin1,danchin1,danchin4,xu1} for critical regularity theories of the compressible Navier--Stokes equations, the compressible Euler equations with damping and more general partially dissipative hyperbolic systems. The fluid--particle coupling in \eqref{m1}, however, leads to a substantially different difficulty. More specifically, the drag force produces quadratic terms without spatial derivatives, which cannot be controlled directly by the critical regularity for the compressible Navier--Stokes equations. Thus, the main issue is not merely to lower the regularity of the existing Sobolev theory, but to overcome the derivative loss generated by the drag coupling.

In this paper, we study the Cauchy problems for the NS--VFP and
Euler--VFP systems in their respective critical Besov spaces, with emphasis on the quantitative effects of the drag force
and Fokker--Planck dissipation on the regularity and
large-time behavior of solutions. For any dimension $d\ge2$, we prove global well-posedness of
strong solutions to the NS--VFP system for small initial
perturbations satisfying
$\varrho_0\in\dot B^{d/2-1}_{2,1}\cap\dot B^{d/2}_{2,1}$,
$u_0\in\dot B^{d/2-1}_{2,1}$ and
$f_0\in\dot B^{d/2-1}_{2,1}L^2_v$.
When $\mu=\lambda=\varepsilon\in(0,1]$ and the initial data
additionally have spatial regularity $\dot B^{d/2+1}_{2,1}$,
we obtain a priori estimates uniform in $\varepsilon$.
These estimates establish global well-posedness for the Euler--VFP
system and justify the vanishing-viscosity limit with the
convergence rate $\mathcal O(\varepsilon)$, uniformly for all time.
Under an additional uniform bound on the initial data in
lower-order $\dot{B}^{\sigma_1}_{2,\infty}$ Besov spaces with $-d/2\le\sigma_1<d/2-1$,
we establish optimal time-decay estimates for
$\sigma_1<\sigma\le d/2+1$, uniformly in
$\varepsilon\in[0,1]$.
Moreover, the relative velocity and the microscopic component exhibit
enhanced decay for $\sigma_1<\sigma\le d/2$, owing to the drag force and 
Fokker--Planck dissipation.

The paper is organized as follows. In Section \ref{s1}, we state the results in this paper. Section \ref{s2} concerns the global a priori estimates for the NS--VFP system and the proof of Theorem \ref{t1}. Section
\ref{s6} establishes uniform higher-order estimates and proves Theorems \ref{t2} and \ref{t3}. Sections
\ref{s9} and \ref{s10} are devoted to the proofs of Theorems \ref{t4} and \ref{t5} regarding the global
vanishing-viscosity and uniform decay estimates, respectively. The
appendices collect some analytic tools and the local well-posedness theory for large data in critical spaces.

\section{Main results}\label{s1}

We investigate the global dynamics of solutions to the Cauchy problem
\eqref{m1} for the NS--VFP system near the equilibrium $(1,0,M)$, where
\[
M(v):=(2\pi)^{-d/2}e^{-|v|^2/2}
\]
denotes the global Maxwellian. We define the perturbations
\[
\varrho:=\rho-1\quad\text{and}\quad \varrho_0:=\rho_0-1,
\quad
f:=\frac{F-M}{M^{1/2}}\quad\text{and}\quad  f_0:=\frac{F_0-M}{M^{1/2}}.
\]
Let $a$ and $b$ be the particle density perturbation
and momentum, respectively, defined by
\begin{equation*}
a:=\int_{\mathbb R^d_v}M^{1/2}f\,dv 
\quad\text{and}\quad 
b:=\int_{\mathbb R^d_v}vM^{1/2}f\,dv.
\end{equation*}
The Fokker--Planck operator associated with $M$ is
\begin{equation*}
\mathcal Lf
:=M^{-1/2}\nabla_v\cdot
\Big(M\nabla_v(M^{-1/2}f)\Big)
=\Delta_vf+\frac14(2d-|v|^2)f,
\end{equation*}
which is self-adjoint and non-positive in $L^2(\mathbb R^d_v)$, with $\ker\mathcal L=\operatorname{span}\{M^{1/2}\}$.  We follow \cite{duan3,guo2,liu1} and introduce the
$L^2(\mathbb R^d_v)$-orthogonal projection $\mathbf P:=\mathbf P_0+\mathbf P_1$
onto
\[
\operatorname{span}
\{M^{1/2},v_iM^{1/2}:1\le i\le d\},
\]
where $\mathbf P_0$ and $\mathbf P_1$ are, respectively, denoted by
\begin{equation*}
\mathbf P_0f:=aM^{1/2}
\quad\text{and}\quad 
\mathbf P_1f:=b\cdot vM^{1/2}.
\end{equation*}
Then $f$ admits the macro--micro decomposition
\begin{align}\label{decom}
f=\mathbf Pf+\{\mathbf I-\mathbf P\}f.
\end{align}
For a function $z=z(v)$, we define the velocity dissipation norm by
\begin{equation}\label{L12}
\|z\|_{L^2_{v,\nu}}
:=
\Big(\|\nabla_vz\|_{L^2_v}^2
+\|(1+|v|)z\|_{L^2_v}^2\Big)^{1/2}.
\end{equation}
Note that the Fokker--Planck operator satisfies the following
coercivity estimate (cf. e.g.,
\cite{carrillo4}):
\begin{align}
-\langle\mathcal Lf,f\rangle_{L^2_v}
&=
|b|^2-
\bigl\langle
\mathcal L\{\mathbf I-\mathbf P\}f,
\{\mathbf I-\mathbf P\}f
\bigr\rangle_{L^2_v}\ge
|b|^2+\lambda_0
\bigl\|
\{\mathbf I-\mathbf P\}f\bigr\|_{L^2_{v,\nu}}^2,
\label{e3}
\end{align}
where $\lambda_0>0$ depends only on $d$.
As $P'(1)>0$, we normalize $P'(1)=1$, and the general case can be addressed similarly. Noting that
\[
\frac12u\cdot v\,\mathbf Pf-u\cdot\nabla_v\mathbf Pf
=a\,u\cdot v\,M^{1/2}
+(v\otimes v-\mathrm{Id}):(u\otimes b)M^{1/2},
\]
we rewrite the Cauchy problem \eqref{m1}
for the NS--VFP system as
\begin{equation}\label{m1n}
\left\{
\begin{aligned}
&\partial_t\varrho
+u\cdot\nabla_x\varrho
+\dive_xu
=-\varrho\dive_xu,\\
&\partial_tu+u\cdot\nabla_xu
+\nabla_x\varrho-\mu\Delta_xu
-(\mu+\lambda)\nabla_x\dive_xu
+u-b
=-au+g(\varrho,u,f),\\
&\partial_tf+v\cdot\nabla_xf
-\mathcal Lf-u\cdot vM^{1/2}\\
&\quad\quad \quad\quad  =a\,u\cdot v\,M^{1/2}
+(v\otimes v-\mathrm{Id}):(u\otimes b)M^{1/2}+h(\varrho,u,f),\\
&\big(\varrho(0,x),u(0,x),f(0,x,v)\big)
=\big(\varrho_0(x),u_0(x),f_0(x,v)\big),
\end{aligned}
\right.
\end{equation}
with
\begin{equation}\label{e4}
\left\{
\begin{aligned}
g(\varrho,u,f):={}&-\Big(
\frac{P'(1+\varrho)}{1+\varrho}-1
\Big)\nabla_x\varrho
-\frac{\varrho}{1+\varrho}
\Big(\mu\Delta_xu
+(\mu+\lambda)\nabla_x\dive_xu\Big)\\
&+\frac{\varrho}{1+\varrho}
\bigl(u-b+au\bigr),\\
h(\varrho,u,f):={}&\frac12u\cdot v
\{\mathbf I-\mathbf P\}f
-u\cdot\nabla_v
\{\mathbf I-\mathbf P\}f.
\end{aligned}
\right.
\end{equation}

Formally, as $\mu=\lambda=\varepsilon\to0$, the solution
$(\varrho^\varepsilon,u^\varepsilon,f^\varepsilon)$ to the Cauchy
problem \eqref{m1n} for the NS--VFP system is expected to converge to a limit
$(\varrho,u,f)$ solving the following Cauchy problem for the
Euler--VFP system:
\begin{equation}\label{m1n0}
\left\{
\begin{aligned}
&\partial_t\varrho
 +u\cdot\nabla_x\varrho+\dive_xu
 =-\varrho\dive_xu,\\
&\partial_tu+u\cdot\nabla_xu
 +\nabla_x\varrho+u-b
 =-au+g^0(\varrho,u,f),\\
&\partial_tf+v\cdot\nabla_xf
 -\mathcal Lf-u\cdot vM^{1/2}
 =\frac12u\cdot v\mathbf Pf
 -u\cdot\nabla_v\mathbf Pf+h(\varrho,u,f),\\
&\big(\varrho(0,x),u(0,x),f(0,x,v)\big)
 =\big(\varrho_0(x),u_0(x),f_0(x,v)\big),
\end{aligned}
\right.
\end{equation}
with
\[
g^0(\varrho,u,f):=-\Big(
\frac{P'(1+\varrho)}{1+\varrho}-1
\Big)\nabla_x\varrho
+\frac{\varrho}{1+\varrho}
\bigl(u-b+au\bigr).
\]

First, we prove global well-posedness near equilibrium for the Cauchy problem \eqref{m1n} in the critical space. We take the fluid data $\varrho_0\in\dot B^{d/2-1}_{2,1}\cap\dot B^{d/2}_{2,1}$ and $u_0\in\dot B^{d/2-1}_{2,1}$ according to the scaling of the compressible Navier--Stokes equations \cite{danchin1}. The drag coupling leads us to measure $f_0$ at the same spatial regularity as $u_0$, while we impose only $L^2(\mathbb R^d_v)$ regularity in the particle velocity variable, i.e., $f_0\in\dot{\mathcal B}^{d/2-1}_{2,1}=\dot B^{d/2-1}_{2,1}L^2_v$. For the definitions and properties of the Besov spaces and the low ($^\ell$)--high ($^h$) frequency decomposition, please refer to Appendix~A.

We define the initial energy norm
\begin{equation}\label{a2}
\begin{aligned}
\|(\varrho_0,u_0,f_0)\|_{\mathcal E_0}:={}&
 \|\varrho_0\|_{\dot B^{d/2-1}_{2,1}
 \cap\dot B^{d/2}_{2,1}}
 +\|u_0\|_{\dot B^{d/2-1}_{2,1}}
 +\|f_0\|_{\dot{\mathcal B}^{d/2-1}_{2,1}},
\end{aligned}
\end{equation}
and the solution norm
\begin{align}
\|(\varrho,u,f)\|_{\mathcal E_t}:={}&
 \|\varrho\|_{L^\infty_t(\dot B^{d/2-1}_{2,1}\cap\dot B^{d/2}_{2,1})}
 +\|u\|_{L^\infty_t(\dot B^{d/2-1}_{2,1})}
 +\|f\|_{L^\infty_t(\dot{\mathcal B}^{d/2-1}_{2,1})}\nonumber\\
&+\|\varrho^\ell\|_{L^1_t(\dot B^{d/2+1}_{2,1})}
 +\|\varrho^h\|_{L^1_t(\dot B^{d/2}_{2,1})}
 +\|u\|_{L^1_t(\dot B^{d/2+1}_{2,1})}\nonumber\\
 &+\|f^\ell\|_{L^1_t(\dot{\mathcal B}^{d/2+1}_{2,1})}+\|f^h\|_{L^1_t(\dot{\mathcal B}^{d/2-1}_{2,1})}
 \nonumber\\
&
 +\|u-b\|_{L^2_t(\dot B^{d/2-1}_{2,1})}
 +\bigl\|\{\mathbf I-\mathbf P\}f\bigr\|_{L^2_t(\dot{\mathcal B}^{d/2-1}_{2,1,\nu})}.
\label{Xt}
\end{align}

\begin{theorem}\label{t1}
For any $d\ge2$, assume that the initial datum
$(\varrho_0,u_0,f_0)$ satisfies $\rho_0=1+\varrho_0>0$,
$F_0=M+M^{1/2}f_0\ge0$, $ \varrho_0\in\dot B^{d/2-1}_{2,1}
  \cap\dot B^{d/2}_{2,1},\quad  u_0\in\dot B^{d/2-1}_{2,1}$ and $f_0\in\dot{\mathcal B}^{d/2-1}_{2,1}$. There exists a constant $\delta_0>0$, depending only on $d$ and the
fixed coefficients of the system, such that if
\begin{equation*}
 \|(\varrho_0,u_0,f_0)\|_{\mathcal E_0}\le\delta_0,
\end{equation*}
then the Cauchy problem \eqref{m1n} for the NS--VFP system admits a
unique global strong solution $(\varrho,u,f)$ satisfying
\begin{equation}\label{e6}
 \sup_{t\geq0}\|(\varrho,u,f)\|_{\mathcal E_t}
 \le C_0\|(\varrho_0,u_0,f_0)\|_{\mathcal E_0},
\end{equation}
where $C_0>0$ is a constant. Moreover, 
$\rho=1+\varrho>0$ and $F=M+M^{1/2}f\ge0$ hold, and $\|(\varrho,u,f)\|_{\mathcal E_t}$ is continuous with respect to $t\in\mathbb R_+$.
\end{theorem}

\begin{remark}\label{r1}\normalfont
The estimate \eqref{e6} in Theorem
\ref{t1} reflects the frequency-dependent dissipation of the
system. Indeed, at low frequencies, $\varrho$, $u$ and $f$ satisfy heat-type
estimates. At high frequencies, $\varrho$ and $f$ exhibit
damping due to a spectral gap, while $u$ retains heat-type dissipation
because of the viscosity. Moreover, we obtain the additional
lower-order regularities
\[
u-b\in L^2(\mathbb R_+;\dot B^{d/2-1}_{2,1})
\quad\text{and}\quad
\{\mathbf I-\mathbf P\}f
\in L^2(\mathbb R_+;\dot{\mathcal B}^{d/2-1}_{2,1,\nu}),
\]
which quantify the dissipative effects of the drag force and the
Fokker--Planck operator, respectively. These regularities play a crucial
role in analyzing the nonlinear terms.
\end{remark}

Next, we study the Cauchy problem \eqref{m1n} for the
NS--VFP system with
$\mu=\lambda=\varepsilon$, focusing on higher-order estimates
uniform in $0<\varepsilon\le1$ and the vanishing-viscosity limit
as $\varepsilon\to0$. The upper regularity index $d/2+1$
is critical for the compressible Euler equations in the sense that
$\dot B^{d/2+1}_{2,1}$ is continuously embedded into the
Lipschitz space \cite{crin1,xu1}. To this end, we define 
\begin{equation}\label{Z0}
\begin{aligned}
\|(\varrho_0,u_0,f_0)\|_{\mathcal Z_0}:={}&
\|(\varrho_0,u_0)\|_{\dot B^{d/2-1}_{2,1}
\cap\dot B^{d/2+1}_{2,1}}
+\|f_0\|_{\dot{\mathcal B}^{d/2-1}_{2,1}
\cap\dot{\mathcal B}^{d/2+1}_{2,1}},
\end{aligned}
\end{equation}
and
\begin{equation}\label{Zt}
\begin{aligned}
\|(\varrho,u,f)\|_{\mathcal Z_t}:={}&
\|(\varrho,u)\|_{L^\infty_t(
\dot B^{d/2-1}_{2,1}
\cap\dot B^{d/2+1}_{2,1})}
+\|f\|_{L^\infty_t(
\dot{\mathcal B}^{d/2-1}_{2,1}
\cap\dot{\mathcal B}^{d/2+1}_{2,1})}\\
&+\|(\varrho,u)\|_{L^1_t(\dot B^{d/2+1}_{2,1})}
+\|f\|_{L^1_t(\dot{\mathcal B}^{d/2+1}_{2,1})}\\
&+\|u-b\|_{L^2_t(
\dot B^{d/2-1}_{2,1}\cap\dot B^{d/2+1}_{2,1})}
+\bigl\|\{\mathbf I-\mathbf P\}f\bigr\|_
{L^2_t(\dot{\mathcal B}^{d/2-1}_{2,1,\nu}
\cap\dot{\mathcal B}^{d/2+1}_{2,1,\nu})}.
\end{aligned}
\end{equation}

\begin{theorem}\label{t2}
Let $d\ge2$ and $0<\varepsilon\le1$. Suppose that
$(\varrho_0^\varepsilon,u_0^\varepsilon,f_0^\varepsilon)$ satisfies
$\rho_0^\varepsilon=1+\varrho_0^\varepsilon>0$,
$F_0^\varepsilon=M+M^{1/2}f_0^\varepsilon\ge0$,
\[
 \varrho_0^\varepsilon,u_0^\varepsilon
 \in\dot B^{d/2-1}_{2,1}\cap
       \dot B^{d/2+1}_{2,1}\quad\text{and}\quad 
 f_0^\varepsilon
 \in\dot{\mathcal B}^{d/2-1}_{2,1}\cap
       \dot{\mathcal B}^{d/2+1}_{2,1}.
\]
There exists a constant $\delta_1>0$, independent of $\varepsilon$,
such that if 
\begin{equation}\label{e7}
 \|(\varrho_0^\varepsilon,u_0^\varepsilon,f_0^\varepsilon)
 \|_{\mathcal Z_0}\le\delta_1,
\end{equation}
then the Cauchy problem \eqref{m1n} for the NS--VFP system
admits a unique global strong
solution $(\varrho^\varepsilon,u^\varepsilon,f^\varepsilon)$ satisfying
\begin{equation}\label{e8}
\sup_{t\geq 0}\big( \|(\varrho^\varepsilon,u^\varepsilon,f^\varepsilon)
 \|_{\mathcal Z_t}+\varepsilon\|u^\varepsilon\|_{L^1_t(
\dot B^{d/2+1}_{2,1}\cap\dot B^{d/2+2}_{2,1})} \big)
 \le C_1
 \|(\varrho_0^\varepsilon,u_0^\varepsilon,f_0^\varepsilon)
 \|_{\mathcal Z_0},
\end{equation}
where the constant $C_1>0$ is independent of $\varepsilon$. Moreover,
$\rho^\varepsilon=1+\varrho^\varepsilon>0$ and
$F^\varepsilon=M+M^{1/2}f^\varepsilon\ge0$, while
$\|(\varrho^\varepsilon,u^\varepsilon,f^\varepsilon)\|_{\mathcal Z_t}$
is continuous with respect to $t\in\mathbb R_+$.
\end{theorem}

The uniform estimates in Theorem \ref{t2} establish global
well-posedness of the limiting Euler--VFP system in the critical space.

\begin{theorem}\label{t3}
Let $d\ge2$. Suppose that $(\varrho_0,u_0,f_0)$ satisfies
$\rho_0=1+\varrho_0>0$, $F_0=M+M^{1/2}f_0\ge0$, $\varrho_0,u_0
 \in\dot B^{d/2-1}_{2,1}\cap
       \dot B^{d/2+1}_{2,1}$ and $f_0
 \in\dot{\mathcal B}^{d/2-1}_{2,1}\cap
       \dot{\mathcal B}^{d/2+1}_{2,1}$. There exists a constant $\delta_2>0$ such that if
\begin{equation}\label{e9}
 \|(\varrho_0,u_0,f_0)\|_{\mathcal Z_0}\le\delta_2,
\end{equation}
then the Cauchy problem \eqref{m1n0} for the Euler--VFP system admits a
unique global strong solution $(\varrho,u,f)$ satisfying
\begin{equation*}
 \sup_{t\geq 0}\|(\varrho,u,f)\|_{\mathcal Z_t}
 \le C_2\|(\varrho_0,u_0,f_0)\|_{\mathcal Z_0},
\end{equation*}
where $C_2>0$ is a constant. The solution also satisfies $\rho=1+\varrho>0$ and $F=M+M^{1/2}f\ge0$, while $\|(\varrho,u,f)\|_{\mathcal Z_t}$ is continuous with respect to $t\in\mathbb R_+$.
\end{theorem}

\begin{remark}\normalfont
In Theorem \ref{t3}, 
by capturing the effects of the Fokker--Planck operator and the drag force, we obtain $u\in L^1(\mathbb R_+;\dot B^{d/2+1}_{2,1})\hookrightarrow L^1(\mathbb R_+;{\rm Lip})$. At low frequencies, this regularity is weaker than the regularity $L^1(\mathbb R_+;\dot B^{d/2}_{2,1})$ obtained for
standard partially dissipative hyperbolic systems such as the Euler
equations with damping \cite{crin1}.
\end{remark}

Then, we establish uniform error estimates between the solutions
to the Cauchy problems \eqref{m1n} and \eqref{m1n0} and justify the global-in-time strong inviscid limit.

\begin{theorem}\label{t4}
Let $(\varrho^\varepsilon,u^\varepsilon,f^\varepsilon)$ be the solution supplied by Theorem \ref{t2}, and let $(\varrho,u,f)$ be the solution supplied by Theorem \ref{t3}. Assume further that
\begin{equation}\label{e11}
 \|(\varrho_0^\varepsilon-\varrho_0,
 u_0^\varepsilon-u_0)\|_{\dot B^{d/2-1}_{2,1}}
 +\|f_0^\varepsilon-f_0\|_{\dot{\mathcal B}^{d/2-1}_{2,1}}
 \le \varepsilon.
\end{equation}
Then the solutions satisfy
\begin{equation}\label{e12}
\sup_{t\geq0} \big( \|(\varrho^\varepsilon-\varrho,
 u^\varepsilon-u)(t)\|_{\dot B^{d/2-1}_{2,1}}
 +\|(f^\varepsilon-f)(t)\|_{\dot{\mathcal B}^{d/2-1}_{2,1}}\big)
 \le C\varepsilon.
\end{equation}
In addition, the following estimate holds:
\begin{align*}
\sup_{t\geq0} \big(\|(\varrho^\varepsilon-\varrho,
u^\varepsilon-u)(t)\|_{L^\infty_x}
+\|(f^\varepsilon-f)(t)\|_{L^\infty_x(L^2_v)}\big)&\leq C\sqrt{\varepsilon}.
\end{align*}
Here, the constant $C$ is independent of $\varepsilon\in(0,1]$.
\end{theorem}

Finally, we establish time-decay estimates for the viscous system
\eqref{m1n} and the inviscid system \eqref{m1n0}, uniformly for
$\varepsilon\in[0,1]$.

\begin{theorem}\label{t5}
For each $\varepsilon\in(0,1]$, let
$(\varrho^\varepsilon,u^\varepsilon,f^\varepsilon)$ be the global
solution to the Cauchy problem for the NS--VFP system given by Theorem
\ref{t2}, and let $(\varrho,u,f)$ be the global solution to the Cauchy
problem for the Euler--VFP system given by Theorem \ref{t3}, subject to
the same data $(\varrho_0,u_0, f_0)$ satisfying \eqref{e9}.
Let $-d/2\leq\sigma_1<d/2-1$ and suppose further that $(\varrho_0,u_0)\in 
 \dot B^{\sigma_1}_{2,\infty}$ and $ f_0\in \dot{\mathcal B}^{\sigma_1}_{2,\infty}$.
Then, for all
$t\geq0$, $(\varrho^\varepsilon,u^\varepsilon,f^\varepsilon)$ and $(\varrho,u,f)$
satisfy the following decay estimates for $\sigma_1<\sigma\leq d/2+1$:
\begin{align}
& \|(\varrho^\varepsilon,u^\varepsilon)(t)\|_{\dot B^\sigma_{2,1}}
 +\|f^\varepsilon(t)\|_{\dot{\mathcal B}^\sigma_{2,1}}
 \le C(1+t)^{-\frac{\sigma-\sigma_1}{2}},\label{e13}\\
  &\|(\varrho,u)(t)\|_{\dot B^\sigma_{2,1}}
 +\|f(t)\|_{\dot{\mathcal B}^\sigma_{2,1}}
 \le C(1+t)^{-\frac{\sigma-\sigma_1}{2}},\label{e14}
\end{align}
and for $\sigma_1<\sigma\leq d/2$:
\begin{align}
&\|(u^\varepsilon-b^\varepsilon)(t)\|_{\dot B^\sigma_{2,1}}
 +\|\{\mathbf I-\mathbf P\}f^\varepsilon(t)\|_
 {\dot{\mathcal B}^\sigma_{2,1}}
 \le C(1+t)^{-\frac{\sigma-\sigma_1}{2}-\frac{1}{2}},\label{e16}\\
& \|(u-b)(t)\|_{\dot B^\sigma_{2,1}}
 +\|\{\mathbf I-\mathbf P\}f(t)\|_
 {\dot{\mathcal B}^\sigma_{2,1}}
 \le C(1+t)^{-\frac{\sigma-\sigma_1}{2}-\frac{1}{2}},\nonumber
\end{align}
where the constant $C>0$ is independent of $\varepsilon$ and $t$.
\end{theorem}

\begin{remark}\normalfont
The decay estimates \eqref{e13} and \eqref{e14} hold up to
the highest-order critical index \(\sigma=d/2+1\). Moreover, the interactions of the drag relaxation
and the Fokker--Planck dissipation lead to an additional decay rate
\((1+t)^{-\frac{1}{2}}\) for the relative velocity
\(u^\varepsilon-b^\varepsilon\) and the microscopic component
\(\{\mathbf I-\mathbf P\}f^\varepsilon\), compared with the solution at the same spatial regularity for
\(\sigma_1<\sigma\leq d/2\).
\end{remark}

\begin{remark}\normalfont
Let $1\leq q<2$ and set
$\sigma_1:=-d(\frac1q-\frac12)$. Then the embedding $L^q(\mathbb R^d_x)\hookrightarrow\dot B^{\sigma_1}_{2,\infty}$ holds. Hence, if the initial data are uniformly bounded in the corresponding
 spatial $L^q$ spaces, using $\dot{B}^{0}_{2,1}\hookrightarrow L^2(\mathbb{R}^d)$ and $\dot{B}^{d/2}_{2,1}\hookrightarrow L^\infty(\mathbb{R}^d)$, we obtain
\begin{align*}
&\|(\varrho^\varepsilon,u^\varepsilon)(t)\|_{L^2_x}
 +\|f^\varepsilon(t)\|_{L^2_x(L^2_v)}
 \leq C(1+t)^{-\frac{d}{2}\left(\frac{1}{q}-\frac{1}{2}\right)},\\
&\|(u^\varepsilon-b^\varepsilon)(t)\|_{L^2_x}
 +\|\{\mathbf I-\mathbf P\}f^\varepsilon(t)\|_{L^2_x(L^2_v)}
 \leq C(1+t)^{-\frac{d}{2}\left(\frac{1}{q}-\frac{1}{2}\right)-\frac{1}{2}},\\
&\|(\varrho^\varepsilon,u^\varepsilon)(t)\|_{L^\infty_x}
 +\|f^\varepsilon(t)\|_{L^\infty_x(L^2_v)}
 \leq C(1+t)^{-\frac{d}{2q}},\\
&\|(u^\varepsilon-b^\varepsilon)(t)\|_{L^\infty_x}
 +\|\{\mathbf I-\mathbf P\}f^\varepsilon(t)\|_{L^\infty_x(L^2_v)}
 \leq C(1+t)^{-\frac{d}{2q}-\frac{1}{2}}.
\end{align*}
For the critical Lipschitz norm, we also have
\begin{align*}
&\|(\nabla_x\varrho^\varepsilon,
      \nabla_xu^\varepsilon)(t)\|_{L^\infty_x}
 +\|\nabla_xf^\varepsilon(t)\|_{L^\infty_x(L^2_v)}
 \leq C(1+t)^{-\frac{d}{2q}-\frac{1}{2}}.
\end{align*}
The above decay rates are optimal in the sense that they coincide with those of the heat equation at low frequencies. Note that the general $B^{\sigma_1}_{2,\infty}$ assumption is in fact equivalent to
the heat-semigroup decay 
\cite{brandolese1}.
\end{remark}


\vspace{2mm}
We explain the main difficulties and ideas in the proofs of
Theorems \ref{t1}--\ref{t5}.
By developing frequency-localized estimates and capturing the effects
of the drag force and the Fokker--Planck operator, we construct
low- and high-frequency Lyapunov inequalities for the linearized problem
(cf.\ Lemmas \ref{l1},
\ref{l3}). 

To analyze the nonlinear problem \eqref{m1n}, we need to overcome
difficulties caused by the drag force. Indeed, the classical critical regularity theory for
the compressible Navier--Stokes equations requires the nonlinear source
terms to be controlled in
$L^1(\mathbb R_+;\dot B^{d/2-1}_{2,1})$, cf.\ \cite{danchin1}.
However, several lower-order quadratic terms generated by
the drag force in \eqref{e4}
contain no spatial derivative and cannot be controlled directly by the
usual critical regularity norms.

The first class consists of 
$\frac{\varrho}{1+\varrho}(u-b)$ and
$u\cdot\nabla_v\{\mathbf I-\mathbf P\}f$. To handle these, we establish the additional $L^2(\mathbb R_+;\dot B^{d/2-1}_{2,1})$ estimate of the relative
velocity $u-b$ and the $L^2(\mathbb R_+;\dot{\mathcal{B}}^{d/2-1}_{2,1,\nu})$ estimate of the microscopic part
$\{\mathbf I-\mathbf P\}f$. Together with the critical regularity estimates for $(\varrho,u)$,
these bounds yield the required
$L^1(\mathbb R_+;\dot B^{d/2-1}_{2,1})$ estimates.

The second difficulty arises from the term $au$ in \eqref{m1n}.
In the usual macro--micro formulation, the drag dissipation is
represented by $u-b$ (cf. \cite{carrillo4,duan3}). However, $u$ is the velocity, whereas $b$ is the  momentum of the particles. Accordingly, $\int_{\mathbb{R}^d_v}(u-v)F\,dv=(u-b)+au$ produces the
additional term $au$. To overcome this mismatch, we make use of the macroscopic velocity variable
\begin{equation}\label{e18}
 V:=\frac{b}{1+a}.
\end{equation}
The equation for $V$, instead of that for $b$, removes the term $au$
from the particle momentum equation, cf.\ \eqref{e27}.

The third difficulty is the quadratic microscopic term $
(v\otimes v-\mathrm{Id}):(u\otimes b)M^{1/2}$ in $\eqref{m1n}_3$. To eliminate it, we introduce the
effective microscopic variable
\begin{equation}\label{e19}
 G:=\{\mathbf I-\mathbf P\}f
 -\frac12(v\otimes v-\mathrm{Id}):(b\otimes b)M^{1/2}.
\end{equation}
Indeed, the definition of \(G\) implies $\mathcal L\{\mathbf I-\mathbf P\}f
+(v\otimes v-\mathrm{Id}):(u\otimes b)M^{1/2}
=
\mathcal LG
+(v\otimes v-\mathrm{Id}):((u-b)\otimes b)M^{1/2}$. Thus, the original quadratic term is reduced to $\mathcal LG$ and a term containing the dissipative quantity $u-b$. By differentiating the quadratic correction in $G$ and substituting the
equation for $\partial_t b$, the extra quadratic terms contain spatial derivatives, while the terms without spatial derivatives are at least cubic, cf. \eqref{e29}. This correction is similar to a normal-form transformation.

However, a new obstruction arises: since $a$
is only controlled in $\dot B^{d/2-1}_{2,1}$, we do not have an
$L^\infty_x$ bound to define $V=b/(1+a)$ directly. We therefore use the low-frequency truncated variables defined in
\eqref{e20}.
By Bernstein's inequality, $a^\ell$ is controlled in $L^\infty_x$.
Thus, we use the reformulated variables at low frequencies and recover
the original variables afterwards, while the high-frequency estimates are
carried out directly in the original variables. Based on these two formulations, we construct suitable low- and
high-frequency Lyapunov functionals and obtain the global a priori
estimate \eqref{e26}. Therefore, we extend the local
solution to a global one and complete the proof of
Theorem~\ref{t1}.

For Theorems \ref{t2} and \ref{t3}, the key point is to
recover the dissipation of $u^\varepsilon$ independent of $\mu=\lambda=\varepsilon$. By the drag force relaxation and the higher-order
moment equations of the Fokker--Planck equation, we obtain 
\[
\min\{2^{2j},1\}
\|\dot\Delta_j u^\varepsilon\|_{L^2}
\lesssim
\begin{cases}
\|\dot\Delta_j(u^\varepsilon-V^\varepsilon)\|_{L^2}
 +2^{2j}\|\dot\Delta_jV^\varepsilon\|_{L^2},
 & j\leq0,\\[2mm]
\|\dot\Delta_j(u^\varepsilon-b^\varepsilon)\|_{L^2}
 +2^{2j}\|\dot\Delta_jb^\varepsilon\|_{L^2},
 & j\geq-1,
\end{cases}
\]
which is independent of $\varepsilon$. We then use the hyperbolic symmetry at the Euler-critical regularity to control the higher derivative. The resulting bounds are global in time and yield global well-posedness for the Cauchy problem for the Euler--VFP system.

For Theorem \ref{t4}, we subtract the viscous and inviscid systems and construct low- and high-frequency error functionals in $\dot B^{d/2-1}_{2,1}$, together with
$\dot{\mathcal B}^{d/2-1}_{2,1}$ for the kinetic error. We retain the
fluid--particle cancellations in these error functionals. The only
additional source term comes from viscosity, and the uniform higher-order
estimate gives
\[
 \left\|\frac{\varepsilon}{1+\varrho^\varepsilon}
 \big(\Delta_xu^\varepsilon
 +2\nabla_x\dive_xu^\varepsilon\big)\right\|_
 {L^1_t(\dot B^{d/2-1}_{2,1})}
 \lesssim\varepsilon
 \|u^\varepsilon\|_{L^1_t(\dot B^{d/2+1}_{2,1})}
 \lesssim\varepsilon.
\]
The nonlinear error terms are absorbed by the small uniform bound, and
this yields the global $\mathcal O(\varepsilon)$ convergence rate.

Finally, to derive the decay estimates in Theorem \ref{t5}, we use a time-weighted energy argument. Negative Besov norms of the form $\dot B^{\sigma_1}_{2,\infty}$ have been used in the decay analysis
of the Boltzmann equation and partially dissipative hyperbolic or
hyperbolic--parabolic systems
\cite{brandolese1,danchin3,LiuShouXu2025,so1,xu2}. For the compressible
Navier--Stokes equations, Guo and Wang \cite{guo1} and Xin and Xu
\cite{xin1} combine a differential Lyapunov inequality with a
propagated negative norm. Here the low-frequency analysis uses the
reformulated variables, while the high-frequency analysis is carried out
in the original variables, so the two estimates cannot be combined into
a single differential Lyapunov inequality. We therefore first propagate
the lower-order Besov norm and then apply time weights to the
frequency-localized hypocoercive inequalities. At low
frequencies, the negative norm and the heat-type dissipation give the
algebraic decay by interpolation, whereas the uniform high-frequency
damping gives the corresponding weighted bounds. Finally, the equations
for $u^\varepsilon-b^\varepsilon$ and
$\{\mathbf I-\mathbf P\}f^\varepsilon$, together with their zeroth-order
damping and the spatial derivative in the leading source terms, yield
the additional decay $(1+t)^{-\frac{1}{2}}$ in  \eqref{e16}.

\section{Global existence}\label{s2}

\subsection{A priori estimate}\label{s3}

In this section, we prove Theorem \ref{t1}. For local
well-posedness, refer to Appendix B.

\begin{proposition}\label{p1}
Let $T>0$ and let $(\varrho,u,f)$ be a smooth solution to the Cauchy
problem \eqref{m1n} for the NS--VFP system on $[0,T]$. There exist
constants $\delta_0^*>0$ and
$C_0>0$, independent of $T$, such that if
\begin{align}\label{e25}
\|(\varrho,u,f)\|_{\mathcal E_T}\le \delta_0^*,
\end{align}
then the following estimate holds:
\begin{equation}\label{e26}
\|(\varrho,u,f)\|_{\mathcal E_T}
\le C_0\Big(
\|(\varrho_0,u_0,f_0)\|_{\mathcal E_0}
+\|(\varrho_0,u_0,f_0)\|_{\mathcal E_0}^2\Big),
\end{equation}
where the $\mathcal E_0$- and $\mathcal E_T$-norms are defined by
\eqref{a2} and \eqref{Xt}, respectively.
\end{proposition}

Proposition \ref{p1} follows from the low- and high-frequency
estimates in Lemmas \ref{l2} and \ref{l4}.

\subsection{Low-frequency analysis}\label{s4}

The Shizuta--Kawashima condition and Villani's hypocoercivity
theory provide general frameworks for analyzing partially
dissipative systems \cite{crin1,villani1,xu1}. However, the
linearized NS--VFP and Euler--VFP systems do not fit directly
into these frameworks because of the drag force coupling.
To address this issue, we adapt the frequency-localized estimates
developed in our previous work on the Navier--Stokes--Euler
system \cite{lishou1} to analyze the linear dissipation structure.

The nonlinear terms create an additional difficulty. In view of the macro--micro decomposition \eqref{decom}, it suffices to
estimate 
$a$, $b$ and $\{\mathbf I-\mathbf P\}f$ due to the fact
\begin{equation}\label{E35}
\|f\|_{L^2_v}\sim |a|+|b|
+\|\{\mathbf I-\mathbf P\}f\|_{L^2_v}.
\end{equation}
As explained in Section 1, it is difficult to estimate $(a,b, \{\mathbf I-\mathbf P\}f)$ directly because of the lower-order
terms $au$ and $u\otimes b$. We first take the zeroth and first velocity
moments of the kinetic equation and obtain
\begin{equation}\label{e22}
\left\{
\begin{aligned}
&\partial_ta+\dive_x b=0,\\
&\partial_tb+\nabla_xa+b-u
+\dive_x\Theta(\{\mathbf I-\mathbf P\}f)=au,
\end{aligned}
\right.
\end{equation}
where $\Theta(z)$ is defined by
\begin{align*}
    \Theta(z):=\int_{\mathbb R^d_v}(v\otimes v-\mathrm{Id})
M^{1/2}z\,dv.
\end{align*}
Moreover, the microscopic projection gives
\begin{equation}\label{e23}
\begin{aligned}
&\partial_t\{\mathbf I-\mathbf P\}f
+\{\mathbf I-\mathbf P\}
\big(v\cdot\nabla_x\{\mathbf I-\mathbf P\}f\big)
-\mathcal L\{\mathbf I-\mathbf P\}f
+(v\otimes v-\mathrm{Id}):\nabla_xb\,M^{1/2}\\
&\qquad=h(\varrho,u,f)
+(v\otimes v-\mathrm{Id}):(u\otimes b)M^{1/2}.
\end{aligned}
\end{equation}

To cancel $au$ and $(v\otimes v-\mathrm{Id}):(u\otimes b)M^{1/2}$ in \eqref{e22}--\eqref{e23}, we use the macroscopic velocity $V$ and the microscopic variable $G$ defined by \eqref{e18} and \eqref{e19}, respectively. Substituting $b=(1+a)V$ into \eqref{e22} and using $\partial_t b=(1+a)\partial_t V-V\dive_x\,b$, we obtain
\begin{equation}\label{e24}
\left\{
\begin{aligned}
&\partial_ta+\dive_x\big((1+a)V\big)=0,\\
&\partial_tV+\nabla_xa-(u-V)
+\dive_x\Theta(\{\mathbf I-\mathbf P\}f)\\
&\qquad=\frac{a}{1+a}
\Big(\nabla_xa+\dive_x\Theta(\{\mathbf I-\mathbf P\}f)\Big)
+\frac{V}{1+a}\dive_xb.
\end{aligned}
\right.
\end{equation}
Thus, the lower-order term $au$ is replaced by terms with spatial
derivatives, while the linear relaxation term takes the form $u-V$.

By $\Theta(G)=\Theta(\{\mathbf I-\mathbf P\}f)-b\otimes b$, $\eqref{e24}_2$ can be rewritten as
\begin{align}
&\partial_tV+\nabla_xa-(u-V)+\dive_x\Theta(G)
\nonumber\\
&\qquad=\frac{a}{1+a}
\Big(\nabla_xa+\dive_x\Theta(\{\mathbf I-\mathbf P\}f)\Big)
+\frac{V}{1+a}\dive_xb-\dive_x(b\otimes b).
\label{e27}
\end{align}
From $\eqref{e22}_2$, we have
\begin{align}
\partial_t(b\otimes b)+2b\otimes b
={}&(1+a)(u\otimes b+b\otimes u)-\Big(\nabla_xa
+\dive_x\Theta(\{\mathbf I-\mathbf P\}f)\Big)\otimes b
\nonumber\\
&-b\otimes\Big(\nabla_xa
+\dive_x\Theta(\{\mathbf I-\mathbf P\}f)\Big).
\label{e28}
\end{align}
The key observation is that by \eqref{e23}, \eqref{e28},
$\nabla_xb=\nabla_xV+\nabla_x(aV)$ and 
\[
 \mathcal L\left(
 \frac12(v\otimes v-\mathrm{Id}):(b\otimes b)M^{1/2}
 \right)
 =-(v\otimes v-\mathrm{Id}):(b\otimes b)M^{1/2},
\]
the equation of $G$ reads
\begin{align}
&\partial_tG
+\{\mathbf I-\mathbf P\}(v\cdot\nabla_xG)
-\mathcal L G
+(v\otimes v-\mathrm{Id}):\nabla_xV\,M^{1/2}
\nonumber\\
&\quad=h(\varrho,u,f)
-a(v\otimes v-\mathrm{Id}):(u\otimes b)M^{1/2}
\nonumber\\
&\qquad+(v\otimes v-\mathrm{Id}):
\left(
\Big(\nabla_xa
+\dive_x\Theta(\{\mathbf I-\mathbf P\}f)\Big)\otimes b
-\nabla_x(aV)
\right) M^{1/2}
\nonumber\\
&\qquad-\{\mathbf I-\mathbf P\}
\left(
v\cdot\nabla_x\left(
\frac12(v\otimes v-\mathrm{Id}):(b\otimes b)M^{1/2}
\right)\right).
\label{e29}
\end{align}
Therefore, all nonlinear terms on the right-hand sides of \eqref{e27} and \eqref{e29} can be analyzed by the critical norms of the solutions without any loss of derivatives. 

At critical regularity, the control of \(a\) in
\(\dot B^{d/2-1}_{2,1}\) does not yield an \(L^\infty(\mathbb R^d_x)\) bound,
so the lower bound for \(1+a\) required to define \(V=b/(1+a)\) is
not available. To resolve this issue, we define the cut-off variables
\begin{equation}\label{e20}
\widetilde V:=\frac{b^\ell}{1+a^\ell}\quad\text{and}\quad 
\widetilde G:=\{\mathbf I-\mathbf P\}f^\ell
-\frac12(v\otimes v-\mathrm{Id}):
(b^\ell\otimes b^\ell)^\ell M^{1/2}.
\end{equation}
By \eqref{e25}
and Bernstein's inequality, the fact that $\delta_0^*>0$ is sufficiently small implies $\|a^\ell\|_{L^\infty_T(L^\infty_x)}\lesssim \|a^\ell\|_{L^\infty_T
 (\dot B^{d/2}_{2,1})}\ll 1$ and
\begin{equation}\label{e30}
\frac{1}{2}\leq 1+a^\ell(t,x) \leq \frac{3}{2}\quad\text{for}\quad (t,x)\in [0,T]\times\mathbb{R}^d.
\end{equation}
The low-frequency energy estimates are carried out for the new variables $\varrho^\ell,u^\ell,a^\ell$, $\widetilde V$ and $\widetilde G$. The low-frequency truncation produces two product defects
\begin{equation}\label{e31}
 \mathcal C_\ell(a,u):=(au)^\ell-a^\ell u^\ell
 \quad\text{and}\quad
 \mathcal D_\ell(u,b):=(u\otimes b)^\ell
 -(u^\ell\otimes b^\ell)^\ell.
\end{equation}
Indeed, we take the low-frequency part of \eqref{e22} and have
\begin{equation}\label{e32}
\left\{
\begin{aligned}
&\partial_ta^\ell+\dive_xb^\ell=0,\\
&\partial_tb^\ell+\nabla_xa^\ell+b^\ell-u^\ell
+\dive_x\Theta(\{\mathbf I-\mathbf P\}f^\ell)
=a^\ell u^\ell+\mathcal C_\ell(a,u).
\end{aligned}
\right.
\end{equation}
Due to
$b^\ell=(1+a^\ell)\widetilde V$, we use \eqref{e32} to obtain
\begin{equation}\label{e33}
\left\{
\begin{aligned}
&\partial_ta^\ell+\dive_x\widetilde V
=-\dive_x(a^\ell\widetilde V),\\
&\partial_t\widetilde V+\nabla_xa^\ell-(u^\ell-\widetilde V)
+\dive_x\Theta(\widetilde G)\\
&\quad=\frac{a^\ell}{1+a^\ell}
\Big(\nabla_xa^\ell
+\dive_x\Theta(\{\mathbf I-\mathbf P\}f^\ell)\Big)
+\frac{\widetilde V}{1+a^\ell}\dive_xb^\ell
+\frac{\mathcal C_\ell(a,u)}{1+a^\ell}
-\dive_x(b^\ell\otimes b^\ell)^\ell.
\end{aligned}
\right.
\end{equation}
Then, we take the low-frequency part of the
fluid equations $\eqref{m1n}_1$--$\eqref{m1n}_2$ and use $b^\ell=(1+a^\ell)\widetilde V$ to obtain
\begin{equation}\label{e34}
\left\{
\begin{aligned}
&\partial_t\varrho^\ell+\dive_xu^\ell
 =-\dive_x(\varrho u)^\ell,\\
&\partial_tu^\ell+\nabla_x\varrho^\ell-\mu\Delta u^\ell
-(\mu+\lambda)\nabla\dive u^\ell+u^\ell-\widetilde V\\
&\quad=-(u\cdot\nabla_xu)^\ell+g(\varrho,u,f)^\ell
-a^\ell(u^\ell-\widetilde V)-\mathcal C_\ell(a,u).
\end{aligned}
\right.
\end{equation}
Here we used the identity
$-(au)^\ell+a^\ell\widetilde V
=-a^\ell(u^\ell-\widetilde V)-\mathcal C_\ell(a,u)$.

We next take the low-frequency part of \eqref{e23}. Differentiating
the quadratic correction $\widetilde G$ in \eqref{e20} and
using the second equation in \eqref{e32}, we have
\begin{equation}\label{e35}
\begin{aligned}
&\partial_t\widetilde G
+\{\mathbf I-\mathbf P\}(v\cdot\nabla_x\widetilde G)
-\mathcal L\widetilde G
+(v\otimes v-\mathrm{Id}):\nabla_x\widetilde V\,M^{1/2}\\
&\quad=h(\varrho,u,f)^\ell-(v\otimes v-\mathrm{Id}):
\nabla_x(a^\ell\widetilde V)M^{1/2}\\
&\qquad +(v\otimes v-\mathrm{Id}):
\Big\{\mathcal D_\ell(u,b)
-\big(a^\ell u^\ell\otimes b^\ell\big)^\ell\\
&\qquad
+\Big(\big(\nabla_xa^\ell
+\dive_x\Theta(\{\mathbf I-\mathbf P\}f^\ell)
-\mathcal C_\ell(a,u)\big)\otimes b^\ell\Big)^\ell
\Big\}M^{1/2}\\
&\qquad-\{\mathbf I-\mathbf P\}
\left(v\cdot\nabla_x\left(
\frac12(v\otimes v-\mathrm{Id}):
(b^\ell\otimes b^\ell)^\ell M^{1/2}\right)\right),
\end{aligned}
\end{equation}
with the initial datum
$\widetilde G(0,x,v)=\widetilde G_0(x,v):=\{\mathbf I-\mathbf P\}f_0^\ell
-\frac12(v\otimes v-\mathrm{Id}):(b_0^\ell\otimes b_0^\ell)^\ell M^{1/2}$.

By the low--high frequency decomposition, the two defects $\mathcal C_\ell(a,u)$ and $\mathcal D_\ell(u,b)$ defined by \eqref{e31} satisfy
\begin{equation}\label{e36}
\mathcal C_\ell(a,u)=-(a^\ell u^\ell)^h
+(a^hu+a^\ell u^h)^\ell\quad\text{and}\quad 
\mathcal D_\ell(u,b)=(u^h\otimes b+u^\ell\otimes b^h)^\ell.
\end{equation}
Hence, both defects are controlled by the corresponding high-frequency
regularity.

We first establish the following low-frequency Lyapunov inequalities.
\begin{lemma}\label{l1}
For every \(j\leq1\), there exist constants \(c,C>0\), independent of
\(j\), and a functional \(\mathcal L_{\ell,j}\) such that
\begin{equation}\label{ELsim}
\mathcal L_{\ell,j}\sim
\|\dot\Delta_j(\varrho^\ell,u^\ell,a^\ell,
\widetilde V)\|_{L^2_x}^2
+\|\dot\Delta_j\widetilde G\|_{L^2_{x,v}}^2,
\end{equation}
and
\begin{equation}\label{low11}
\frac{d}{dt}\mathcal L_{\ell,j}+c2^{2j}\mathcal L_{\ell,j}
+c\|\dot\Delta_j(u^\ell-\widetilde V)\|_{L^2_x}^2
+c\|\dot\Delta_j\widetilde G\|_{L^2_xL^2_{v,\nu}}^2
\lesssim\mathcal R_{\ell,j}\sqrt{\mathcal L_{\ell,j}},
\end{equation}
with
\begin{equation}\label{e37}
\begin{aligned}
\mathcal R_{\ell,j}:={}&
\left\|\dot\Delta_j\nabla_x\Big(
(\varrho u)^\ell,a^\ell\widetilde V,
(b^\ell\otimes b^\ell)^\ell\Big)\right\|_{L^2_x}\\
&+\left\|\dot\Delta_j\Big(
(u\cdot\nabla_xu)^\ell,g(\varrho,u,f)^\ell,
a^\ell(u^\ell-\widetilde V)
\Big)\right\|_{L^2_x}\\
&+\left\|\dot\Delta_j\left(
\frac{a^\ell}{1+a^\ell}
\Big(
\nabla_xa^\ell
+\dive_x\Theta(\{\mathbf I-\mathbf P\}f^\ell)
\Big),
\frac{\widetilde V}{1+a^\ell}\dive_xb^\ell
\right)\right\|_{L^2_x}\\
&+\|\dot\Delta_jh(\varrho,u,f)^\ell\|_{L^2_{x,v}}
+\left\|\dot\Delta_j\Big(
\mathcal C_\ell(a,u),
\frac{\mathcal C_\ell(a,u)}{1+a^\ell},
\mathcal D_\ell(u,b),
(\mathcal C_\ell(a,u)\otimes b^\ell)^\ell
\Big)\right\|_{L^2_x}\\
&+\left\|\dot\Delta_j\Big(
(a^\ell u^\ell\otimes b^\ell)^\ell,
\big(\big(\nabla_xa^\ell
+\dive_x\Theta(\{\mathbf I-\mathbf P\}f^\ell)
\big)\otimes b^\ell\big)^\ell
\Big)\right\|_{L^2_x}.
\end{aligned}
\end{equation}
\end{lemma}
\begin{proof}
 We apply
\(\dot\Delta_j\) to
\eqref{e33}--\eqref{e35} and take the
\(L^2(\mathbb R^d_x)\) and \(L^2(\mathbb R^d_x\times\mathbb R^d_v)\) inner products with
\(\dot\Delta_j(\varrho^\ell,u^\ell,a^\ell,
\widetilde V,\widetilde G)\), respectively. Since \(\mathbf P\widetilde G=0\), \eqref{e3} gives
\begin{equation}\label{e38}
\begin{aligned}
&\frac12\frac{d}{dt}
\|\dot\Delta_j(\varrho^\ell,u^\ell,a^\ell,
\widetilde V)\|_{L^2_x}^2
+\frac12\frac{d}{dt}
\|\dot\Delta_j\widetilde G\|_{L^2_{x,v}}^2\\
&\quad+\nu_*\|\nabla_x\dot\Delta_ju^\ell\|_{L^2_x}^2
+\|\dot\Delta_j(u^\ell-\widetilde V)\|_{L^2_x}^2+\lambda_0\|
\dot\Delta_j\widetilde G\|_{L^2_xL^2_{v,\nu}}^2\\
&\quad
\lesssim
\mathcal R_{\ell,j}\Big(
\|\dot\Delta_j(\varrho^\ell,u^\ell,a^\ell,
\widetilde V)\|_{L^2_x}
+\|\dot\Delta_j\widetilde G\|_{L^2_{x,v}}\Big),
\end{aligned}
\end{equation}
with \(\nu_*:=\min\{\mu,2\mu+\lambda\}>0\). 

To capture the dissipation, we take the \(L^2(\mathbb R^d_x)\) inner products of the
\(\dot\Delta_j\)-localized in
$\eqref{e33}_2$ and $\eqref{e34}_2$ with
\(\nabla_x\dot\Delta_j a^\ell\) and
\(\nabla_x\dot\Delta_j\varrho^\ell\), respectively, and use the
corresponding first equations to obtain
\begin{align}
&\frac d{dt}\bigl(\dot\Delta_ju^\ell,\, \nabla_x\dot\Delta_j\varrho^\ell\bigr)_{L^2_x}+\|\nabla_x\dot\Delta_j\varrho^\ell\|_{L^2_x}^2\nonumber\\
&\quad\leq \|\dive_x\dot\Delta_ju^\ell\|_{L^2_x}^2-\bigl(\dot\Delta_j(u^\ell-\widetilde V),\, \nabla_x\dot\Delta_j\varrho^\ell\bigr)_{L^2_x}\nonumber\\
&\qquad+\bigl(\mu\Delta_x\dot\Delta_ju^\ell+(\mu+\lambda)\nabla_x\dive_x\dot\Delta_ju^\ell,\, \nabla_x\dot\Delta_j\varrho^\ell\bigr)_{L^2_x}\nonumber\\
&\qquad+C\mathcal R_{\ell,j}\Big(\|\dot\Delta_j(\varrho^\ell,u^\ell,a^\ell,\widetilde V)\|_{L^2_x}+\|\dot\Delta_j\widetilde G\|_{L^2_{x,v}}\Big),
\label{e39}
\end{align}
and
\begin{align}
&\frac d{dt}\bigl(\dot\Delta_j\widetilde V,\, \nabla_x\dot\Delta_ja^\ell\bigr)_{L^2_x}+\|\nabla_x\dot\Delta_ja^\ell\|_{L^2_x}^2\nonumber\\
&\quad\leq \|\dive_x\dot\Delta_j\widetilde V\|_{L^2_x}^2+\bigl(\dot\Delta_j(u^\ell-\widetilde V),\, \nabla_x\dot\Delta_ja^\ell\bigr)_{L^2_x}\nonumber\\
&\qquad-\bigl(\dive_x\Theta(\dot\Delta_j\widetilde G),\, \nabla_x\dot\Delta_ja^\ell\bigr)_{L^2_x}+C\mathcal R_{\ell,j}\Big(\|\dot\Delta_j(\varrho^\ell,u^\ell,a^\ell,\widetilde V)\|_{L^2_x}+\|\dot\Delta_j\widetilde G\|_{L^2_{x,v}}\Big).
\label{e40}
\end{align}

We define
\begin{equation*}
\begin{aligned}
\mathcal L_{\ell,j}:={}&
\frac12\|\dot\Delta_j(\varrho^\ell,u^\ell,a^\ell,
\widetilde V)\|_{L^2_x}^2
+\frac12\|\dot\Delta_j\widetilde G\|_{L^2_{x,v}}^2\\
&+\eta_1
\bigl(\dot\Delta_ju^\ell,
\nabla_x\dot\Delta_j\varrho^\ell\bigr)_{L^2_x}
+\eta_2
\bigl(\dot\Delta_j\widetilde V,
\nabla_x\dot\Delta_ja^\ell\bigr)_{L^2_x}.
\end{aligned}
\end{equation*}
For \(j\leq1\), Bernstein's inequality gives
\begin{equation*}
\left|\bigl(\dot\Delta_ju^\ell,\nabla_x\dot\Delta_j\varrho^\ell\bigr)_{L^2_x}\right|
+\left|\bigl(\dot\Delta_j\widetilde V,\nabla_x\dot\Delta_ja^\ell\bigr)_{L^2_x}\right|
\lesssim\|\dot\Delta_j(u^\ell,\varrho^\ell)\|_{L^2_x}^2
+\|\dot\Delta_j(\widetilde V,a^\ell)\|_{L^2_x}^2.
\end{equation*}
We multiply these two inequalities by \(\eta_1\) and \(\eta_2\),
respectively, and add them to \eqref{e38}.  Note that
\begin{align}
&2^{2j}\|\dot\Delta_j\widetilde V\|_{L^2_x}^2
\lesssim
\|\nabla_x\dot\Delta_ju^\ell\|_{L^2_x}^2
+\|\dot\Delta_j(u^\ell-\widetilde V)\|_{L^2_x}^2,
\label{e42}\\
&2^{2j}\|\dot\Delta_j\widetilde G\|_{L^2_{x,v}}^2
\lesssim
\|
\dot\Delta_j\widetilde G\|_{L^2_xL^2_{v,\nu}}^2.
\label{e43}
\end{align}
We now apply Young's inequality to the linear pairings in
\eqref{e39}--\eqref{e40}, using
at most half of each weighted density-gradient term. Since $j\leq1$,
the remaining terms are bounded by $C(\eta_1+\eta_2)$ times the basic
viscous, relative-velocity and microscopic dissipations. Choosing
\(\eta_1,\eta_2>0\) sufficiently small ensures \eqref{ELsim}
and absorbs these terms into \eqref{e38}, yielding
\begin{align*}
\frac{d}{dt}\mathcal L_{\ell,j}
&+\frac{\nu_*}{2}\|\nabla_x\dot\Delta_ju^\ell\|_{L^2_x}^2
+\frac{\eta_1}{2}\|\nabla_x\dot\Delta_j\varrho^\ell\|_{L^2_x}^2
+\frac{\eta_2}{2}\|\nabla_x\dot\Delta_ja^\ell\|_{L^2_x}^2\\
&+\frac12\|\dot\Delta_j(u^\ell-\widetilde V)\|_{L^2_x}^2
+\frac{\lambda_0}{2}\|
\dot\Delta_j\widetilde G\|_{L^2_xL^2_{v,\nu}}^2
\lesssim\mathcal R_{\ell,j}\sqrt{\mathcal L_{\ell,j}}.
\end{align*}
Finally, \eqref{ELsim}, \eqref{e42}--\eqref{e43} and Bernstein's
inequality show that
\[
2^{2j}\mathcal L_{\ell,j}
\lesssim
\|\nabla_x\dot\Delta_ju^\ell\|_{L^2_x}^2
+2^{2j}\|\dot\Delta_j(\varrho^\ell,a^\ell)\|_{L^2_x}^2
+\|\dot\Delta_j(u^\ell-\widetilde V)\|_{L^2_x}^2
+\|
\dot\Delta_j\widetilde G\|_{L^2_xL^2_{v,\nu}}^2.
\]
This proves \eqref{low11}.
\end{proof}

Before proving the low-frequency a priori estimate, we record the
bounds for $\rho=1+\varrho$. Bernstein's
inequality and the embedding
$\dot B^{d/2}_{2,1}\hookrightarrow L^\infty(\mathbb R^d_x)$ give 
\begin{equation}\label{e44}
 \frac12\leq\rho(t,x)=1+\varrho(t,x)\leq\frac32
 \quad\text{for}\quad (t,x)\in[0,T]\times\mathbb R^d,
\end{equation}
provided that $\delta_0^*$ is chosen to be small.
\begin{lemma}\label{l2}
Let \(T>0\), and let \((\varrho,u,f)\) be a solution to the
Cauchy problem \eqref{m1n} for the NS--VFP system on \([0,T]\).
Suppose that \eqref{e25}
holds. Then, for every \(t\in[0,T]\), we have
\begin{equation}\label{e45}
\begin{aligned}
&\|(\varrho^\ell,u^\ell)\|_{\widetilde L^{\infty}_t
(\dot B^{d/2-1}_{2,1})\cap L^1_t
(\dot B^{d/2+1}_{2,1})}
+\|f^\ell\|_{\widetilde L^{\infty}_t
(\dot{\mathcal B}^{d/2-1}_{2,1})\cap L^1_t
(\dot{\mathcal B}^{d/2+1}_{2,1})}\\
&\quad+\|(u-b)^\ell\|_{\widetilde L^2_t
(\dot B^{d/2-1}_{2,1})}
+\|
\{\mathbf I-\mathbf P\}f^\ell\|_{\widetilde L^2_t
(\dot{\mathcal B}^{d/2-1}_{2,1,\nu})}\\
&\qquad\lesssim
\|(\varrho_0^\ell,u_0^\ell)\|_{
\dot B^{d/2-1}_{2,1}}
+\|f_0^\ell\|_{\dot{\mathcal B}^{d/2-1}_{2,1}}
+\|f_0^\ell\|_{\dot{\mathcal B}^{d/2-1}_{2,1}}^2
+\|(\varrho,u,f)\|_{\mathcal E_t}^2+\|(\varrho,u,f)\|_{\mathcal E_t}^3.
\end{aligned}
\end{equation}
\end{lemma}

\begin{proof}
Fix \(j\leq1\). We divide \eqref{low11} by
\(2\sqrt{\mathcal L_{\ell,j}+\kappa^2}\), integrate in time and let
\(\kappa\to0\). Then one has
\begin{equation}\label{e46}
\sup_{0\leq\tau\leq t}\sqrt{\mathcal L_{\ell,j}(\tau)}
+c2^{2j}\int_0^t\sqrt{\mathcal L_{\ell,j}(\tau)}\,d\tau
\lesssim \sqrt{\mathcal L_{\ell,j}(0)}
+\|\mathcal R_{\ell,j}\|_{L^1_t}.
\end{equation}
Integrating \eqref{low11} and using Young's inequality, we obtain
\begin{align}
\|\dot\Delta_j(u^\ell-\widetilde V)\|_{L^2_tL^2_x}
+\|
\dot\Delta_j\widetilde G\|_{L^2_tL^2_xL^2_{v,\nu}}&\lesssim
\left(
\mathcal L_{\ell,j}(0)
+\|\mathcal R_{\ell,j}\|_{L^1_t}
\sup_{0\leq\tau\leq t}
\sqrt{\mathcal L_{\ell,j}(\tau)}
\right)^{1/2}
\nonumber\\
&\quad\lesssim
\sqrt{\mathcal L_{\ell,j}(0)}
+\|\mathcal R_{\ell,j}\|_{L^1_t}+\sup_{0\leq\tau\leq t}\sqrt{\mathcal L_{\ell,j}(\tau)}.
\label{e47}
\end{align}
Since the external low-frequency seminorm is defined by summation over
\(j\leq1\), and
\(\|z^\ell\|_{\dot B^s_{2,1}}^\ell
=\|z^\ell\|_{\dot B^s_{2,1}}\), we multiply \eqref{e46} and
\eqref{e47} by \(2^{j(d/2-1)}\), sum over
\(j\leq1\) and use \eqref{ELsim} to get
\begin{equation}\label{e48}
\begin{aligned}
&\|(\varrho^\ell,u^\ell,a^\ell,\widetilde V)\|_
{\widetilde L^\infty_t
(\dot B^{d/2-1}_{2,1})}^{\ell}
+\|\widetilde G\|_{\widetilde L^\infty_t
(\dot{\mathcal B}^{d/2-1}_{2,1})}^{\ell}\\
&\quad+\|(\varrho^\ell,u^\ell,a^\ell,\widetilde V)\|_
{L^1_t(\dot B^{d/2+1}_{2,1})}^{\ell}
+\|\widetilde G\|_{L^1_t
(\dot{\mathcal B}^{d/2+1}_{2,1})}^{\ell}+\|u^\ell-\widetilde V\|_{\widetilde L^2_t
(\dot B^{d/2-1}_{2,1})}^{\ell}
+\|\widetilde G\|_{\widetilde L^2_t
(\dot{\mathcal B}^{d/2-1}_{2,1,\nu})}^{\ell}\\
&\qquad\lesssim
\|(\varrho_0^\ell,u_0^\ell,a_0^\ell,\widetilde V_0)\|_
{\dot B^{d/2-1}_{2,1}}^{\ell}
+\|\widetilde G_0\|_{\dot{\mathcal B}^{d/2-1}_{2,1}}^{\ell}
+\sum_{j\leq1}2^{j(d/2-1)}
\|\mathcal R_{\ell,j}\|_{L^1_t}.
\end{aligned}
\end{equation}
By \eqref{e20}, \eqref{e30} and \eqref{uv2}, it follows that
\begin{equation}\label{e49}
\|(a_0^\ell,\widetilde V_0)\|_{\dot B^{d/2-1}_{2,1}}^{\ell}
+\|\widetilde G_0\|_{\dot{\mathcal B}^{d/2-1}_{2,1}}^{\ell}
\lesssim \|f_0^\ell\|_{\dot{\mathcal B}^{d/2-1}_{2,1}}+\|f_0^\ell\|_{\dot{\mathcal B}^{d/2-1}_{2,1}}^2.
\end{equation}

We now estimate the remainder in \eqref{e48}.
From the definition of \(\mathcal E_t\) and interpolation, one has
\begin{equation}\label{e50}
\begin{aligned}
&\|(\varrho,u)\|_{L^2_t
(\dot B^{d/2}_{2,1})}
+\|(a^\ell,b^\ell)\|_{L^2_t
(\dot B^{d/2}_{2,1})}
+\|(a^h,b^h)\|_{L^2_t
(\dot B^{d/2-1}_{2,1})}\\
&\quad+\|f^\ell\|_{L^2_t
(\dot{\mathcal B}^{d/2}_{2,1})}+\|f^h\|_{L^2_t
(\dot{\mathcal B}^{d/2-1}_{2,1})}
+\|\{\mathbf I-\mathbf P\}f\|_
{L^2_t
(\dot{\mathcal B}^{d/2-1}_{2,1,\nu})}
\lesssim \|(\varrho,u,f)\|_{\mathcal E_t}.
\end{aligned}
\end{equation}
Since $\widetilde V=b^\ell-\frac{a^\ell}{1+a^\ell}b^\ell$ and $u^\ell-\widetilde V=(u-b)^\ell
+\frac{a^\ell}{1+a^\ell}b^\ell$, the product and composition estimates, together with
\eqref{Xt}, \eqref{e30} and
\eqref{e50}, yield
\begin{equation}\label{e51}
\begin{aligned}
\|\widetilde V\|_{L^2_t(\dot B^{d/2}_{2,1})}
&\leq
\|b^\ell\|_{L^2_t(\dot B^{d/2}_{2,1})}
+\left\|\frac{a^\ell}{1+a^\ell}b^\ell
\right\|_{L^2_t(\dot B^{d/2}_{2,1})}\\
&\lesssim
\Bigl(1+\|a^\ell\|_{L^\infty_t
(\dot B^{d/2}_{2,1})}\Bigr)
\|b^\ell\|_{L^2_t(\dot B^{d/2}_{2,1})}\lesssim \|(\varrho,u,f)\|_{\mathcal E_t},\\
\|u^\ell-\widetilde V\|_{L^2_t(\dot B^{d/2-1}_{2,1})}
&\leq
\|(u-b)^\ell\|_{L^2_t(\dot B^{d/2-1}_{2,1})}
+\left\|\frac{a^\ell}{1+a^\ell}b^\ell
\right\|_{L^2_t(\dot B^{d/2-1}_{2,1})}\\
&\lesssim
\|u-b\|_{L^2_t(\dot B^{d/2-1}_{2,1})}+
\|a^\ell\|_{L^2_t(\dot B^{d/2}_{2,1})}
\|b^\ell\|_{L^\infty_t
(\dot B^{d/2-1}_{2,1})}\lesssim \|(\varrho,u,f)\|_{\mathcal E_t}.
\end{aligned}
\end{equation}
 The bounds
\eqref{e30} and \eqref{e44} also give control of the coefficients \((1+a^\ell)^{-1}\) and
\((1+\varrho)^{-1}\).

Every term in \eqref{e37} can be analyzed as follows.

\smallskip
\noindent\emph{Terms with one spatial derivative.}
The product and composition estimates (Lemmas \ref{l11} and \ref{l12}) as well as \eqref{e50} yield
\begin{align}
&\|\nabla_x(\varrho u)\|_{L^1_t
(\dot B^{d/2-1}_{2,1})}
+\|u\cdot\nabla_xu\|_{L^1_t
(\dot B^{d/2-1}_{2,1})}
+\|\nabla_x(b^\ell\otimes b^\ell)\|_{L^1_t
(\dot B^{d/2-1}_{2,1})} \nonumber\\
&\quad+\left\|\Big(
\frac{P'(1+\varrho)}{1+\varrho}-1\Big)\nabla_x\varrho
\right\|_{L^1_t(\dot B^{d/2-1}_{2,1})}+\left\|\frac{\varrho}{1+\varrho}
\Big(\mu\Delta_xu+(\mu+\lambda)\nabla_x\dive_xu\Big)
\right\|_{L^1_t(\dot B^{d/2-1}_{2,1})} \nonumber\\
&\quad+\left\|\frac{a^\ell}{1+a^\ell}
\Big(\nabla_xa^\ell
+\dive_x\Theta(\{\mathbf I-\mathbf P\}f^\ell)\Big)
\right\|_{L^1_t(\dot B^{d/2-1}_{2,1})} \nonumber\\
&\quad+\left\|\left(
\nabla_xa^\ell
+\dive_x\Theta(\{\mathbf I-\mathbf P\}f^\ell)
\right)\otimes b^\ell
\right\|_{L^1_t(\dot B^{d/2-1}_{2,1})}\nonumber\\
&\lesssim \|(\varrho,u,f)\|_{\mathcal E_t}^2,\label{e52}
\end{align}
and
\begin{equation*}
\|\nabla_x(a^\ell\widetilde V)\|_{L^1_t
(\dot B^{d/2-1}_{2,1})}
+\left\|\frac{\widetilde V}{1+a^\ell}
\dive_xb^\ell\right\|_{L^1_t
(\dot B^{d/2-1}_{2,1})}
\lesssim
\|(a^\ell,b^\ell)\|_{L^2_t
(\dot B^{d/2}_{2,1})}
\|\widetilde V\|_{L^2_t
(\dot B^{d/2}_{2,1})}.
\end{equation*}

\smallskip
\noindent\emph{Lower-order terms.}
By \eqref{e50} and \eqref{e51}, we have
\begin{equation}\label{e54}
\begin{aligned}
&\|a^\ell(u^\ell-\widetilde V)\|_
{L^1_t(\dot B^{d/2-1}_{2,1})}
\lesssim
\|a^\ell\|_{L^2_t
(\dot B^{d/2}_{2,1})}
\|u^\ell-\widetilde V\|_{L^2_t
(\dot B^{d/2-1}_{2,1})},\\
&\|h(\varrho,u,f)\|_{L^1_t
(\dot{\mathcal B}^{d/2-1}_{2,1})}
\lesssim
\|u\|_{L^2_t
(\dot B^{d/2}_{2,1})}
\|\{\mathbf I-\mathbf P\}f\|_
{L^2_t
(\dot{\mathcal B}^{d/2-1}_{2,1,\nu})}
\lesssim\|(\varrho,u,f)\|_{\mathcal E_t}^2,\\
&\left\|\frac{\varrho}{1+\varrho}(u-b)\right\|_
{L^1_t(\dot B^{d/2-1}_{2,1})}
\lesssim
\|\varrho\|_{L^2_t
(\dot B^{d/2}_{2,1})}
\|u-b\|_{L^2_t
(\dot B^{d/2-1}_{2,1})}
\lesssim\|(\varrho,u,f)\|_{\mathcal E_t}^2.
\end{aligned}
\end{equation}

\smallskip
\noindent\emph{Trilinear terms.}
The product and composition estimates, together with \eqref{e50}, imply
\begin{equation}\label{e55}
\begin{aligned}
&\left\|\frac{\varrho}{1+\varrho}au\right\|_
{L^1_t(\dot B^{d/2-1}_{2,1})}
+\|a^\ell u^\ell\otimes b^\ell\|_{L^1_t
(\dot B^{d/2-1}_{2,1})}\\
&\quad\lesssim \|(\varrho, a)\|_{L^\infty_t(\dot{B}^{d/2-1}_{2,1})} \|(\varrho,b^\ell)\|_{L^2_t(\dot{B}^{d/2}_{2,1})}\|u\|_{L^2_t(\dot{B}^{d/2}_{2,1})}\lesssim \|(\varrho,u,f)\|_{\mathcal E_t}^3.
\end{aligned}
\end{equation}

\smallskip
\noindent\emph{Truncation defects.}
It follows from \eqref{e36}, \eqref{e50} and \eqref{uv2} that
\begin{equation*}
\begin{aligned}
&\|(a^\ell u^\ell)^h\|_{L^1_t
(\dot B^{d/2-1}_{2,1})}
\lesssim
\|(a^\ell u^\ell)^h\|_{L^1_t
(\dot B^{d/2}_{2,1})}
\lesssim
\|a^\ell\|_{L^2_t
(\dot B^{d/2}_{2,1})}
\|u^\ell\|_{L^2_t
(\dot B^{d/2}_{2,1})},\\
&\|(a^hu+a^\ell u^h)^\ell\|_{L^1_t
(\dot B^{d/2-1}_{2,1})}
\lesssim
\|a^h\|_{L^2_t
(\dot B^{d/2-1}_{2,1})}
\|u\|_{L^2_t
(\dot B^{d/2}_{2,1})}
+\|a^\ell\|_{L^2_t
(\dot B^{d/2}_{2,1})}
\|u^h\|_{L^2_t
(\dot B^{d/2-1}_{2,1})},\\
&\|\mathcal D_\ell(u,b)\|_{L^1_t
(\dot B^{d/2-1}_{2,1})}
\lesssim
\|u^h\|_{L^1_t
(\dot B^{d/2+1}_{2,1})}
\|b\|_{L^\infty_t
(\dot B^{d/2-1}_{2,1})}
+\|u^\ell\|_{L^2_t
(\dot B^{d/2}_{2,1})}
\|b^h\|_{L^2_t
(\dot B^{d/2-1}_{2,1})},
\end{aligned}
\end{equation*}
where we use $\|z^h\|_{\dot B^{d/2-1}_{2,1}}\lesssim
\|z^h\|_{\dot B^{d/2}_{2,1}}$. Consequently, one has
$\|\mathcal C_\ell(a,u)\|_{L^1_t
(\dot B^{d/2-1}_{2,1})}
+\|\mathcal D_\ell(u,b)\|_{L^1_t
(\dot B^{d/2-1}_{2,1})}
\lesssim\|(\varrho,u,f)\|_{\mathcal E_t}^2$ and
\begin{equation*}
\left\|\left(\frac{\mathcal C_\ell(a,u)}{1+a^\ell},
\mathcal C_\ell(a,u)\otimes b^\ell\right)\right\|_
{L^1_t(\dot B^{d/2-1}_{2,1})}
\lesssim\|(\varrho,u,f)\|_{\mathcal E_t}^2
+\|(\varrho,u,f)\|_{\mathcal E_t}^3.
\end{equation*}

By \eqref{e4},
\eqref{e52},
\eqref{e54} and
\eqref{e55}, one obtains
\begin{equation*}
\|g(\varrho,u,f)\|_{L^1_t
(\dot B^{d/2-1}_{2,1})}
\lesssim\|(\varrho,u,f)\|_{\mathcal E_t}^2
+\|(\varrho,u,f)\|_{\mathcal E_t}^3.
\end{equation*}

The preceding estimates imply
\begin{equation}\label{e59}
\begin{aligned}
\sum_{j\leq1}2^{j(d/2-1)}
\|\mathcal R_{\ell,j}\|_{L^1_t}
&\lesssim\|(\varrho,u,f)\|_{\mathcal E_t}^2+\|(\varrho,u,f)\|_{\mathcal E_t}^3\\
&\quad+\|(\varrho,u,f)\|_{\mathcal E_t}
\Big(\|\widetilde V\|_{L^2_t
(\dot B^{d/2}_{2,1})}
+\|u^\ell-\widetilde V\|_{L^2_t
(\dot B^{d/2-1}_{2,1})}\Big).
\end{aligned}
\end{equation}
The last line is bounded by
$\|(\varrho,u,f)\|_{\mathcal E_t}^2$ in view of
\eqref{e51}. Inserting
\eqref{e49} and
\eqref{e59} into
\eqref{e48}, one obtains
\begin{equation}\label{e60}
\begin{aligned}
&\|(\varrho^\ell,u^\ell,a^\ell,\widetilde V)\|_
{\widetilde L^\infty_t
(\dot B^{d/2-1}_{2,1})}^{\ell}
+\|\widetilde G\|_{\widetilde L^\infty_t
(\dot{\mathcal B}^{d/2-1}_{2,1})}^{\ell}\\
&\quad+\|(\varrho^\ell,u^\ell,a^\ell,\widetilde V)\|_
{L^1_t(\dot B^{d/2+1}_{2,1})}^{\ell}
+\|\widetilde G\|_{L^1_t
(\dot{\mathcal B}^{d/2+1}_{2,1})}^{\ell}+\|u^\ell-\widetilde V\|_{\widetilde L^2_t
(\dot B^{d/2-1}_{2,1})}^{\ell}
+\|\widetilde G\|_{\widetilde L^2_t
(\dot{\mathcal B}^{d/2-1}_{2,1,\nu})}^{\ell}\\
&\qquad\lesssim
\|(\varrho_0^\ell,u_0^\ell)\|_{\dot B^{d/2-1}_{2,1}}
+\|f_0^\ell\|_{\dot{\mathcal B}^{d/2-1}_{2,1}}
+\|f_0^\ell\|_{\dot{\mathcal B}^{d/2-1}_{2,1}}^2
+\|(\varrho,u,f)\|_{\mathcal E_t}^2
+\|(\varrho,u,f)\|_{\mathcal E_t}^3.
\end{aligned}
\end{equation}

It remains to recover the original variables. We use
$b^\ell=(1+a^\ell)\widetilde V$,
$\{\mathbf I-\mathbf P\}f^\ell=\widetilde G
+\frac12(v\otimes v-\mathrm{Id}):(b^\ell\otimes b^\ell)^\ell
M^{1/2}$ and
$(u-b)^\ell=(u^\ell-\widetilde V)-a^\ell\widetilde V$.
The low-frequency localization, \eqref{e50},
\eqref{e60} and the product estimates
\eqref{uv1}--\eqref{uv2} imply
\begin{equation}\label{e61}
\begin{aligned}
&\|a^\ell\widetilde V\|_{\widetilde L^\infty_t
(\dot B^{d/2-1}_{2,1})}^{\ell}
+\|a^\ell\widetilde V\|_{L^1_t
(\dot B^{d/2+1}_{2,1})}^{\ell}
+\|a^\ell\widetilde V\|_{\widetilde L^2_t
(\dot B^{d/2-1}_{2,1})}^{\ell}\\
&\quad+\|(b^\ell\otimes b^\ell)^\ell\|_{\widetilde L^\infty_t
(\dot B^{d/2-1}_{2,1})}
+\|(b^\ell\otimes b^\ell)^\ell\|_{L^1_t
(\dot B^{d/2+1}_{2,1})}+\|(b^\ell\otimes b^\ell)^\ell\|_{\widetilde L^2_t
(\dot B^{d/2-1}_{2,1})}
\\
&\quad\quad \lesssim\|(\varrho,u,f)\|_{\mathcal E_t}^2.
\end{aligned}
\end{equation}
The estimates in \eqref{e45} now follow from \eqref{e60}. For example,
\eqref{E35}, \eqref{e60} and \eqref{e61} imply
\begin{align*}
\|f^\ell\|_{\widetilde L^\infty_t
(\dot{\mathcal B}^{d/2-1}_{2,1})}
&\lesssim
\|(a^\ell,\widetilde V)\|_{\widetilde L^\infty_t
(\dot B^{d/2-1}_{2,1})}^{\ell}
+\|\widetilde G\|_{\widetilde L^\infty_t
(\dot{\mathcal B}^{d/2-1}_{2,1})}^{\ell}\\
&\quad
+\|a^\ell\widetilde V\|_{\widetilde L^\infty_t
(\dot B^{d/2-1}_{2,1})}^{\ell}
+\|(b^\ell\otimes b^\ell)^\ell\|_{\widetilde L^\infty_t
(\dot B^{d/2-1}_{2,1})}.
\end{align*}
The \(L^1_t\) estimate for \(f^\ell\) and the \(L^2_t\) estimates for
\((u-b)^\ell\) and
\(\{\mathbf I-\mathbf P\}f^\ell\) in the velocity dissipation norm follow in the same
way. Hence, by
\eqref{e60} and \eqref{e61}, we 
prove \eqref{e45}.
\end{proof}

\subsection{High-frequency analysis}\label{s5}

We establish a high-frequency Lyapunov estimate exhibiting
exponential damping for each high-frequency block.

\begin{lemma}\label{l3}
For every $j\geq0$, there exists a functional
$\mathcal L_{h,j}$ such that
\begin{equation}\label{EHsim}
\mathcal L_{h,j}\sim
\|\dot\Delta_j(\varrho,\nabla_x\varrho,u)\|_{L^2_x}^2
+\|\dot\Delta_jf\|_{L^2_{x,v}}^2,
\end{equation}
and there exist constants $c,C>0$, independent of $j$, such that
\begin{equation}\label{e62}
\begin{aligned}
&\frac{d}{dt}\mathcal L_{h,j}
+c\mathcal L_{h,j}
+c\|\dot\Delta_j(u-b)\|_{L^2_x}^2
+c\|
\dot\Delta_j\{\mathbf I-\mathbf P\}f\|_{L^2_xL^2_{v,\nu}}^2\\
&\lesssim\Big(
2^j\|\dot\Delta_j(\varrho u)\|_{L^2_x}
+\|\dot\Delta_j(au)\|_{L^2_x}
+\|\dot\Delta_j g(\varrho,u,f)\|_{L^2_x}\\
&\qquad
+\left\|\dot\Delta_j\Big(
\frac12u\cdot v\,\mathbf Pf
-u\cdot\nabla_v\mathbf Pf+h(\varrho,u,f)\Big)\right\|_{L^2_{x,v}}\\
&\qquad
+\|\nabla_xu\|_{L^\infty_x}
\|\dot\Delta_j(\nabla_x\varrho,u)\|_{L^2_x}
+2^j\|\dot\Delta_j(\varrho\dive_xu)\|_{L^2_x}
+\mathcal R_j\Big)\sqrt{\mathcal L_{h,j}},
\end{aligned}
\end{equation}
where
\begin{equation*}
\mathcal R_j:=
\|[u,\dot\Delta_j]\cdot\nabla_xu\|_{L^2_x}
+\left\|\nabla_x\bigl(
[u,\dot\Delta_j]\cdot\nabla_x\varrho\bigr)\right\|_{L^2_x}.
\end{equation*}
\end{lemma}

\begin{proof}
We apply $\dot\Delta_j$ to \eqref{m1n}, take the corresponding $L^2(\mathbb R^d_x)$ and $L^2(\mathbb R^d_x\times\mathbb R^d_v)$ inner
products with $\dot\Delta_j(\varrho,u,f)$ and use
\eqref{e3} to obtain
\begin{equation}\label{E1}
\begin{aligned}
&\frac12\frac{d}{dt}\Big(
\|\dot\Delta_j(\varrho,u)\|_{L^2_x}^2
+\|\dot\Delta_jf\|_{L^2_{x,v}}^2\Big)
+\nu_*\|\nabla_x\dot\Delta_ju\|_{L^2_x}^2
+\|\dot\Delta_j(u-b)\|_{L^2_x}^2\\
&\quad
+\lambda_0\|
\dot\Delta_j\{\mathbf I-\mathbf P\}f\|_{L^2_xL^2_{v,\nu}}^2\\
&\lesssim
\Big(2^j\|\dot\Delta_j(\varrho u)\|_{L^2_x}
+\|\dot\Delta_j(au)\|_{L^2_x}
+\|\dot\Delta_j g(\varrho,u,f)\|_{L^2_x}\Big)
\|\dot\Delta_j(\varrho,u)\|_{L^2_x}\\
&\quad
+\left\|\dot\Delta_j\Big(
\frac12u\cdot v\,\mathbf Pf
-u\cdot\nabla_v\mathbf Pf+h(\varrho,u,f)\Big)\right\|_{L^2_{x,v}}
\|\dot\Delta_jf\|_{L^2_{x,v}}\\
&\quad
+\Big(\|\dive_xu\|_{L^\infty_x}
\|\dot\Delta_ju\|_{L^2_x}
+\|[u,\dot\Delta_j]\cdot\nabla_xu\|_{L^2_x}\Big)
\|\dot\Delta_ju\|_{L^2_x}.
\end{aligned}
\end{equation}
By \eqref{e22}, we infer
\begin{equation}\label{E3}
\begin{aligned}
&\frac{d}{dt}\int_{\mathbb R^d_x}
\nabla_x\dot\Delta_ja\cdot\dot\Delta_jb\,dx
+\|\nabla_x\dot\Delta_ja\|_{L^2_x}^2
-\|\dive_x\dot\Delta_jb\|_{L^2_x}^2\\
&\quad+\int_{\mathbb R^d_x}
\Big(\dot\Delta_j(b-u)
+\dive_x\Theta(\dot\Delta_j
\{\mathbf I-\mathbf P\}f)\Big)
\cdot\nabla_x\dot\Delta_ja\,dx\\
&\leq
\|\dot\Delta_j(au)\|_{L^2_x}
\|\nabla_x\dot\Delta_ja\|_{L^2_x}.
\end{aligned}
\end{equation}

Next, we estimate the higher-order energy of $\varrho$. To avoid losing
regularity, we rewrite the nonlinear terms as commutators. Set
\[
 R_{1,j}:=[u,\dot\Delta_j]\cdot\nabla_x\varrho
 -\dot\Delta_j(\varrho\dive_xu)\quad\text{and}\quad
 R_{2,j}:=[u,\dot\Delta_j]\cdot\nabla_xu
 -\dot\Delta_j(au)+\dot\Delta_j g(\varrho,u,f).
\]
The localized fluid equations are
\begin{align*}
&\partial_t\dot\Delta_j\varrho
+u\cdot\nabla_x\dot\Delta_j\varrho
+\dive_x\dot\Delta_ju
=R_{1,j},\\
&\partial_t\dot\Delta_ju
+u\cdot\nabla_x\dot\Delta_ju
+\nabla_x\dot\Delta_j\varrho
-\mu\Delta_x\dot\Delta_ju
-(\mu+\lambda)\nabla_x\dive_x\dot\Delta_ju
+\dot\Delta_ju-\dot\Delta_jb
=R_{2,j}.
\end{align*}
Direct differentiation gives
\begin{align}
&\frac d{dt}
\bigl(\dot\Delta_ju,\nabla_x\dot\Delta_j\varrho\bigr)_{L^2_x}
+\|\nabla_x\dot\Delta_j\varrho\|_{L^2_x}^2
-\|\dive_x\dot\Delta_ju\|_{L^2_x}^2
+\bigl(\dot\Delta_j(u-b),
\nabla_x\dot\Delta_j\varrho\bigr)_{L^2_x}
\nonumber\\
&\quad
-\bigl(
\mu\Delta_x\dot\Delta_ju
+(\mu+\lambda)\nabla_x\dive_x\dot\Delta_ju,
\nabla_x\dot\Delta_j\varrho
\bigr)_{L^2_x}
\nonumber\\
&=
-\bigl(
u\cdot\nabla_x\dot\Delta_ju,
\nabla_x\dot\Delta_j\varrho
\bigr)_{L^2_x}
-\bigl(
\dot\Delta_ju,
\nabla_x(u\cdot\nabla_x\dot\Delta_j\varrho)
\bigr)_{L^2_x}
\nonumber\\
&\qquad\quad
+\bigl(R_{2,j},
\nabla_x\dot\Delta_j\varrho\bigr)_{L^2_x}
+\bigl(\dot\Delta_ju,
\nabla_xR_{1,j}\bigr)_{L^2_x}.
\label{e64}
\end{align}
On the other hand, one has
\begin{align}
&\frac{2\mu+\lambda}{2}\frac d{dt}
\|\nabla_x\dot\Delta_j\varrho\|_{L^2_x}^2
+(2\mu+\lambda)
\bigl(
\nabla_x\dive_x\dot\Delta_ju,
\nabla_x\dot\Delta_j\varrho
\bigr)_{L^2_x}
\nonumber\\
&\qquad
=-(2\mu+\lambda)
\bigl(
\nabla_x(u\cdot\nabla_x\dot\Delta_j\varrho),
\nabla_x\dot\Delta_j\varrho
\bigr)_{L^2_x}
+(2\mu+\lambda)
\bigl(
\nabla_xR_{1,j},
\nabla_x\dot\Delta_j\varrho
\bigr)_{L^2_x}.
\label{e65}
\end{align}
An integration by parts gives
\begin{equation*}
\begin{aligned}
&\bigl(
\mu\Delta_x\dot\Delta_ju
+(\mu+\lambda)\nabla_x\dive_x\dot\Delta_ju,
\nabla_x\dot\Delta_j\varrho
\bigr)_{L^2_x}
=(2\mu+\lambda)
\bigl(
\nabla_x\dive_x\dot\Delta_ju,
\nabla_x\dot\Delta_j\varrho
\bigr)_{L^2_x},
\end{aligned}
\end{equation*}
and
\begin{align*}
&\left|
\bigl(
u\cdot\nabla_x\dot\Delta_ju,
\nabla_x\dot\Delta_j\varrho
\bigr)_{L^2_x}
+\bigl(
\dot\Delta_ju,
\nabla_x(u\cdot\nabla_x\dot\Delta_j\varrho)
\bigr)_{L^2_x}
\right|
+\left|
\bigl(
\nabla_x(u\cdot\nabla_x\dot\Delta_j\varrho),
\nabla_x\dot\Delta_j\varrho
\bigr)_{L^2_x}
\right|\\
&\qquad\lesssim
\|\nabla_xu\|_{L^\infty_x}
\|\dot\Delta_j(\nabla_x\varrho,u)\|_{L^2_x}^2.
\end{align*}
Moreover, Bernstein's inequality and the definitions of
\(R_{1,j}\), \(R_{2,j}\) lead to
\begin{align*}
&\left|
\bigl(
R_{2,j},
\nabla_x\dot\Delta_j\varrho
\bigr)_{L^2_x}
\right|
+\left|
\bigl(
\dot\Delta_ju,
\nabla_xR_{1,j}
\bigr)_{L^2_x}
\right|
+\left|
\bigl(
\nabla_xR_{1,j},
\nabla_x\dot\Delta_j\varrho
\bigr)_{L^2_x}
\right|\\
&\qquad\lesssim
\Big(
2^j\|\dot\Delta_j(\varrho\dive_xu)\|_{L^2_x}
+\|\dot\Delta_j(au)\|_{L^2_x}
+\|\dot\Delta_j g(\varrho,u,f)\|_{L^2_x}
+\mathcal R_j
\Big)
\|\dot\Delta_j(\nabla_x\varrho,u)\|_{L^2_x}.
\end{align*}
Then, we add \eqref{e64} and
\eqref{e65} to obtain
\begin{equation}\label{E2h}
\begin{aligned}
&\frac{d}{dt}\Big(
\int_{\mathbb R^d_x}\dot\Delta_ju\cdot
\nabla_x\dot\Delta_j\varrho\,dx
+\frac{2\mu+\lambda}{2}
\|\nabla_x\dot\Delta_j\varrho\|_{L^2_x}^2\Big)\\
&\quad+\|\nabla_x\dot\Delta_j\varrho\|_{L^2_x}^2
-\|\dive_x\dot\Delta_ju\|_{L^2_x}^2
+\int_{\mathbb R^d_x}\dot\Delta_j(u-b)\cdot
\nabla_x\dot\Delta_j\varrho\,dx\\
&\lesssim
\|\nabla_xu\|_{L^\infty_x}
\|\dot\Delta_j\nabla_x\varrho\|_{L^2_x}
\|\dot\Delta_j(\nabla_x\varrho,u)\|_{L^2_x}\\
&\quad+\Big(
2^j\|\dot\Delta_j(\varrho\dive_xu)\|_{L^2_x}
+\|\dot\Delta_j(au)\|_{L^2_x}
+\|\dot\Delta_j g(\varrho,u,f)\|_{L^2_x}
+\mathcal R_j\Big)
\|\dot\Delta_j(\nabla_x\varrho,u)\|_{L^2_x}.
\end{aligned}
\end{equation}

For a sufficiently small constant $\eta>0$, we define
\begin{equation}\label{e67}
\begin{aligned}
\mathcal L_{h,j}:={}&
\frac12\|\dot\Delta_j(\varrho,u)\|_{L^2_x}^2
+\frac12\|\dot\Delta_jf\|_{L^2_{x,v}}^2\\
&+\eta\Big(
\int_{\mathbb R^d_x}\dot\Delta_ju\cdot
\nabla_x\dot\Delta_j\varrho\,dx
+\frac{2\mu+\lambda}{2}
\|\nabla_x\dot\Delta_j\varrho\|_{L^2_x}^2\Big)+\eta 2^{-2j}\int_{\mathbb R^d_x}
\nabla_x\dot\Delta_ja\cdot\dot\Delta_jb\,dx.
\end{aligned}
\end{equation}
For $j\geq0$, one applies Bernstein's inequality and \eqref{E35} to control the two
cross terms in \eqref{e67}. Hence, if $\eta>0$ is sufficiently
small, we obtain \eqref{EHsim}. Moreover, one has
\[
\|\dot\Delta_ju\|_{L^2_x}^2
+\|\dot\Delta_j(u-b)\|_{L^2_x}^2
\geq\frac12\|\dot\Delta_jb\|_{L^2_x}^2.
\]
We combine
\eqref{E1}, $\eta2^{-2j}$ times \eqref{E3}, and $\eta$ times
\eqref{E2h}. Young's inequality uses at most half of
$\eta\|\nabla_x\dot\Delta_j\varrho\|_2^2$
and $\eta2^{-2j}\|\nabla_x\dot\Delta_ja\|_2^2$.
Moreover, for $j\geq0$,
\[
2^{-2j}\|\dive_x\dot\Delta_jb\|_2^2
\lesssim\|\nabla_x\dot\Delta_ju\|_2^2
+\|\dot\Delta_j(u-b)\|_2^2.
\]
Choosing $\eta$ sufficiently small, the remaining linear terms are
absorbed by the basic dissipations, which proves \eqref{e62}.
\end{proof}

We have the following lemma regarding the high-frequency estimates.

\begin{lemma}\label{l4}
Let $T>0$, and let $(\varrho,u,f)$ be a solution to the Cauchy
problem \eqref{m1n} for the NS--VFP system on $[0,T]$ satisfying
\eqref{e25}. Then, for every $t\in[0,T]$, we have
\begin{align}
&\|\varrho^h\|_{\widetilde L^\infty_t
(\dot B^{d/2}_{2,1})}
+\|u^h\|_{\widetilde L^\infty_t
(\dot B^{d/2-1}_{2,1})}
+\|f^h\|_{\widetilde L^\infty_t
(\dot{\mathcal B}^{d/2-1}_{2,1})}\nonumber\\
&\quad+\|\varrho^h\|_{L^1_t
(\dot B^{d/2}_{2,1})}
+\|u^h\|_{L^1_t
(\dot B^{d/2+1}_{2,1})}
+\|f^h\|_{L^1_t
(\dot{\mathcal B}^{d/2-1}_{2,1})}\nonumber\\
&\quad+\|(u-b)^h\|_{\widetilde L^2_t
(\dot B^{d/2-1}_{2,1})}
+\|
\{\mathbf I-\mathbf P\}f^h\|_{\widetilde L^2_t
(\dot{\mathcal B}^{d/2-1}_{2,1,\nu})}\nonumber\\
&\qquad\lesssim
\|(\varrho_0,u_0,f_0)\|_{\mathcal E_0}
+\|(\varrho,u,f)\|_{\mathcal E_t}^2
+\|(\varrho,u,f)\|_{\mathcal E_t}^3.\label{high}
\end{align}
\end{lemma}

\begin{proof}
We first work in the high-frequency
seminorms, as required by the commutator estimates and integrations
by parts. After summing over \(j\geq0\), the inequality
\(\|z^h\|_{\dot B^s_{2,1}}\lesssim
\|z\|_{\dot B^s_{2,1}}^h\) recovers the projected norms in \eqref{high}. Applying the same argument as in
\eqref{e46}--\eqref{e47} to \eqref{e62},
we obtain
\begin{equation}\label{e68}
\begin{aligned}
&\|\dot\Delta_j(\varrho,\nabla_x\varrho,u)\|_{L^\infty_tL^2_x}
+\|\dot\Delta_jf\|_{L^\infty_tL^2_{x,v}}\\
&\quad+\|\dot\Delta_j(\varrho,\nabla_x\varrho,u)\|_{L^1_tL^2_x}
+\|\dot\Delta_jf\|_{L^1_tL^2_{x,v}}+\|\dot\Delta_j(u-b)\|_{L^2_tL^2_x}
+\|
\dot\Delta_j\{\mathbf I-\mathbf P\}f\|_{L^2_tL^2_xL^2_{v,\nu}}\\
&\lesssim
\|\dot\Delta_j(\varrho_0,\nabla_x\varrho_0,u_0)\|_{L^2_x}
+\|\dot\Delta_jf_0\|_{L^2_{x,v}}\\
&\quad+2^j\|\dot\Delta_j(\varrho u)\|_{L^1_tL^2_x}
+\|\dot\Delta_j(au)\|_{L^1_tL^2_x}
+\|\dot\Delta_j g(\varrho,u,f)\|_{L^1_tL^2_x}\\
&\quad+\left\|\dot\Delta_j\Big(
a\,u\cdot v\,M^{1/2} +(v\otimes v-\mathrm{Id}):(u\otimes b)M^{1/2}+h(\varrho,u,f)\Big)
\right\|_{L^1_tL^2_{x,v}}\\
&\quad+\int_0^t\Big(
\|\nabla_xu\|_{L^\infty_x}
\|\dot\Delta_j(\nabla_x\varrho,u)\|_{L^2_x}
+2^j\|\dot\Delta_j(\varrho\dive_xu)\|_{L^2_x}
+\mathcal R_j\Big)d\tau.
\end{aligned}
\end{equation}
Then, we multiply \eqref{e68} by $2^{j(d/2-1)}$ and sum over
$j\geq0$ to get 
\begin{equation}\label{l1ua}
\begin{aligned}
&\|\varrho\|_{\widetilde L^\infty_t
(\dot B^{d/2}_{2,1})}^{h}
+\|u\|_{\widetilde L^\infty_t
(\dot B^{d/2-1}_{2,1})}^{h}
+\|f\|_{\widetilde L^\infty_t
(\dot{\mathcal B}^{d/2-1}_{2,1})}^{h}
+\|\varrho\|_{L^1_t
(\dot B^{d/2}_{2,1})}^{h}\\
&\quad+\|u\|_{L^1_t
(\dot B^{d/2-1}_{2,1})}^{h}
+\|f\|_{L^1_t
(\dot{\mathcal B}^{d/2-1}_{2,1})}^{h}
+\|u-b\|_{\widetilde L^2_t
(\dot B^{d/2-1}_{2,1})}^{h}
+\|
\{\mathbf I-\mathbf P\}f\|_{\widetilde L^2_t
(\dot{\mathcal B}^{d/2-1}_{2,1,\nu})}^{h}\\
&\lesssim
\|\varrho_0\|_{\dot B^{d/2}_{2,1}}^{h}
+\|u_0\|_{\dot B^{d/2-1}_{2,1}}^{h}
+\|f_0\|_{\dot{\mathcal B}^{d/2-1}_{2,1}}^{h}
+\sum_{j\geq0}2^{j(d/2-1)}
\int_0^t\mathcal R_j\,d\tau\\
&\quad+\|\varrho u\|_{L^1_t
(\dot B^{d/2}_{2,1})}^{h}
+\|(au,g(\varrho,u,f))\|_{L^1_t
(\dot B^{d/2-1}_{2,1})}^{h}\\
&\quad+\left\|
a\,u\cdot v\,M^{1/2} +(v\otimes v-\mathrm{Id}):(u\otimes b)M^{1/2}+h(\varrho,u,f)
\right\|_{L^1_t
(\dot{\mathcal B}^{d/2-1}_{2,1})}^{h}\\
&\quad+\int_0^t\Big(
\|\nabla_xu\|_{L^\infty_x}
\big(\|\varrho\|_{\dot B^{d/2}_{2,1}}^{h}
+\|u\|_{\dot B^{d/2-1}_{2,1}}^{h}\big)
+\|\varrho\dive_xu\|_{\dot B^{d/2}_{2,1}}
\Big)d\tau.
\end{aligned}
\end{equation}

We now estimate the nonlinear terms in \eqref{l1ua}. For the
contribution involving $f^\ell$, the high-frequency localization
allows us to gain one derivative. By the decomposition
$(uf)^h=(uf^\ell)^h+(uf^h)^h$, the product estimates in Appendix~A,
\eqref{Xt} and \eqref{e50}, we obtain
\begin{equation}\label{non1}
\begin{aligned}
\|(uf)^h\|_{L^1_t(\dot{\mathcal B}^{d/2-1}_{2,1})}
&\lesssim
\|uf^\ell\|_{L^1_t(\dot{\mathcal B}^{d/2}_{2,1})}^h
+\|uf^h\|_{L^1_t(\dot{\mathcal B}^{d/2-1}_{2,1})}^h\\
&\lesssim
\|f^\ell\|_{L^2_t(\dot{\mathcal B}^{d/2}_{2,1})}
\|u\|_{L^2_t(\dot B^{d/2}_{2,1})}
+\|u\|_{L^2_t(\dot B^{d/2}_{2,1})}
\|f^h\|_{L^2_t(\dot{\mathcal B}^{d/2-1}_{2,1})}\\
&\lesssim \|(\varrho,u,f)\|_{\mathcal E_t}^2.
\end{aligned}
\end{equation}
Similarly, by the velocity-moment estimates and \eqref{e50}, we have
\begin{equation}\label{non2}
\begin{aligned}
&\left\|
a\,u\cdot v\,M^{1/2}
+(v\otimes v-\mathrm{Id}):(u\otimes b)M^{1/2}
\right\|_{L^1_t
(\dot{\mathcal B}^{d/2-1}_{2,1})}^{h}+\|au\|_{L^1_t
(\dot B^{d/2-1}_{2,1})}^{h}\\
&\lesssim
\|u\|_{L^2_t(\dot B^{d/2}_{2,1})}
\Big(
\|(a^\ell,b^\ell)\|_{L^2_t(\dot B^{d/2}_{2,1})}
+\|(a^h,b^h)\|_{L^2_t(\dot B^{d/2-1}_{2,1})}
\Big)
\lesssim \|(\varrho,u,f)\|_{\mathcal E_t}^2.
\end{aligned}
\end{equation}
For the other nonlinear terms, one has
\begin{equation}\label{e69}
\begin{aligned}
&\|\varrho u\|_{L^1_t
(\dot B^{d/2}_{2,1})}^{h}
+\|g(\varrho,u,f)\|_{L^1_t
(\dot B^{d/2-1}_{2,1})}^{h}+\left\| h(\varrho,u,f)
\right\|_{L^1_t
(\dot{\mathcal B}^{d/2-1}_{2,1})}^{h}\\
&\lesssim
\Big(
\|\varrho\|_{L^\infty_t
(\dot B^{d/2}_{2,1})}
+\|\varrho\|_{L^\infty_t
(\dot B^{d/2}_{2,1})}
\|f\|_{L^\infty_t
(\dot{\mathcal B}^{d/2-1}_{2,1})}\Big)
\|u\|_{L^1_t(\dot B^{d/2+1}_{2,1})}\\
&\quad+\|\varrho\|_{L^2_t
(\dot B^{d/2}_{2,1})}
\|u-b\|_{L^2_t
(\dot B^{d/2-1}_{2,1})}
\\
&\quad+\|u\|_{L^2_t
(\dot B^{d/2}_{2,1})}
\|\{\mathbf I-\mathbf P\}f\|_{
L^2_t(\dot{\mathcal B}^{d/2-1}_{2,1,\nu})}
+\|\varrho\|_{L^2_t
(\dot B^{d/2}_{2,1})}^2\\
&\lesssim\|(\varrho,u,f)\|_{\mathcal E_t}^2
+\|(\varrho,u,f)\|_{\mathcal E_t}^3.
\end{aligned}
\end{equation}
Furthermore, in view of Lemma \ref{l11}, we have
\begin{equation}\label{e70}
\begin{aligned}
&\int_0^t\Big(
\|\nabla_xu\|_{L^\infty_x}
\big(\|\varrho\|_{\dot B^{d/2}_{2,1}}^{h}
+\|u\|_{\dot B^{d/2-1}_{2,1}}^{h}\big)
+\|\varrho\dive_xu\|_{\dot B^{d/2}_{2,1}}
\Big)d\tau\\
&\quad\lesssim
\|u\|_{L^1_t(\dot B^{d/2+1}_{2,1})}
\Big(
\|\varrho\|_{L^\infty_t
(\dot B^{d/2}_{2,1})}
+\|u\|_{L^\infty_t
(\dot B^{d/2-1}_{2,1})}\Big)
\lesssim\|(\varrho,u,f)\|_{\mathcal E_t}^2.
\end{aligned}
\end{equation}
By Lemma \ref{l13}, the commutator term also satisfies
\begin{equation}\label{e71}
\sum_{j\geq0}2^{j(d/2-1)}
\int_0^t\mathcal R_j\,d\tau
\lesssim
\|u\|_{L^1_t(\dot B^{d/2+1}_{2,1})}
\Big(
\|u\|_{L^\infty_t
(\dot B^{d/2-1}_{2,1})}
+\|\varrho\|_{L^\infty_t
(\dot B^{d/2}_{2,1})}\Big)
\lesssim\|(\varrho,u,f)\|_{\mathcal E_t}^2.
\end{equation}
Substituting \eqref{non1}--\eqref{e71} into \eqref{l1ua} yields
\begin{equation}\label{l1ua111}
\begin{aligned}
&\|\varrho\|_{\widetilde L^\infty_t
(\dot B^{d/2}_{2,1})}^{h}
+\|u\|_{\widetilde L^\infty_t
(\dot B^{d/2-1}_{2,1})}^{h}
+\|f\|_{\widetilde L^\infty_t
(\dot{\mathcal B}^{d/2-1}_{2,1})}^{h}
+\|\varrho\|_{L^1_t
(\dot B^{d/2}_{2,1})}^{h}\\
&\quad+\|u\|_{L^1_t
(\dot B^{d/2-1}_{2,1})}^{h}
+\|f\|_{L^1_t
(\dot{\mathcal B}^{d/2-1}_{2,1})}^{h}
+\|u-b\|_{\widetilde L^2_t
(\dot B^{d/2-1}_{2,1})}^{h}
+\|
\{\mathbf I-\mathbf P\}f\|_{\widetilde L^2_t
(\dot{\mathcal B}^{d/2-1}_{2,1,\nu})}^{h}\\
&\quad\quad\lesssim \|\varrho_0\|_{\dot B^{d/2}_{2,1}}^{h}
+\|u_0\|_{\dot B^{d/2-1}_{2,1}}^{h}
+\|f_0\|_{\dot{\mathcal B}^{d/2-1}_{2,1}}^{h}+\|(\varrho,u,f)\|_{\mathcal E_t}^2
+\|(\varrho,u,f)\|_{\mathcal E_t}^3.
\end{aligned}
\end{equation}

It remains to estimate the high-frequency
$L^1_t\dot B^{d/2+1}_{2,1}$ norm of $u$. The maximal
regularity estimate in Lemma \ref{heat} applied to $\eqref{m1n}_2$
gives
\begin{align*}
\|u\|_{L^1_t(\dot B^{d/2+1}_{2,1})}^{h}
&\lesssim
\|u_0\|_{\dot B^{d/2-1}_{2,1}}^{h}
+\|\varrho\|_{L^1_t(\dot B^{d/2}_{2,1})}^{h}
+\|(u,b)\|_{L^1_t
(\dot B^{d/2-1}_{2,1})}^{h}\\
&\quad+\|u\cdot\nabla_xu\|_{L^1_t
(\dot B^{d/2-1}_{2,1})}^{h}
+\|au\|_{L^1_t
(\dot B^{d/2-1}_{2,1})}^{h}
+\|g(\varrho,u,f)\|_{L^1_t
(\dot B^{d/2-1}_{2,1})}^{h}\\
&\lesssim
\|\varrho_0\|_{\dot B^{d/2}_{2,1}}^{h}
+\|u_0\|_{\dot B^{d/2-1}_{2,1}}^{h}
+\|f_0\|_{\dot{\mathcal B}^{d/2-1}_{2,1}}^{h}
+\|(\varrho,u,f)\|_{\mathcal E_t}^2
+\|(\varrho,u,f)\|_{\mathcal E_t}^3,
\end{align*}
where the terms involving $\varrho$ and $(u,b)$ are controlled by  \eqref{l1ua111}. This completes the proof
of Lemma \ref{l4}.
\end{proof}

\vspace{2mm}

\begin{proof}[\textbf{Proof of Proposition \ref{p1}}]
Under assumption \eqref{e25}, combining the estimates in
Lemmas \ref{l2} and \ref{l4} and noting that the Chemin--Lerner
norms control the corresponding ordinary Lebesgue--Besov norms in
$\mathcal E_T$, we obtain
$$
\|(\varrho,u,f)\|_{\mathcal E_T}
\le C\|(\varrho_0,u_0,f_0)\|_{\mathcal E_0}
+C\|(\varrho_0,u_0,f_0)\|_{\mathcal E_0}^2
+C\|(\varrho,u,f)\|_{\mathcal E_T}^2
+C\|(\varrho,u,f)\|_{\mathcal E_T}^3.
$$
Choosing $\delta_0^*$ so that
$C(\delta_0^*+(\delta_0^*)^2)\leq\frac12$, we can absorb the last two
terms and proves
\eqref{e26}.
\end{proof}

\vspace{2mm}

\begin{proof}[\normalfont\bfseries Proof of Theorem \ref{t1}]
According to the local well-posedness theory in Appendix B, the Cauchy problem
\eqref{m1n} for the NS--VFP system admits a unique strong solution on
$[0,T_0]$ for some $T_0>0$. Let $T^*>0$ be the maximal existence
time.

By Proposition~\ref{p1}, \eqref{e30} and
\eqref{e44}, there exists a sufficiently small constant
$\delta_*>0$ such that, if $\|(\varrho,u,f)\|_{\mathcal E_T}\leq\delta_*$ for some $T<T^*$, then
\begin{equation}\label{e72}
\|(\varrho,u,f)\|_{\mathcal E_T}
\leq C_0\Big(
\|(\varrho_0,u_0,f_0)\|_{\mathcal E_0}
+\|(\varrho_0,u_0,f_0)\|_{\mathcal E_0}^2
\Big),
\end{equation}
and \eqref{e30} and \eqref{e44} hold on $[0,T]$.

We take $\delta_0>0$ in Theorem~\ref{t1} sufficiently small so
that, whenever $\|(\varrho_0,u_0,f_0)\|_{\mathcal E_0}\leq\delta_0$, we have
\[
2C_0\Big(
\|(\varrho_0,u_0,f_0)\|_{\mathcal E_0}
+\|(\varrho_0,u_0,f_0)\|_{\mathcal E_0}^2
\Big)<\delta_*.
\]
Set
\[
T_1:=\sup\left\{T<T^*:
\sup_{0\leq t\leq T}\|(\varrho,u,f)\|_{\mathcal E_t}
\leq
2C_0\Big(
\|(\varrho_0,u_0,f_0)\|_{\mathcal E_0}
+\|(\varrho_0,u_0,f_0)\|_{\mathcal E_0}^2
\Big)
\right\}.
\]
Time continuity implies that $T_1>0$. Assume $T_1<T^*$. For every $T<T_1$, the estimate in the
definition of $T_1$ is in fact strict using \eqref{e72}, which is a contradiction. Therefore, we have $T_1=T^*$. Choosing $(\varrho,u,f)(t)$ as a new initial datum with $t$ close to $T^*$ and using \eqref{e72} and local existence again, one can extend the solution beyond $T^*$, which implies $T^*=\infty$. Since $\delta_0$ is sufficiently small, \eqref{e6} follows.  The local theory in Appendix B gives the continuity of the
ordinary Besov norms of \(\varrho\) and \(u\). This completes the proof.
\end{proof}

\section{Uniform-in-viscosity regularity estimates}
\label{s6}

In this section, we prove Theorems \ref{t2} and
\ref{t3}, and we assume that $\mu=\lambda=\varepsilon\in (0,1]$ throughout.

\subsection{A priori estimate}

We first state the uniform counterpart of Proposition \ref{p1}.

\begin{proposition}\label{p2}
Let $T>0$ and $0<\varepsilon\leq1$, and let
$(\varrho^\varepsilon,u^\varepsilon,f^\varepsilon)$ be a smooth
solution to the Cauchy problem \eqref{m1n} for the NS--VFP system with
$\mu=\lambda=\varepsilon$ on $[0,T]$.
There exist constants $\delta_1^*>0$ and $C_1>0$, independent of
$T$ and $\varepsilon$, such that if
\begin{equation}\label{e73}
 \|(\varrho^\varepsilon,u^\varepsilon,f^\varepsilon)\|_{\mathcal Z_T}
 \leq\delta_1^*,
\end{equation}
then
\begin{equation}\label{e74}
 \|(\varrho^\varepsilon,u^\varepsilon,f^\varepsilon)\|_{\mathcal Z_T}
 +\varepsilon\|u^\varepsilon\|_
 {L^1_T(\dot B^{d/2+1}_{2,1}\cap
 \dot B^{d/2+2}_{2,1})}
 \leq C_1\Big(
 \|(\varrho_0^\varepsilon,u_0^\varepsilon,f_0^\varepsilon)\|_{\mathcal Z_0}
 +\|(\varrho_0^\varepsilon,u_0^\varepsilon,f_0^\varepsilon)\|_{\mathcal Z_0}^2
 \Big),
\end{equation}
where the $\mathcal Z_0$- and $\mathcal Z_T$-norms are defined by
\eqref{Z0} and \eqref{Zt}, respectively.
\end{proposition}

The proof follows by combining Lemmas
\ref{l6} and
\ref{l8}. The use of interpolation and the embedding
$\dot B^{d/2}_{2,1}\hookrightarrow L^\infty(\mathbb R^d_x)$ shows
$\|\varrho^\varepsilon\|_{L^\infty_{T,x}}
+\|a^\varepsilon\|_{L^\infty_{T,x}}\lesssim\delta_1^*\ll 1$. Thus, we have
\begin{equation}\label{e75}
 \begin{aligned}
&\frac12\leq1+\varrho^\varepsilon\leq\frac32\quad\text{and}\quad
 \frac12\leq1+a^\varepsilon\leq\frac32,\\
 & \frac{P'(1)}{2}\leq \frac{P'(1+\varrho^\varepsilon)}{1+\varrho^\varepsilon}\leq 2 P'(1)
 \quad\text{and}\quad |u^\varepsilon|\leq 1 \quad\text{on }[0,T]\times\mathbb R^d.
 \end{aligned}
\end{equation}

\subsection{Low-frequency estimates}
\label{s7}

Applying $\Theta$ to
\eqref{e29}, using
$\Theta(h(\varrho^\varepsilon,u^\varepsilon,f^\varepsilon))=0$, we obtain
\begin{equation}\label{e76}
\begin{aligned}
&\partial_t\Theta(G^\varepsilon)
+2\Theta(G^\varepsilon)
+2\mathbb D_x V^\varepsilon
+\dive_x\int_{\mathbb R^d_v}
(v\otimes v-\mathrm{Id})\otimes v\,
M^{1/2}G^\varepsilon\,dv\\
&\quad=
-a^\varepsilon
\bigl(u^\varepsilon\otimes b^\varepsilon
+b^\varepsilon\otimes u^\varepsilon\bigr)+
\Bigl(\nabla_xa^\varepsilon
+\dive_x\Theta(\{\mathbf I-\mathbf P\}f^\varepsilon)\Bigr)
\otimes b^\varepsilon\\
&\qquad+
b^\varepsilon\otimes
\Bigl(\nabla_xa^\varepsilon
+\dive_x\Theta(\{\mathbf I-\mathbf P\}f^\varepsilon)\Bigr)
-2\mathbb D_x(a^\varepsilon V^\varepsilon),
\end{aligned}
\end{equation}
where
\begin{equation*}
 \mathbb D_x z:=\frac12\bigl(\nabla_xz+(\nabla_xz)^{\mathsf T}\bigr).
\end{equation*}
The term $2\mathbb D_x V^\varepsilon$ is used below to recover the
dissipation of $\nabla_xV^\varepsilon$, uniformly with respect to $\varepsilon$. Combining this dissipation with the relaxation of $u^\varepsilon-V^\varepsilon$, we recover the uniform-in-$\varepsilon$ dissipation of $u^\varepsilon$.

\begin{lemma}
\label{l5}
Let $0<\varepsilon\leq1$ and $j\leq1$. There exist constants
$c,C>0$, independent of $\varepsilon$ and $j$, and a  
functional $\mathcal L_{\ell,j}^\varepsilon$ such that
\begin{equation}\label{e78}
 \mathcal L_{\ell,j}^\varepsilon\sim
 \|\dot\Delta_j(\varrho^\varepsilon,u^\varepsilon,
 a^\varepsilon,V^\varepsilon)\|_{L^2_x}^2
 +\|\dot\Delta_jG^\varepsilon\|_{L^2_{x,v}}^2,
\end{equation}
and
\begin{equation}\label{e79}
\begin{aligned}
\frac{\mathrm d}{\mathrm dt}\mathcal L_{\ell,j}^\varepsilon
&+c2^{2j}\mathcal L_{\ell,j}^\varepsilon
+c\|\dot\Delta_j(u^\varepsilon-V^\varepsilon)\|_{L^2_x}^2
+c\|\dot\Delta_jG^\varepsilon\|_{L^2_xL^2_{v,\nu}}^2\\
&+c\varepsilon\,2^{2j}\|\dot\Delta_ju^\varepsilon\|_{L^2_x}^2
\lesssim\mathcal R_{\ell,j}^\varepsilon
\sqrt{\mathcal L_{\ell,j}^\varepsilon},
\end{aligned}
\end{equation}
where $\mathcal R_{\ell,j}^\varepsilon$ is defined by
\begin{align}
\mathcal R_{\ell,j}^\varepsilon:={}&
\|\dot\Delta_j\nabla_x(\varrho^\varepsilon u^\varepsilon,
a^\varepsilon V^\varepsilon,b^\varepsilon\otimes b^\varepsilon)\|_{L^2_x}+\|\dot\Delta_j(u^\varepsilon\cdot\nabla_xu^\varepsilon,
g(\varrho^\varepsilon,u^\varepsilon,f^\varepsilon),
a^\varepsilon(u^\varepsilon-V^\varepsilon))\|_{L^2_x}
\nonumber\\
&+\left\|\dot\Delta_j\left(
\frac{a^\varepsilon}{1+a^\varepsilon}
\Big(\nabla_xa^\varepsilon
+\dive_x\Theta(\{\mathbf I-\mathbf P\}f^\varepsilon)\Big),
\frac{V^\varepsilon}{1+a^\varepsilon}\dive_xb^\varepsilon
\right)\right\|_{L^2_x}
\nonumber\\
&+\|\dot\Delta_j h(\varrho^\varepsilon,u^\varepsilon,
f^\varepsilon)\|_{L^2_{x,v}}+\|\dot\Delta_j(a^\varepsilon u^\varepsilon
\otimes b^\varepsilon)\|_{L^2_x}\nonumber\\
&+\left\|\dot\Delta_j\left(
\Big(\nabla_xa^\varepsilon
+\dive_x\Theta(\{\mathbf I-\mathbf P\}f^\varepsilon)\Big)
\otimes b^\varepsilon-\nabla_x(a^\varepsilon V^\varepsilon)
\right)\right\|_{L^2_x}.
\label{e80}
\end{align}
\end{lemma}

\begin{proof}
Applying $\dot\Delta_j$ to the fluid equations in \eqref{m1n}, the
first equation of \eqref{e24},
\eqref{e27} and
\eqref{e29}, and taking the corresponding
$L^2(\mathbb R^d_x)$ and $L^2(\mathbb R^d_x\times\mathbb R^d_v)$ inner products with
$\dot\Delta_j\varrho^\varepsilon$,
$\dot\Delta_ju^\varepsilon$,
$\dot\Delta_ja^\varepsilon$,
$\dot\Delta_jV^\varepsilon$ and
$\dot\Delta_jG^\varepsilon$, respectively, we obtain
\begin{equation}\label{e81}
\begin{aligned}
&\frac12\frac{\mathrm d}{\mathrm dt}
 \left(
 \|\dot\Delta_j(\varrho^\varepsilon,u^\varepsilon,
 a^\varepsilon,V^\varepsilon)\|_{L^2_x}^2
 +\|\dot\Delta_jG^\varepsilon\|_{L^2_{x,v}}^2
 \right)\\
&\quad
 +\varepsilon\left(
 \|\nabla_x\dot\Delta_ju^\varepsilon\|_{L^2_x}^2
 +2\|\dive_x\dot\Delta_ju^\varepsilon\|_{L^2_x}^2
 \right)
 +\|\dot\Delta_j(u^\varepsilon-V^\varepsilon)\|_{L^2_x}^2
 +\lambda_0
 \|\dot\Delta_jG^\varepsilon
 \|_{L^2_xL^2_{v,\nu}}^2\\
&\qquad\lesssim
 \mathcal R_{\ell,j}^\varepsilon
 \left(
 \|\dot\Delta_j(\varrho^\varepsilon,u^\varepsilon,
 a^\varepsilon,V^\varepsilon)\|_{L^2_x}
 +\|\dot\Delta_jG^\varepsilon\|_{L^2_{x,v}}
 \right).
\end{aligned}
\end{equation}

To recover the dissipation of
$\dot\Delta_j\varrho^\varepsilon$, we take the $L^2(\mathbb R^d_x)$ inner product
of the localized momentum equation with
$\nabla_x\dot\Delta_j\varrho^\varepsilon$ and use the localized
continuity equation to treat the time derivative to obtain
\begin{align}
&\frac{\mathrm d}{\mathrm dt}\bigl(\dot\Delta_ju^\varepsilon,\, \nabla_x\dot\Delta_j\varrho^\varepsilon\bigr)_{L^2_x}+\|\nabla_x\dot\Delta_j\varrho^\varepsilon\|_{L^2_x}^2\nonumber\\
&\quad\leq \|\dive_x\dot\Delta_ju^\varepsilon\|_{L^2_x}^2-\bigl(\dot\Delta_j(u^\varepsilon-V^\varepsilon),\, \nabla_x\dot\Delta_j\varrho^\varepsilon\bigr)_{L^2_x}\nonumber\\
&\qquad+\bigl(\varepsilon\bigl(\Delta_x\dot\Delta_ju^\varepsilon+2\nabla_x\dive_x\dot\Delta_ju^\varepsilon\bigr),\, \nabla_x\dot\Delta_j\varrho^\varepsilon\bigr)_{L^2_x}\nonumber\\
&\qquad+C\mathcal R_{\ell,j}^\varepsilon\Big(\|\dot\Delta_j(\varrho^\varepsilon,u^\varepsilon,a^\varepsilon,V^\varepsilon)\|_{L^2_x}+\|\dot\Delta_jG^\varepsilon\|_{L^2_{x,v}}\Big).
\label{e82}
\end{align}

By applying $\dot\Delta_j$ to the first equation of
\eqref{e24} and to \eqref{e27}, taking the
$L^2(\mathbb R^d_x)$ inner product of the localized equation for $V^\varepsilon$
with $\nabla_x\dot\Delta_ja^\varepsilon$, we get
\begin{align}
&\frac{\mathrm d}{\mathrm dt}
 \bigl(\dot\Delta_jV^\varepsilon,
 \nabla_x\dot\Delta_ja^\varepsilon\bigr)_{L^2_x}
 +\|\nabla_x\dot\Delta_ja^\varepsilon\|_{L^2_x}^2
 \nonumber\\
&\quad\leq
 \|\dive_x\dot\Delta_jV^\varepsilon\|_{L^2_x}^2
 +\bigl(\dot\Delta_j(u^\varepsilon-V^\varepsilon),
 \nabla_x\dot\Delta_ja^\varepsilon\bigr)_{L^2_x}
 \nonumber\\
&\qquad
 -\bigl(\dive_x\Theta(\dot\Delta_jG^\varepsilon),
 \nabla_x\dot\Delta_ja^\varepsilon\bigr)_{L^2_x}
 +C\mathcal R_{\ell,j}^\varepsilon
 \|\dot\Delta_j(a^\varepsilon,V^\varepsilon)\|_{L^2_x}.
\label{e83}
\end{align}

We now recover the dissipation of
$\nabla_x\dot\Delta_jV^\varepsilon$. Applying $\dot\Delta_j$ to
\eqref{e76}, taking the $L^2(\mathbb R^d_x)$ inner
product with $\mathbb D_x\dot\Delta_jV^\varepsilon$, and using
\eqref{e27} to treat
$\partial_t\dot\Delta_jV^\varepsilon$, we integrate by parts and bound
the nonlinear source pairings using \eqref{e80}
and $j\leq1$ to obtain
\begin{align}
&\frac{\mathrm d}{\mathrm dt}\bigl(\Theta(\dot\Delta_jG^\varepsilon),\, \mathbb D_x\dot\Delta_jV^\varepsilon\bigr)_{L^2_x}+2\|\mathbb D_x\dot\Delta_jV^\varepsilon\|_{L^2_x}^2\nonumber\\
&\quad\leq \bigl(\dive_x\Theta(\dot\Delta_jG^\varepsilon),\, \nabla_x\dot\Delta_ja^\varepsilon\bigr)_{L^2_x}-\bigl(\dive_x\Theta(\dot\Delta_jG^\varepsilon),\, \dot\Delta_j(u^\varepsilon-V^\varepsilon)\bigr)_{L^2_x}\nonumber\\
&\qquad-2\bigl(\Theta(\dot\Delta_jG^\varepsilon),\, \mathbb D_x\dot\Delta_jV^\varepsilon\bigr)_{L^2_x}+\|\dive_x\Theta(\dot\Delta_jG^\varepsilon)\|_{L^2_x}^2\nonumber\\
&\qquad-\bigl(\Theta(v\cdot\nabla_x\dot\Delta_jG^\varepsilon),\, \mathbb D_x\dot\Delta_jV^\varepsilon\bigr)_{L^2_x}+C\mathcal R_{\ell,j}^\varepsilon\bigl(\|\dot\Delta_jV^\varepsilon\|_{L^2_x}+\|\dot\Delta_jG^\varepsilon\|_{L^2_{x,v}}\bigr).
\label{e84}
\end{align}

The two acoustic stress pairings in
\eqref{e83} and
\eqref{e84} cancel upon addition.
We also use the whole-space identity
\begin{equation}\label{e85}
2\|\mathbb D_x\dot\Delta_jV^\varepsilon\|_{L^2_x}^2
=\|\nabla_x\dot\Delta_jV^\varepsilon\|_{L^2_x}^2
+\|\dive_x\dot\Delta_jV^\varepsilon\|_{L^2_x}^2,
\end{equation}
which implies
$\|\dive_x\dot\Delta_jV^\varepsilon\|_2^2
\leq\|\mathbb D_x\dot\Delta_jV^\varepsilon\|_2^2$.

We define
\begin{equation}\label{e86}
\begin{aligned}
\mathcal L_{\ell,j}^\varepsilon:={}&
 \frac12
 \|\dot\Delta_j(\varrho^\varepsilon,u^\varepsilon,
 a^\varepsilon,V^\varepsilon)\|_{L^2_x}^2
 +\frac12\|\dot\Delta_jG^\varepsilon\|_{L^2_{x,v}}^2+\eta_1
 \bigl(\dot\Delta_ju^\varepsilon,
 \nabla_x\dot\Delta_j\varrho^\varepsilon\bigr)_{L^2_x}\\
&+\eta_2\left(
 \bigl(\dot\Delta_jV^\varepsilon,
 \nabla_x\dot\Delta_ja^\varepsilon\bigr)_{L^2_x}
 +\bigl(\Theta(\dot\Delta_jG^\varepsilon),
 \mathbb D_x\dot\Delta_jV^\varepsilon\bigr)_{L^2_x}
 \right).
\end{aligned}
\end{equation}
For $j\leq1$, Bernstein's inequality and the Gaussian moment bounds
show that the cross terms in \eqref{e86} are bounded by the
quadratic part of $\mathcal L_{\ell,j}^\varepsilon$. Hence, choosing
$\eta_1,\eta_2>0$ sufficiently small gives
\eqref{e78}.

To proceed, we first choose $\eta_2>0$ sufficiently small and then
take $0<\eta_1\ll\eta_2$. We add $\eta_1$ times \eqref{e82} and
$\eta_2$ times each of \eqref{e83} and \eqref{e84} to \eqref{e81}.
By Young's inequality, \eqref{e85}, and Bernstein's inequality for
$j\leq1$, all the terms arising from the linear couplings are absorbed
by the dissipative terms. Consequently,
\begin{align}
\frac{\mathrm d}{\mathrm dt}\mathcal L_{\ell,j}^\varepsilon
&+\frac{\eta_1}{2}
 \|\nabla_x\dot\Delta_j\varrho^\varepsilon\|_{L^2_x}^2
+\frac{\eta_2}{2}
 \|\nabla_x\dot\Delta_ja^\varepsilon\|_{L^2_x}^2
+\frac{\eta_2}{2}
 \|\mathbb D_x\dot\Delta_jV^\varepsilon\|_{L^2_x}^2
\nonumber\\
&\quad+\frac12
 \|\dot\Delta_j(u^\varepsilon-V^\varepsilon)\|_{L^2_x}^2
+\frac{\lambda_0}{2}
 \|\dot\Delta_jG^\varepsilon\|_{L^2_xL^2_{v,\nu}}^2+\frac{\varepsilon}{2}
 \|\nabla_x\dot\Delta_ju^\varepsilon\|_{L^2_x}^2
\lesssim
\mathcal R_{\ell,j}^\varepsilon
\sqrt{\mathcal L_{\ell,j}^\varepsilon}.
\label{e87}
\end{align}
Finally, Bernstein's inequality, the relative-velocity dissipation
and the microscopic coercivity imply
\begin{align*}
2^{2j}\mathcal L_{\ell,j}^\varepsilon
\lesssim{}&
\|\nabla_x\dot\Delta_j
(\varrho^\varepsilon,a^\varepsilon,V^\varepsilon)\|_{L^2_x}^2+\|\dot\Delta_j(u^\varepsilon-V^\varepsilon)\|_{L^2_x}^2
+\|\dot\Delta_jG^\varepsilon\|_{L^2_xL^2_{v,\nu}}^2,
\end{align*}
and $\varepsilon
\|\nabla_x\dot\Delta_ju^\varepsilon\|_{L^2_x}^2
\gtrsim
\varepsilon 2^{2j}
\|\dot\Delta_ju^\varepsilon\|_{L^2_x}^2$. Together with \eqref{e87}, these estimates prove \eqref{e79}, with
constants independent of $\varepsilon$ and $j$.
\end{proof}

We next derive the low-frequency estimates uniformly in $\varepsilon$. 
\begin{lemma}
\label{l6}
Let $0<\varepsilon\leq1$ and $T>0$. Suppose that
\eqref{e73} holds for a sufficiently small
constant $\delta_1^*>0$ independent of $T$ and $\varepsilon$.
Then, for every $t\in[0,T]$, we have
\begin{equation}\label{e88}
\begin{aligned}
&\|(\varrho^{\varepsilon,\ell},u^{\varepsilon,\ell})\|_
{\widetilde L^\infty_t(\dot B^{d/2-1}_{2,1})}
+\|f^{\varepsilon,\ell}\|_
{\widetilde L^\infty_t(\dot{\mathcal B}^{d/2-1}_{2,1})}
+\|(\varrho^{\varepsilon,\ell},u^{\varepsilon,\ell})\|_
{L^1_t(\dot B^{d/2+1}_{2,1})}
+\|f^{\varepsilon,\ell}\|_
{L^1_t(\dot{\mathcal B}^{d/2+1}_{2,1})}\\
&\quad
+\|(u^\varepsilon-b^\varepsilon)^\ell\|_
{\widetilde L^2_t(\dot B^{d/2-1}_{2,1})}
+\left\|\{\mathbf I-\mathbf P\}
f^{\varepsilon,\ell}\right\|_
{\widetilde L^2_t(\dot{\mathcal B}^{d/2-1}_{2,1,\nu})}
+\varepsilon\|u^{\varepsilon,\ell}\|_
{L^1_t(\dot B^{d/2+2}_{2,1})}\\
&\qquad\leq
C\|(\varrho_0^\varepsilon,u_0^\varepsilon,f_0^\varepsilon)\|_
{\mathcal Z_0}
+C\|(\varrho_0^\varepsilon,u_0^\varepsilon,f_0^\varepsilon)\|_
{\mathcal Z_0}^2
+C\|(\varrho^\varepsilon,u^\varepsilon,f^\varepsilon)\|_
{\mathcal Z_t}^2
+C\|(\varrho^\varepsilon,u^\varepsilon,f^\varepsilon)\|_
{\mathcal Z_t}^3\\
&\qquad\quad
+C\|(\varrho^\varepsilon,u^\varepsilon,f^\varepsilon)\|_
{\mathcal Z_t}
\varepsilon\|u^\varepsilon\|_
{L^1_t(\dot B^{d/2+2}_{2,1})},
\end{aligned}
\end{equation}
where the constant $C>0$ is independent of $\varepsilon$ and $t$.
\end{lemma}

\begin{proof}
For $j\leq1$, arguing as in \eqref{e46}--\eqref{e47} and using
\eqref{e79}, we obtain
\begin{equation}\label{e89}
\begin{aligned}
&\|\dot\Delta_j(\varrho^\varepsilon,u^\varepsilon,
a^\varepsilon,V^\varepsilon)\|_{L^\infty_tL^2_x}
+\|\dot\Delta_jG^\varepsilon\|_{L^\infty_tL^2_{x,v}}\\
&\quad
+2^{2j}\|\dot\Delta_j(\varrho^\varepsilon,u^\varepsilon,
a^\varepsilon,V^\varepsilon)\|_{L^1_tL^2_x}
+2^{2j}\|\dot\Delta_jG^\varepsilon\|_{L^1_tL^2_{x,v}}\\
&\quad
+\|\dot\Delta_j(u^\varepsilon-V^\varepsilon)\|_{L^2_tL^2_x}
+\|\dot\Delta_jG^\varepsilon
\|_{L^2_tL^2_xL^2_{v,\nu}}\\
&\qquad\lesssim
\|\dot\Delta_j(\varrho_0^\varepsilon,u_0^\varepsilon,
a_0^\varepsilon,V_0^\varepsilon)\|_{L^2_x}
+\|\dot\Delta_jG_0^\varepsilon\|_{L^2_{x,v}}
+\|\mathcal R_{\ell,j}^\varepsilon\|_{L^1_t}.
\end{aligned}
\end{equation}
Then, multiplying \eqref{e89} by
$2^{j(d/2-1)}$ and summing over $j\leq1$, we have
\begin{align}
&\|(\varrho^\varepsilon,u^\varepsilon,a^\varepsilon,
V^\varepsilon)^\ell\|_
{\widetilde L^\infty_t(\dot B^{d/2-1}_{2,1})}
+\|G^{\varepsilon,\ell}\|_
{\widetilde L^\infty_t(\dot{\mathcal B}^{d/2-1}_{2,1})}
\nonumber\\
&\quad
+\|(\varrho^\varepsilon,u^\varepsilon,a^\varepsilon,
V^\varepsilon)^\ell\|_
{L^1_t(\dot B^{d/2+1}_{2,1})}
+\|G^{\varepsilon,\ell}\|_
{L^1_t(\dot{\mathcal B}^{d/2+1}_{2,1})}
\nonumber\\
&\quad
+\|(u^\varepsilon-V^\varepsilon)^\ell\|_
{\widetilde L^2_t(\dot B^{d/2-1}_{2,1})}
+\|G^{\varepsilon,\ell}\|_
{\widetilde L^2_t(\dot{\mathcal B}^{d/2-1}_{2,1,\nu})}
\nonumber\\
&\qquad\lesssim
\|(\varrho_0^\varepsilon,u_0^\varepsilon,a_0^\varepsilon,
V_0^\varepsilon)^\ell\|_{\dot B^{d/2-1}_{2,1}}
+\|G_0^{\varepsilon,\ell}\|_
{\dot{\mathcal B}^{d/2-1}_{2,1}}
+\sum_{j\leq1}2^{j(d/2-1)}
\|\mathcal R_{\ell,j}^\varepsilon\|_{L^1_t}.
\label{e90}
\end{align}

The next step is to estimate $\mathcal{R}^\varepsilon_{\ell,j}$ defined by \eqref{e80}. The definitions of $V^\varepsilon$ and $G^\varepsilon$ give
\[
 b^\varepsilon=(1+a^\varepsilon)V^\varepsilon\quad\text{and}\quad 
 \{\mathbf I-\mathbf P\}f^\varepsilon
 =G^\varepsilon+\frac12(v\otimes v-\mathrm{Id}):
 (b^\varepsilon\otimes b^\varepsilon)M^{1/2},
\]
and hence
$u^\varepsilon-b^\varepsilon
=u^\varepsilon-V^\varepsilon-a^\varepsilon V^\varepsilon$.
The combination of product and composition estimates (Lemmas \ref{l11} and \ref{l12}) yields
\begin{align}
&\|(a^\varepsilon V^\varepsilon,
b^\varepsilon\otimes b^\varepsilon)^\ell\|_
{\widetilde L^\infty_t(\dot B^{d/2-1}_{2,1})}
+\|(a^\varepsilon V^\varepsilon,
b^\varepsilon\otimes b^\varepsilon)^\ell\|_
{L^1_t(\dot B^{d/2+1}_{2,1})}
\nonumber\\
&\quad
+\|(a^\varepsilon V^\varepsilon,
b^\varepsilon\otimes b^\varepsilon)^\ell\|_
{\widetilde L^2_t(\dot B^{d/2-1}_{2,1})}
\lesssim
\|(\varrho^\varepsilon,u^\varepsilon,
f^\varepsilon)\|_{\mathcal Z_t}^2.
\label{e91}
\end{align}
A similar argument shows
\begin{align}
&\|(a_0^\varepsilon,
V_0^\varepsilon)^\ell\|_{\dot B^{d/2-1}_{2,1}}
+\|G_0^{\varepsilon,\ell}\|_
{\dot{\mathcal B}^{d/2-1}_{2,1}}\lesssim
\|(\varrho_0^\varepsilon,u_0^\varepsilon,f_0^\varepsilon)\|_
{\mathcal Z_0}
+\|(\varrho_0^\varepsilon,u_0^\varepsilon,f_0^\varepsilon)\|_
{\mathcal Z_0}^2.
\label{e92}
\end{align}
Moreover, interpolation between the
$L^\infty_t$ and $L^1_t$ components of $\mathcal Z_t$, together with
the moment bounds and the definition of $V^\varepsilon$, gives
\[
 \|(\varrho^\varepsilon,u^\varepsilon,a^\varepsilon,
 b^\varepsilon,V^\varepsilon)\|_
 {L^2_t(\dot B^{d/2}_{2,1})}
 +\|f^\varepsilon\|_
 {L^2_t(\dot{\mathcal B}^{d/2}_{2,1})}
 \lesssim
 \|(\varrho^\varepsilon,u^\varepsilon,
 f^\varepsilon)\|_{\mathcal Z_t}.
\]
The terms with one spatial derivative in
\eqref{e80} are therefore bounded by products of
two $L^2_t\dot B^{d/2}_{2,1}$ terms. The terms
$a^\varepsilon(u^\varepsilon-V^\varepsilon)$ and
$h(\varrho^\varepsilon,u^\varepsilon,f^\varepsilon)$ are controlled
by the relative-velocity and microscopic $L^2_t$ norms, using
$u^\varepsilon-V^\varepsilon
=u^\varepsilon-b^\varepsilon+a^\varepsilon V^\varepsilon$.
The quotient terms are handled by the composition estimates and
\eqref{e75}, while the terms containing
$a^\varepsilon u^\varepsilon\otimes b^\varepsilon$ are cubic. The only additional contribution comes from the viscous part of
$g(\varrho^\varepsilon,u^\varepsilon,f^\varepsilon)$. By \eqref{e4}, \eqref{e75}, \eqref{E35}, and the product and
composition estimates, we get
\begin{equation*}
\begin{aligned}
\|g(\varrho^\varepsilon,u^\varepsilon,f^\varepsilon)\|_
{L^1_t(\dot B^{d/2-1}_{2,1})}^{\ell}&\lesssim
\|\varrho^\varepsilon\|_{L^2_t(\dot B^{d/2}_{2,1})}^2
+\|\varrho^\varepsilon\|_{L^2_t(\dot B^{d/2}_{2,1})}
 \|u^\varepsilon-b^\varepsilon\|_
 {L^2_t(\dot B^{d/2-1}_{2,1})}\\
&\qquad+
\|\varrho^\varepsilon\|_{L^\infty_t(\dot B^{d/2}_{2,1})}
\|f^\varepsilon\|_
{L^\infty_t(\dot{\mathcal B}^{d/2-1}_{2,1})}
\|u^\varepsilon\|_{L^1_t(\dot B^{d/2+1}_{2,1})}\\
&\qquad+
\|\varrho^\varepsilon\|_{L^\infty_t(\dot B^{d/2}_{2,1})}
\,\varepsilon
\|u^\varepsilon\|_{L^1_t(\dot B^{d/2+2}_{2,1})}\\
&\quad\lesssim
\|(\varrho^\varepsilon,u^\varepsilon,f^\varepsilon)\|_{\mathcal Z_t}^2
+\|(\varrho^\varepsilon,u^\varepsilon,f^\varepsilon)\|_{\mathcal Z_t}^3\\
&\qquad+
\|(\varrho^\varepsilon,u^\varepsilon,f^\varepsilon)\|_{\mathcal Z_t}
\,\varepsilon
\|u^\varepsilon\|_{L^1_t(\dot B^{d/2+2}_{2,1})}.
\end{aligned}
\end{equation*}
Consequently, we have
\begin{equation}
\begin{aligned}
&\sum_{j\leq1}2^{j(d/2-1)}
\|\mathcal R_{\ell,j}^\varepsilon\|_{L^1_t}\\
&\quad\lesssim
\|(\varrho^\varepsilon,u^\varepsilon,
f^\varepsilon)\|_{\mathcal Z_t}^2
+\|(\varrho^\varepsilon,u^\varepsilon,
f^\varepsilon)\|_{\mathcal Z_t}^3
+\|(\varrho^\varepsilon,u^\varepsilon,
f^\varepsilon)\|_{\mathcal Z_t}
\varepsilon\|u^\varepsilon\|_
{L^1_t(\dot B^{d/2+2}_{2,1})}.
\label{e94}
\end{aligned}
\end{equation}

Combining \eqref{e90},
\eqref{e91},
\eqref{e92}, \eqref{e94} and $\|\cdot^\ell\|_{\dot{B}^s_{2,1}}\lesssim \|\cdot \|_{\dot{B}^s_{2,1}}^\ell$, we control all terms in
\eqref{e88} except the last viscous term.
Since $0<\varepsilon\leq1$, one has $ \varepsilon\|u^{\varepsilon,\ell}\|_
 {L^1_t(\dot B^{d/2+2}_{2,1})}
 \lesssim
 \|u^{\varepsilon,\ell}\|_
 {L^1_t(\dot B^{d/2+1}_{2,1})}$.
 This completes the proof of 
\eqref{e88}.
\end{proof}

\subsection{Uniform high-frequency estimates}
\label{s8}

We first derive a Lyapunov inequality for each high-frequency block.
The summation over $j\geq0$ and the weighted parabolic estimate will
be carried out in Lemma \ref{l8}.

\begin{lemma}
\label{l7}
Let $0<\varepsilon\leq1$, $T>0$ and $j\geq0$. Assume that
$(\varrho^\varepsilon,u^\varepsilon,f^\varepsilon)$ is a smooth
solution to \eqref{m1n} on $[0,T]$ and that
\eqref{e73} holds for a sufficiently small
constant $\delta_1^*>0$. There exist constants $c,C>0$, independent
of $\varepsilon$ and $j$, and a functional
$\mathcal L_{h,j}^\varepsilon$ such that
\begin{equation}\label{e95}
 \mathcal L_{h,j}^\varepsilon
 \sim
 \|\dot\Delta_j(\varrho^\varepsilon,u^\varepsilon)\|_{L^2_x}^2
 +\|\dot\Delta_jf^\varepsilon\|_{L^2_{x,v}}^2,
\end{equation}
and
\begin{align}
\frac{\mathrm d}{\mathrm dt}\mathcal L_{h,j}^\varepsilon
&+c\mathcal L_{h,j}^\varepsilon
+c\varepsilon
 \|\nabla_x\dot\Delta_ju^\varepsilon\|_{L^2_x}^2
+c\|\dot\Delta_j(u^\varepsilon-b^\varepsilon)\|_{L^2_x}^2
 \nonumber\\
&+c\left\|\dot\Delta_j
 \{\mathbf I-\mathbf P\}f^\varepsilon\right\|_{L^2_xL^2_{v,\nu}}^2
 \lesssim
 \mathcal R_{h,j}^\varepsilon
 \sqrt{\mathcal L_{h,j}^\varepsilon},
\label{e96}
\end{align}
with
\begin{equation}\label{e97}
\begin{aligned}
\mathcal R_{h,j}^\varepsilon:={}&
 \|[\dot\Delta_j,u^\varepsilon]\cdot
 \nabla_x\varrho^\varepsilon\|_{L^2_x}
 +\|[\dot\Delta_j,u^\varepsilon]\cdot
 \nabla_xu^\varepsilon\|_{L^2_x}\\
&+\|[\dot\Delta_j,\varrho^\varepsilon]
 \dive_xu^\varepsilon\|_{L^2_x}
 +\left\|\left[\dot\Delta_j,
 \frac{P'(1+\varrho^\varepsilon)}
 {1+\varrho^\varepsilon}\right]
 \nabla_x\varrho^\varepsilon\right\|_{L^2_x}\\
&+\varepsilon\left\|\left[\dot\Delta_j,
 \frac{1}{1+\varrho^\varepsilon}\right]
 \bigl(\Delta_xu^\varepsilon
 +2\nabla_x\dive_xu^\varepsilon\bigr)\right\|_{L^2_x}\\
&+\left\|\dot\Delta_j\left(
 \frac{a^\varepsilon u^\varepsilon
 +\varrho^\varepsilon(b^\varepsilon-u^\varepsilon)}
 {1+\varrho^\varepsilon}\right)\right\|_{L^2_x}
 +\|\dot\Delta_j(a^\varepsilon u^\varepsilon)\|_{L^2_x}\\
&+\|\dot\Delta_j(u^\varepsilon\otimes
 b^\varepsilon)\|_{L^2_x}
 +\|\dot\Delta_j
 h(\varrho^\varepsilon,u^\varepsilon,f^\varepsilon)
 \|_{L^2_{x,v}}\\
&+\Big(
 \|\varrho^\varepsilon\|_{L^\infty_x}
 +\|(\nabla_x\varrho^\varepsilon,
 \nabla_xu^\varepsilon)\|_{L^\infty_x}\Big)
 \|\dot\Delta_j(\varrho^\varepsilon,u^\varepsilon)\|_{L^2_x}.
\end{aligned}
\end{equation}
\end{lemma}

\begin{proof}
Applying $\dot\Delta_j$ to \eqref{m1n}, we obtain
\begin{equation}\label{e98}
\left\{
\begin{aligned}
&\partial_t\dot\Delta_j\varrho^\varepsilon
 +u^\varepsilon\cdot\nabla_x\dot\Delta_j\varrho^\varepsilon
 +(1+\varrho^\varepsilon)
 \dive_x\dot\Delta_ju^\varepsilon\\
&\quad=
 -[\dot\Delta_j,u^\varepsilon]\cdot
 \nabla_x\varrho^\varepsilon
 -[\dot\Delta_j,\varrho^\varepsilon]
 \dive_xu^\varepsilon,\\
&\partial_t\dot\Delta_ju^\varepsilon
 +u^\varepsilon\cdot\nabla_x\dot\Delta_ju^\varepsilon
 +\frac{P'(1+\varrho^\varepsilon)}
 {1+\varrho^\varepsilon}
 \nabla_x\dot\Delta_j\varrho^\varepsilon\\
&\quad
 \quad -\frac{\varepsilon}{1+\varrho^\varepsilon}
 \bigl(\Delta_x\dot\Delta_ju^\varepsilon
 +2\nabla_x\dive_x\dot\Delta_ju^\varepsilon\bigr)
 +\dot\Delta_j(u^\varepsilon-b^\varepsilon)\\
&\quad=
 -[\dot\Delta_j,u^\varepsilon]\cdot\nabla_xu^\varepsilon
 -\left[\dot\Delta_j,
 \frac{P'(1+\varrho^\varepsilon)}
 {1+\varrho^\varepsilon}\right]
 \nabla_x\varrho^\varepsilon\\
&\qquad
 +\varepsilon\left[\dot\Delta_j,
 \frac{1}{1+\varrho^\varepsilon}\right]
 \bigl(\Delta_xu^\varepsilon
 +2\nabla_x\dive_xu^\varepsilon\bigr)
 -\dot\Delta_j\left(
 \frac{a^\varepsilon u^\varepsilon
 +\varrho^\varepsilon(b^\varepsilon-u^\varepsilon)}
 {1+\varrho^\varepsilon}\right),
\end{aligned}
\right.
\end{equation}
and
\begin{align}
&\partial_t\dot\Delta_jf^\varepsilon
 +v\cdot\nabla_x\dot\Delta_jf^\varepsilon
 -\mathcal L\dot\Delta_jf^\varepsilon
 -\dot\Delta_ju^\varepsilon\cdot vM^{1/2}
 \nonumber\\
&\quad=
 \dot\Delta_j(a^\varepsilon u^\varepsilon)\cdot vM^{1/2}
 +(v\otimes v-\mathrm{Id}):
 \dot\Delta_j(u^\varepsilon\otimes b^\varepsilon)M^{1/2}
 +\dot\Delta_j
 h(\varrho^\varepsilon,u^\varepsilon,f^\varepsilon).
\label{e99}
\end{align}

We first derive the basic energy estimate. By
\eqref{e75}, the coefficients in the Friedrichs
symmetrizer
\[
\begin{pmatrix}
\dfrac{P'(1+\varrho^\varepsilon)}
      {1+\varrho^\varepsilon}
& 0\\[2mm]
0
& (1+\varrho^\varepsilon)\mathrm{Id}
\end{pmatrix}
\]
are bounded above and below by positive constants. Taking the
$L^2(\mathbb R^d_x)$ inner products of the two equations in
\eqref{e98} with
$\frac{P'(1+\varrho^\varepsilon)}
{1+\varrho^\varepsilon}\dot\Delta_j\varrho^\varepsilon$ and
$(1+\varrho^\varepsilon)\dot\Delta_ju^\varepsilon$, respectively,
and taking the $L^2(\mathbb R^d_x\times\mathbb R^d_v)$ inner product of
\eqref{e99} with
$\dot\Delta_jf^\varepsilon$, we obtain
\begin{align}
&\frac{\mathrm d}{\mathrm dt}\left(
 \frac12\int_{\mathbb R^d_x}
 \frac{P'(1+\varrho^\varepsilon)}{1+\varrho^\varepsilon}
 |\dot\Delta_j\varrho^\varepsilon|^2\,\mathrm dx
 +\frac12\int_{\mathbb R^d_x}
 (1+\varrho^\varepsilon)|\dot\Delta_ju^\varepsilon|^2\,\mathrm dx
 +\frac12\|\dot\Delta_jf^\varepsilon\|_{L^2_{x,v}}^2
 \right)
 \nonumber\\
&\quad+\varepsilon\Big(
 \|\nabla_x\dot\Delta_ju^\varepsilon\|_{L^2_x}^2
 +2\|\dive_x\dot\Delta_ju^\varepsilon\|_{L^2_x}^2\Big)
 +\|\dot\Delta_j(u^\varepsilon-b^\varepsilon)\|_{L^2_x}^2
 +\lambda_0\left\|\dot\Delta_j
 \{\mathbf I-\mathbf P\}f^\varepsilon\right\|_{L^2_xL^2_{v,\nu}}^2
 \nonumber\\
&\leq
 -\left(
 [\dot\Delta_j,u^\varepsilon]\cdot\nabla_x\varrho^\varepsilon
 +[\dot\Delta_j,\varrho^\varepsilon]\dive_xu^\varepsilon,\,
 \frac{P'(1+\varrho^\varepsilon)}
 {1+\varrho^\varepsilon}\dot\Delta_j\varrho^\varepsilon
 \right)_{L^2_x}
 \nonumber\\
&\quad+\left(
 -[\dot\Delta_j,u^\varepsilon]\cdot\nabla_xu^\varepsilon
 -\left[\dot\Delta_j,
 \frac{P'(1+\varrho^\varepsilon)}
 {1+\varrho^\varepsilon}\right]\nabla_x\varrho^\varepsilon
 +\varepsilon\left[\dot\Delta_j,
 \frac1{1+\varrho^\varepsilon}\right]
 \bigl(\Delta_xu^\varepsilon
 +2\nabla_x\dive_xu^\varepsilon\bigr)
 \right.
 \nonumber\\
&\hspace{26mm}\left.
 -\dot\Delta_j\left(
 \frac{a^\varepsilon u^\varepsilon
 +\varrho^\varepsilon(b^\varepsilon-u^\varepsilon)}
 {1+\varrho^\varepsilon}\right),\,
 (1+\varrho^\varepsilon)\dot\Delta_ju^\varepsilon
 \right)_{L^2_x}
 \nonumber\\
&\quad+\left(
 \dot\Delta_j(a^\varepsilon u^\varepsilon)\cdot vM^{1/2}
 +(v\otimes v-\mathrm{Id}):
 \dot\Delta_j(u^\varepsilon\otimes b^\varepsilon)M^{1/2}
 +\dot\Delta_jh(\varrho^\varepsilon,u^\varepsilon,f^\varepsilon),\,
 \dot\Delta_jf^\varepsilon
 \right)_{L^2_{x,v}}
 \nonumber\\
&\quad+\frac12\int_{\mathbb R^d_x}
 \left(
 \partial_t\left(
 \frac{P'(1+\varrho^\varepsilon)}
 {1+\varrho^\varepsilon}\right)
 +\dive_x\left(
 u^\varepsilon\frac{P'(1+\varrho^\varepsilon)}
 {1+\varrho^\varepsilon}\right)
 \right)
 |\dot\Delta_j\varrho^\varepsilon|^2\,\mathrm dx
 \nonumber\\
&\quad+\int_{\mathbb R^d_x}
 \nabla_xP'(1+\varrho^\varepsilon)\cdot
 \dot\Delta_ju^\varepsilon\,
 \dot\Delta_j\varrho^\varepsilon\,\mathrm dx
 -\int_{\mathbb R^d_x}
 \varrho^\varepsilon\dot\Delta_ju^\varepsilon\cdot
 \dot\Delta_j(u^\varepsilon-b^\varepsilon)\,\mathrm dx.
\label{e100}
\end{align}
Indeed, the continuity equation gives $ \partial_t\varrho^\varepsilon
 =-\dive_x\bigl((1+\varrho^\varepsilon)u^\varepsilon\bigr)$, and hence \eqref{e75} implies
\begin{align}
&\left\|
 \partial_t\left(
 \frac{P'(1+\varrho^\varepsilon)}
 {1+\varrho^\varepsilon}\right)
 +\dive_x\left(
 u^\varepsilon\frac{P'(1+\varrho^\varepsilon)}
 {1+\varrho^\varepsilon}\right)
 \right\|_{L^\infty_x}
 +\|\nabla_xP'(1+\varrho^\varepsilon)\|_{L^\infty_x}
 \nonumber\\
&\qquad\lesssim
 \|(\nabla_x\varrho^\varepsilon,
 \nabla_xu^\varepsilon)\|_{L^\infty_x}.
\label{e101}
\end{align}
By Cauchy--Schwarz, the moment bounds and
\eqref{e97}, the last integral in
\eqref{e100} is bounded by
$C\mathcal R_{h,j}^\varepsilon
(\|\dot\Delta_j(\varrho^\varepsilon,u^\varepsilon)\|_{L^2_x}
+\|\dot\Delta_jf^\varepsilon\|_{L^2_{x,v}})$.
It follows from \eqref{e75},
\eqref{e97},
\eqref{e100} and
\eqref{e101} that
\begin{align}
&\frac{\mathrm d}{\mathrm dt}\left(
 \frac12\int_{\mathbb R^d_x}
 \frac{P'(1+\varrho^\varepsilon)}{1+\varrho^\varepsilon}
 |\dot\Delta_j\varrho^\varepsilon|^2\,\mathrm dx
 +\frac12\int_{\mathbb R^d_x}
 (1+\varrho^\varepsilon)|\dot\Delta_ju^\varepsilon|^2\,\mathrm dx
 +\frac12\|\dot\Delta_jf^\varepsilon\|_{L^2_{x,v}}^2
 \right)
 \nonumber\\
&\quad
 +c\varepsilon\|\nabla_x\dot\Delta_ju^\varepsilon\|_{L^2_x}^2
 +c\|\dot\Delta_j(u^\varepsilon-b^\varepsilon)\|_{L^2_x}^2
 +c\left\|\dot\Delta_j
 \{\mathbf I-\mathbf P\}f^\varepsilon\right\|_{L^2_xL^2_{v,\nu}}^2
 \nonumber\\
&\qquad\lesssim
 \mathcal R_{h,j}^\varepsilon
 \left(
 \|\dot\Delta_j(\varrho^\varepsilon,u^\varepsilon)\|_{L^2_x}
 +\|\dot\Delta_jf^\varepsilon\|_{L^2_{x,v}}
 \right).
\label{e102}
\end{align}

We next recover the dissipation of
$\dot\Delta_j\varrho^\varepsilon$. By
$\eqref{e98}$, we have
\begin{align}
&2^{-2j}\frac{\mathrm d}{\mathrm dt}\bigl(\dot\Delta_ju^\varepsilon,\, \nabla_x\dot\Delta_j\varrho^\varepsilon\bigr)_{L^2_x}+2^{-2j}\int_{\mathbb R^d_x}\frac{P'(1+\varrho^\varepsilon)}{1+\varrho^\varepsilon}|\nabla_x\dot\Delta_j\varrho^\varepsilon|^2\,dx\nonumber\\
&\quad\leq 2^{-2j}\int_{\mathbb R^d_x}(1+\varrho^\varepsilon)|\dive_x\dot\Delta_ju^\varepsilon|^2\,dx-2^{-2j}\bigl(\dot\Delta_j(u^\varepsilon-b^\varepsilon),\, \nabla_x\dot\Delta_j\varrho^\varepsilon\bigr)_{L^2_x}\nonumber\\
&\qquad+2^{-2j}\biggl(\frac{\varepsilon}{1+\varrho^\varepsilon}\bigl(\Delta_x\dot\Delta_ju^\varepsilon+2\nabla_x\dive_x\dot\Delta_ju^\varepsilon\bigr),\, \nabla_x\dot\Delta_j\varrho^\varepsilon\bigr)_{L^2_x}\nonumber\\
&\qquad+C\mathcal R_{h,j}^\varepsilon\Big(\|\dot\Delta_j(\varrho^\varepsilon,u^\varepsilon)\|_{L^2_x}+\|\dot\Delta_jf^\varepsilon\|_{L^2_{x,v}}\bigg).
\label{e103}
\end{align}
To recover the dissipation of $a^\varepsilon$ and $b^\varepsilon$, we localize
\eqref{e22} and pair the resulting equations with $\nabla_x\dot\Delta_ja^\varepsilon$.
We also apply $\Theta\dot\Delta_j$ to \eqref{e23} and take the inner product of the resulting equation 
with $\mathbb D_x\dot\Delta_jb^\varepsilon$. Consequently, we have
\begin{align}
&2^{-2j}\frac{\mathrm d}{\mathrm dt}\Big(\bigl(\nabla_x\dot\Delta_ja^\varepsilon,\, \dot\Delta_jb^\varepsilon\bigr)_{L^2_x}+\bigl(\Theta(\dot\Delta_j\{\mathbf I-\mathbf P\}f^\varepsilon),\, \mathbb D_x\dot\Delta_jb^\varepsilon\bigr)_{L^2_x}\Big)\nonumber\\
&\quad+2^{-2j}\|\nabla_x\dot\Delta_ja^\varepsilon\|_{L^2_x}^2+2\,2^{-2j}\|\mathbb D_x\dot\Delta_jb^\varepsilon\|_{L^2_x}^2\nonumber\\
&\quad\leq2^{-2j}\|\dive_x\dot\Delta_jb^\varepsilon\|_{L^2_x}^2+2^{-2j}\bigl(\dot\Delta_j(u^\varepsilon-b^\varepsilon),\, \nabla_x\dot\Delta_ja^\varepsilon\bigr)_{L^2_x}\nonumber\\
&\qquad-2\,2^{-2j}\bigl(\Theta(\dot\Delta_j\{\mathbf I-\mathbf P\}f^\varepsilon),\, \mathbb D_x\dot\Delta_jb^\varepsilon\bigr)_{L^2_x}-2^{-2j}\bigl(\Theta(v\cdot\nabla_x\dot\Delta_j\{\mathbf I-\mathbf P\}f^\varepsilon),\, \mathbb D_x\dot\Delta_jb^\varepsilon\bigr)_{L^2_x}\nonumber\\
&\qquad-2^{-2j}\bigl(\dive_x\Theta(\dot\Delta_j\{\mathbf I-\mathbf P\}f^\varepsilon),\, \dot\Delta_j(u^\varepsilon-b^\varepsilon)\bigr)_{L^2_x}+2^{-2j}\|\dive_x\Theta(\dot\Delta_j\{\mathbf I-\mathbf P\}f^\varepsilon)\|_{L^2_x}^2\nonumber\\
&\qquad+C\mathcal R_{h,j}^\varepsilon\bigl(\|\dot\Delta_j(a^\varepsilon,b^\varepsilon)\|_{L^2_x}+\|\dot\Delta_j\{\mathbf I-\mathbf P\}f^\varepsilon\|_{L^2_{x,v}}\bigr).
\label{e104}
\end{align}

We now define the Lyapunov functional
\begin{equation}\label{e105}
\begin{aligned}
\mathcal L_{h,j}^\varepsilon:={}&
 \frac12\int_{\mathbb R^d_x}
 \frac{P'(1+\varrho^\varepsilon)}
 {1+\varrho^\varepsilon}
 |\dot\Delta_j\varrho^\varepsilon|^2\,\mathrm dx
 +\frac12\int_{\mathbb R^d_x}
 (1+\varrho^\varepsilon)
 |\dot\Delta_ju^\varepsilon|^2\,\mathrm dx\\
&+\frac12\|\dot\Delta_jf^\varepsilon\|_{L^2_{x,v}}^2
 +\eta_1 2^{-2j}
 \bigl(\dot\Delta_ju^\varepsilon,
 \nabla_x\dot\Delta_j\varrho^\varepsilon\bigr)_{L^2_x}\\
&+\eta_2 2^{-2j}\left(
 \bigl(\nabla_x\dot\Delta_ja^\varepsilon,
 \dot\Delta_jb^\varepsilon\bigr)_{L^2_x}
 +\bigl(\Theta(\dot\Delta_j
 \{\mathbf I-\mathbf P\}f^\varepsilon),
 \mathbb D_x\dot\Delta_jb^\varepsilon\bigr)_{L^2_x}
 \right).
\end{aligned}
\end{equation}
By Bernstein's inequality, \eqref{E35}, and the Gaussian moment
bounds, the cross terms in \eqref{e105} are controlled by its
quadratic energy terms. Hence, for sufficiently small
$\eta_1,\eta_2>0$, equivalence \eqref{e95} holds.

We next multiply \eqref{e103} and \eqref{e104} by $\eta_1$ and
$\eta_2$, respectively, and add the resulting inequalities to
\eqref{e102}. Using Young's inequality, $0<\varepsilon\leq1$, $2\|\mathbb D_xb\|_2^2=\|\nabla_xb\|_2^2+
\|\dive_xb\|_2^2$, $u=(u-b)+b$ and \eqref{e75}, we first choose $\eta_2>0$
sufficiently small and then take $0<\eta_1\ll\eta_2$. Thus, all the linear coupling terms are
absorbed by the dissipative terms, while the remaining terms are bounded by $\mathcal R_{h,j}^\varepsilon
\sqrt{\mathcal L_{h,j}^\varepsilon}$. This proves \eqref{e96}.
\end{proof}

\begin{lemma}
\label{l8}
Let $0<\varepsilon\leq1$ and $T>0$. Assume that
$(\varrho^\varepsilon,u^\varepsilon,f^\varepsilon)$ is a solution to
the Cauchy problem \eqref{m1n} for the NS--VFP system on $[0,T]$ satisfying 
\eqref{e73}.
Then, for every $t\in[0,T]$, we have
\begin{align}
&\|(\varrho^{\varepsilon,h},u^{\varepsilon,h})\|_
{\widetilde L^\infty_t(\dot B^{d/2+1}_{2,1})}
+\|f^{\varepsilon,h}\|_
{\widetilde L^\infty_t(\dot{\mathcal B}^{d/2+1}_{2,1})}
\nonumber\\
&\quad
+\|(\varrho^{\varepsilon,h},u^{\varepsilon,h})\|_
{L^1_t(\dot B^{d/2+1}_{2,1})}
+\|f^{\varepsilon,h}\|_
{L^1_t(\dot{\mathcal B}^{d/2+1}_{2,1})}
\nonumber\\
&\quad
+\|(u^\varepsilon-b^\varepsilon)^h\|_
{\widetilde L^2_t(\dot B^{d/2+1}_{2,1})}
+\left\|
\{\mathbf I-\mathbf P\}f^{\varepsilon,h}\right\|_
{\widetilde L^2_t(\dot{\mathcal B}^{d/2+1}_{2,1,\nu})}+\varepsilon\|u^{\varepsilon,h}\|_
{L^1_t(\dot B^{d/2+2}_{2,1})}
\nonumber\\
&\quad
\leq C\|(\varrho_0^\varepsilon,u_0^\varepsilon,f_0^\varepsilon)\|_{\mathcal Z_0}
+C\bigl(\|(\varrho^\varepsilon,u^\varepsilon,f^\varepsilon)\|_{\mathcal Z_t}\bigr)^2
+C\bigl(\|(\varrho^\varepsilon,u^\varepsilon,f^\varepsilon)\|_{\mathcal Z_t}\bigr)^3,
\label{e106}
\end{align}
where the constant $C>0$ is independent of $\varepsilon$ and $t$.
\end{lemma}

\begin{proof}
As in \eqref{e46}--\eqref{e47}, \eqref{e75} and \eqref{e96} imply
\begin{align}
&\|(\varrho^\varepsilon,u^\varepsilon)\|_
{\widetilde L^\infty_t(\dot B^{d/2+1}_{2,1})}^{h}
+\|f^\varepsilon\|_
{\widetilde L^\infty_t(\dot{\mathcal B}^{d/2+1}_{2,1})}^{h}
+\|(\varrho^\varepsilon,u^\varepsilon)\|_
{L^1_t(\dot B^{d/2+1}_{2,1})}^{h}
+\|f^\varepsilon\|_
{L^1_t(\dot{\mathcal B}^{d/2+1}_{2,1})}^{h}
\nonumber\\
&\quad
+\|u^\varepsilon-b^\varepsilon\|_
{\widetilde L^2_t(\dot B^{d/2+1}_{2,1})}^{h}
+\left\|
\{\mathbf I-\mathbf P\}f^\varepsilon\right\|_
{\widetilde L^2_t(\dot{\mathcal B}^{d/2+1}_{2,1,\nu})}^{h}
+\varepsilon^{1/2}
\|u^\varepsilon\|_
{\widetilde L^2_t(\dot B^{d/2+2}_{2,1})}^{h}
\nonumber\\
&\qquad\lesssim
\|(\varrho_0^\varepsilon,u_0^\varepsilon,f_0^\varepsilon)\|_{\mathcal Z_0}
+\sum_{j\geq0}2^{j(d/2+1)}
\|\mathcal R_{h,j}^\varepsilon\|_{L^1_t}.
\label{e107}
\end{align}
By Lemmas \ref{l12} and
\ref{l13}, we have
\begin{align*}
&\sum_{j\geq0}2^{j(d/2+1)}
\|[\dot\Delta_j,u^\varepsilon]\cdot
\nabla_x(\varrho^\varepsilon,u^\varepsilon)\|_{L^1_tL^2_x}
\\
&\quad+
\sum_{j\geq0}2^{j(d/2+1)}
\left(
\|[\dot\Delta_j,\varrho^\varepsilon]\dive_xu^\varepsilon\|_{L^1_tL^2_x}
+\left\|\left[\dot\Delta_j,
\frac{P'(1+\varrho^\varepsilon)}
{1+\varrho^\varepsilon}\right]
\nabla_x\varrho^\varepsilon\right\|_{L^1_tL^2_x}
\right)
\\
&\qquad\lesssim
\|(\nabla_x\varrho^\varepsilon,\nabla_xu^\varepsilon)\|_
{L^1_t(\dot B^{d/2}_{2,1})}
\|(\varrho^\varepsilon,u^\varepsilon)\|_
{L^\infty_t(\dot B^{d/2+1}_{2,1})}.
\end{align*}
For the commutator involving the viscous term, we also have
\begin{align*}
&\varepsilon
\sum_{j\geq0}2^{j(d/2+1)}
\left\|
\left[\dot\Delta_j,
\frac{1}{1+\varrho^\varepsilon}\right]
\bigl(\Delta_xu^\varepsilon
+2\nabla_x\dive_xu^\varepsilon\bigr)
\right\|_{L^1_tL^2_x}\lesssim
 \|\nabla_x\varrho^\varepsilon\|_{L^\infty_t(\dot B^{d/2}_{2,1})} 
\,\varepsilon
\|u^\varepsilon\|_
{L^1_t(\dot B^{d/2+2}_{2,1})}.
\end{align*}
Lemmas \ref{l11} and \ref{l12}, the moment bound
\eqref{E35} and H\"older's inequality in time give
\begin{align}
&\left\|
\frac{a^\varepsilon u^\varepsilon
+\varrho^\varepsilon(b^\varepsilon-u^\varepsilon)}
{1+\varrho^\varepsilon}
\right\|_{L^1_t(\dot B^{d/2+1}_{2,1})}
+\|a^\varepsilon u^\varepsilon\|_
{L^1_t(\dot B^{d/2+1}_{2,1})}
+\|u^\varepsilon\otimes b^\varepsilon\|_
{L^1_t(\dot B^{d/2+1}_{2,1})}
\nonumber\\
&\quad+\|h(\varrho^\varepsilon,u^\varepsilon,f^\varepsilon)\|_
{L^1_t(\dot{\mathcal B}^{d/2+1}_{2,1})}
\nonumber\\
&\qquad\lesssim
\Bigl(1+\|\varrho^\varepsilon\|_
{L^\infty_t(\dot B^{d/2}_{2,1}\cap\dot B^{d/2+1}_{2,1})}\Bigr)
\Bigl(
\|a^\varepsilon\|_
{L^2_t(\dot B^{d/2}_{2,1}\cap\dot B^{d/2+1}_{2,1})}
\|u^\varepsilon\|_
{L^2_t(\dot B^{d/2}_{2,1}\cap\dot B^{d/2+1}_{2,1})}
\nonumber\\
&\hspace{47mm}
+\|\varrho^\varepsilon\|_
{L^2_t(\dot B^{d/2}_{2,1}\cap\dot B^{d/2+1}_{2,1})}
\|u^\varepsilon-b^\varepsilon\|_
{L^2_t(\dot B^{d/2}_{2,1}\cap\dot B^{d/2+1}_{2,1})}
\Bigr)
\nonumber\\
&\qquad\quad
+\|u^\varepsilon\|_
{L^2_t(\dot B^{d/2}_{2,1}\cap\dot B^{d/2+1}_{2,1})}
\Bigl(
\|b^\varepsilon\|_
{L^2_t(\dot B^{d/2}_{2,1}\cap\dot B^{d/2+1}_{2,1})}
\nonumber\\
&\hspace{58mm}
+\|\{\mathbf I-\mathbf P\}f^\varepsilon\|_
{L^2_t(\dot{\mathcal B}^{d/2}_{2,1,\nu}
\cap\dot{\mathcal B}^{d/2+1}_{2,1,\nu})}
\Bigr).
\label{e110}
\end{align}
Indeed, the moment bound controls $a^\varepsilon$ and $b^\varepsilon$
by $f^\varepsilon$. Interpolation between the $L^\infty_t$ and
$L^1_t$ components of $\mathcal Z_t$, and between the two indices in
its $L^2_t$ components, shows that the right-hand side of
\eqref{e110} is bounded by
$\|(\varrho^\varepsilon,u^\varepsilon,f^\varepsilon)\|_{\mathcal Z_t}^2
+\|(\varrho^\varepsilon,u^\varepsilon,f^\varepsilon)\|_{\mathcal Z_t}^3$.
Consequently, we have
\begin{align}
&\sum_{j\geq0}2^{j(d/2+1)}
\|\mathcal R_{h,j}^\varepsilon\|_{L^1_t}
\nonumber\\
&\qquad\lesssim
\|(\varrho^\varepsilon,u^\varepsilon,f^\varepsilon)\|_{\mathcal Z_t}^2
+\|(\varrho^\varepsilon,u^\varepsilon,f^\varepsilon)\|_{\mathcal Z_t}^3
+\|(\varrho^\varepsilon,u^\varepsilon,f^\varepsilon)\|_{\mathcal Z_t}
\,\varepsilon
\|u^\varepsilon\|_
{L^1_t(\dot B^{d/2+2}_{2,1})}.
\label{e111}
\end{align}
Substituting \eqref{e111} into
\eqref{e107}, we obtain
\begin{align}
&\|(\varrho^\varepsilon,u^\varepsilon)\|_
{\widetilde L^\infty_t(\dot B^{d/2+1}_{2,1})}^{h}
+\|f^\varepsilon\|_
{\widetilde L^\infty_t(\dot{\mathcal B}^{d/2+1}_{2,1})}^{h}
+\|(\varrho^\varepsilon,u^\varepsilon)\|_
{L^1_t(\dot B^{d/2+1}_{2,1})}^{h}
+\|f^\varepsilon\|_
{L^1_t(\dot{\mathcal B}^{d/2+1}_{2,1})}^{h}
\nonumber\\
&\quad
+\|u^\varepsilon-b^\varepsilon\|_
{\widetilde L^2_t(\dot B^{d/2+1}_{2,1})}^{h}
+\left\|
\{\mathbf I-\mathbf P\}f^\varepsilon\right\|_
{\widetilde L^2_t(\dot{\mathcal B}^{d/2+1}_{2,1,\nu})}^{h}
+\varepsilon^{1/2}
\|u^\varepsilon\|_
{\widetilde L^2_t(\dot B^{d/2+2}_{2,1})}^{h}
\nonumber\\
&\qquad\lesssim
\|(\varrho_0^\varepsilon,u_0^\varepsilon,f_0^\varepsilon)\|_{\mathcal Z_0}
+\bigl(\|(\varrho^\varepsilon,u^\varepsilon,f^\varepsilon)\|_{\mathcal Z_t}\bigr)^2
+\bigl(\|(\varrho^\varepsilon,u^\varepsilon,f^\varepsilon)\|_{\mathcal Z_t}\bigr)^3
\nonumber\\
&\qquad\quad
+\|(\varrho^\varepsilon,u^\varepsilon,f^\varepsilon)\|_{\mathcal Z_t}
\,\varepsilon
\|u^\varepsilon\|_
{L^1_t(\dot B^{d/2+2}_{2,1})}^{h}.
\label{e112}
\end{align}

It remains to recover the higher-order estimate of $u^\varepsilon$ depending on $\varepsilon$. The Lam\'e maximal regularity estimate for the momentum equation $\eqref{m1n}_2$ at
regularity $d/2$ (cf. Lemma \ref{heat}) leads to
\begin{align*}
\varepsilon
\|u^\varepsilon\|_
{L^1_t(\dot B^{d/2+2}_{2,1})}^{h}
&\lesssim
\|u^\varepsilon_0\|_{\dot B^{d/2}_{2,1}}^{h}
+\|\varrho^\varepsilon\|_
{L^1_t(\dot B^{d/2+1}_{2,1})}^{h}
+\|u^\varepsilon-b^\varepsilon\|_
{L^1_t(\dot B^{d/2}_{2,1})}^{h}
\\
&\quad
+\|u^\varepsilon\cdot\nabla_xu^\varepsilon\|_
{L^1_t(\dot B^{d/2}_{2,1})}^{h}
+\left\|
\frac{a^\varepsilon u^\varepsilon
+\varrho^\varepsilon(b^\varepsilon-u^\varepsilon)}
{1+\varrho^\varepsilon}
\right\|_
{L^1_t(\dot B^{d/2}_{2,1})}^{h}
\\
&\quad
+\|(\varrho^\varepsilon,u^\varepsilon,f^\varepsilon)\|_{\mathcal Z_t}
\,\varepsilon
\|u^\varepsilon\|_
{L^1_t(\dot B^{d/2+2}_{2,1})}^{h}.
\end{align*}
For high frequencies, we have
\[
\|u^\varepsilon-b^\varepsilon\|_
{L^1_t(\dot B^{d/2}_{2,1})}^{h}
\lesssim
\|(u^\varepsilon,b^\varepsilon)\|_
{L^1_t(\dot B^{d/2+1}_{2,1})}^{h}.
\]
Moreover, it is standard to verify
\[
\|u^\varepsilon\cdot\nabla_xu^\varepsilon\|_
{L^1_t(\dot B^{d/2}_{2,1})}^{h}
\lesssim
\|u^\varepsilon\|_{L^2_t(\dot B^{d/2}_{2,1})}
\|u^\varepsilon\|_{L^2_t(\dot B^{d/2+1}_{2,1})}
\lesssim
\|(\varrho^\varepsilon,u^\varepsilon,f^\varepsilon)\|_{\mathcal Z_t}^2,
\]
and, by the composition and product estimates,
\[
\left\|
\frac{a^\varepsilon u^\varepsilon
+\varrho^\varepsilon(b^\varepsilon-u^\varepsilon)}
{1+\varrho^\varepsilon}
\right\|_{L^1_t(\dot B^{d/2}_{2,1})}^{h}
\lesssim
\|(\varrho^\varepsilon,u^\varepsilon,f^\varepsilon)\|_{\mathcal Z_t}^2
+\|(\varrho^\varepsilon,u^\varepsilon,f^\varepsilon)\|_{\mathcal Z_t}^3.
\]
Using these estimates and \eqref{e112}, we obtain
\begin{align}
\varepsilon
\|u^\varepsilon\|_
{L^1_t(\dot B^{d/2+2}_{2,1})}^{h}
\lesssim
\|(\varrho_0^\varepsilon,u_0^\varepsilon,f_0^\varepsilon)\|_{\mathcal Z_0}
+\bigl(\|(\varrho^\varepsilon,u^\varepsilon,f^\varepsilon)\|_{\mathcal Z_t}\bigr)^2
+\bigl(\|(\varrho^\varepsilon,u^\varepsilon,f^\varepsilon)\|_{\mathcal Z_t}\bigr)^3.
\label{e114}
\end{align}
Substituting \eqref{e114} into
\eqref{e112} and using $\|z^h\|_{\dot B^r_{2,1}}
\lesssim\|z\|_{\dot B^r_{2,1}}^h$, we end up with
\eqref{e106}. This finishes the proof of Lemma \ref{l8}.
\end{proof}

\begin{proof}[\textbf{Proof of Proposition
\ref{p2}}]
We now combine Lemmas \ref{l6} and
\ref{l8}. The Chemin--Lerner bounds in these
lemmas control the ordinary norms in \eqref{Zt}, and hence we get
\begin{align*}
&\|(\varrho^\varepsilon,u^\varepsilon,f^\varepsilon)\|_{\mathcal Z_T}
+\varepsilon\|u^\varepsilon\|_
{L^1_T(\dot B^{d/2+1}_{2,1}\cap
\dot B^{d/2+2}_{2,1})}\\
&\quad\leq C\|(\varrho_0^\varepsilon,u_0^\varepsilon,
f_0^\varepsilon)\|_{\mathcal Z_0}
+C\|(\varrho_0^\varepsilon,u_0^\varepsilon,
f_0^\varepsilon)\|_{\mathcal Z_0}^2
+C\|(\varrho^\varepsilon,u^\varepsilon,
f^\varepsilon)\|_{\mathcal Z_T}^2
+C\|(\varrho^\varepsilon,u^\varepsilon,
f^\varepsilon)\|_{\mathcal Z_T}^3\\
&\qquad+C\|(\varrho^\varepsilon,u^\varepsilon,
f^\varepsilon)\|_{\mathcal Z_T}
\varepsilon\|u^\varepsilon\|_
{L^1_T(\dot B^{d/2+1}_{2,1}\cap
\dot B^{d/2+2}_{2,1})}.
\end{align*}
Under \eqref{e73}, one can absorb the terms on the right-hand side involving the solution
norm by choosing $\delta_1^*$ sufficiently small, and therefore arrive at \eqref{e74}. \end{proof}

\begin{proof}[\textbf{Proof of Theorem \ref{t2}}]
For each fixed $\varepsilon\in(0,1]$, the local well-posedness theory in Appendix B
gives a unique strong solution to the Cauchy problem \eqref{m1n} for
the NS--VFP system on a maximal interval
$[0,T_*^\varepsilon)$.

As in the proof of Theorem
\ref{t1}, Proposition
\ref{p2} shows that there exists a sufficiently small constant
$\delta_1>0$, independent of $\varepsilon$, such that
\eqref{e7} implies, for every
$T<T_*^\varepsilon$,
\[
\begin{aligned}
&\|(\varrho^\varepsilon,u^\varepsilon,f^\varepsilon)\|_{\mathcal Z_T}
+\varepsilon\|u^\varepsilon\|_
{L^1_T(\dot B^{d/2+1}_{2,1}
\cap\dot B^{d/2+2}_{2,1})}\leq
C_1\Big(
\|(\varrho_0^\varepsilon,u_0^\varepsilon,f_0^\varepsilon)\|_{\mathcal Z_0}
+\|(\varrho_0^\varepsilon,u_0^\varepsilon,f_0^\varepsilon)\|_{\mathcal Z_0}^2
\Big).
\end{aligned}
\]
The above uniform estimate and \eqref{e75} allow
the local solution to be continued beyond any
$T<T_*^\varepsilon$. Hence $T_*^\varepsilon=\infty$.

Finally, since $\delta_1$ is sufficiently small, the quadratic term of
the initial norm is controlled by the linear one. Thus
\eqref{e8} follows.
\end{proof}

\begin{proof}[\textbf{Proof of Theorem \ref{t3}}]
Let the constant $\delta_2\leq\delta_1$. Assume that
\eqref{e9} holds. For every $\varepsilon\in(0,1]$, let
$(\varrho^\varepsilon,u^\varepsilon,f^\varepsilon)$ denote the unique
global strong solution to the Cauchy problem \eqref{m1n} for the
NS--VFP system with the initial data $(\varrho_0^\varepsilon,u_0^\varepsilon,f_0^\varepsilon)
=(\varrho_0,u_0,f_0)$ obtained in Theorem \ref{t2}. Estimate
\eqref{e8} holds for every $T>0$:
\begin{align}
&\|(\varrho^\varepsilon,u^\varepsilon)\|_
 {L^\infty_T(\dot B^{d/2-1}_{2,1}\cap\dot B^{d/2+1}_{2,1})}
+\|f^\varepsilon\|_
 {L^\infty_T(\dot{\mathcal B}^{d/2-1}_{2,1}\cap
 \dot{\mathcal B}^{d/2+1}_{2,1})}
\le C\|(\varrho_0,u_0,f_0)\|_{\mathcal Z_0}.
\label{e115}
\end{align}
Interpolating the $L^\infty_t$ and $L^1_t$ bounds in
\eqref{Zt} and applying the resulting $L^2_t$ bounds and the product
estimates to $\eqref{m1n}_3$, we obtain
\begin{align}
&\|\partial_t\varrho^\varepsilon\|_
 {L^2_T(\dot B^{d/2-1}_{2,1})}
+\|\partial_tu^\varepsilon\|_
 {L^2_T(\dot B^{d/2-1}_{2,1})}
\le C_T\|(\varrho_0,u_0,f_0)\|_{\mathcal Z_0},
\label{e116}\\
&\|\partial_t(\chi\psi f^\varepsilon)\|_
 {L^2_T(H^{-1}(\mathbb R^{2d}))}
\le C_{T,\chi,\psi}\|(\varrho_0,u_0,f_0)\|_{\mathcal Z_0}.
\label{e117}
\end{align}
Here $\chi\in C_c^\infty(\mathbb R^d_x)$ and
$\psi\in C_c^\infty(\mathbb R^d_v)$ are arbitrary. Let $K$ be a compact subset of $\mathbb R^d$ and let $0<\eta<1$.
By the local compact embedding of Besov spaces in the $x$-variable,
\eqref{e115}--\eqref{e117} and the
Aubin--Lions--Simon theorem, we may extract a subsequence such that
\begin{align}
(\varrho^\varepsilon,u^\varepsilon)
&\longrightarrow(\varrho,u)
&&\text{in }C([0,T]; H^{d/2-1-\eta}(K)),
\label{e118}\\
(a^\varepsilon,b^\varepsilon)
&\longrightarrow(a,b)
&&\text{in }L^2(0,T; H^{d/2-1}(K)),
\label{e119}\\
f^\varepsilon&\rightharpoonup f
&&\text{in }\mathcal D'((0,T)\times K\times\mathbb R^d_v).
\label{e120}
\end{align}
By a diagonal argument, the subsequence may be chosen so that the
above convergence holds for every $K$ and $T$. By interpolation between \eqref{e115} and
\eqref{e118}, the fluid variables converge strongly
locally in $C([0,T];H^s(K))$ for some $s>d/2$. Together with
\eqref{e119}--\eqref{e120}, this convergence allows us to pass to
the limit in the nonlinear terms. In particular,
$u^\varepsilon f^\varepsilon\to uf$ in distributions. The term
$u^\varepsilon\cdot\nabla_vf^\varepsilon$ is treated after integration
by parts in $v$. Moreover, \eqref{e8} gives
\[
 \left\|\frac{\varepsilon}{1+\varrho^\varepsilon}
 \bigl(\Delta_xu^\varepsilon
 +2\nabla_x\dive_xu^\varepsilon\bigr)\right\|_
 {L^1_T(\dot B^{d/2-1}_{2,1})}
 \lesssim\varepsilon
 \|u^\varepsilon\|_{L^1_T(\dot B^{d/2+1}_{2,1})}
 \longrightarrow0.
\]
Hence $(\varrho,u,f)$ satisfies the Euler--VFP system \eqref{m1n0}
in the distributional sense and attains the initial data
$(\varrho_0,u_0,f_0)$.

For each fixed dyadic block, weak lower semicontinuity gives the
corresponding estimate for the limit.  The continuity of
$\|(\varrho,u,f)\|_{\mathcal Z_t}$ follows from the same
finite-frequency argument as in the proof of Theorem
\ref{t6}. The positivity of the density and distribution
function follows by passing to the limit from the approximating
solutions. Uniqueness follows from similar calculations in Section
\ref{s9}.
\end{proof}

\section{Vanishing viscosity limit: quantitative error estimate}
\label{s9}

\subsection{Difference equations}

In this section, we prove Theorem \ref{t4} regarding the global convergence rate of the inviscid limit 
by analyzing the
difference between the two solutions and treating the viscous term as a source.
Let $Z^\varepsilon:=(\varrho^\varepsilon,u^\varepsilon,f^\varepsilon)$
and $Z:=(\varrho,u,f)$ be the solutions given by Theorems
\ref{t2} and \ref{t3}, respectively, where we omit the
superscript \(0\) for the inviscid solution. We set
$\delta\varrho:=\varrho^\varepsilon-\varrho$,
$\delta u:=u^\varepsilon-u$ and
$\delta f:=f^\varepsilon-f$,
and define \(\delta a,\delta b\) by taking the corresponding moments
of \(\delta f\). Since $1+a^\varepsilon$ and $1+a$ are bounded away
from zero, we may define, without a low-frequency cut-off,
\[
\begin{aligned}
 V^\varepsilon&:=\frac{b^\varepsilon}{1+a^\varepsilon}\quad\text{and}\quad&
 G^\varepsilon&:=\{\mathbf I-\mathbf P\}f^\varepsilon
 -\frac12(v\otimes v-\mathrm{Id}):(b^\varepsilon\otimes b^\varepsilon)M^{1/2},\\
 V&:=\frac{b}{1+a}\quad\text{and}\quad&
 G&:=\{\mathbf I-\mathbf P\}f
 -\frac12(v\otimes v-\mathrm{Id}):(b\otimes b)M^{1/2}.
\end{aligned}
\]
We define the corrected error variables by
\begin{equation}\label{e121}
 \delta V:=V^\varepsilon-V
 =\frac{\delta b-\delta a V}{1+a^\varepsilon}
 \quad\text{and}\quad
 \delta G:=G^\varepsilon-G.
\end{equation}
The original and corrected errors are related by
\begin{equation}\label{e122}
\begin{aligned}
 \delta b&=(1+a^\varepsilon)\delta V+\delta a V,\\
 \{\mathbf I-\mathbf P\}\delta f&
 =\delta G+\frac12(v\otimes v-\mathrm{Id}):
 \bigl(\delta b\otimes b^\varepsilon+b\otimes\delta b\bigr)M^{1/2}.
\end{aligned}
\end{equation}
One can use \eqref{e122}  to recover the estimates for
$\delta b$ and $\{\mathbf I-\mathbf P\}\delta f$ from those for
$\delta V$ and $\delta G$.

Define $q(r):=P'(1+r)/(1+r)$ and use the function $g^0$ defined
in \eqref{m1n0}.  Subtracting the Euler--VFP system \eqref{m1n0} from the NS--VFP
system \eqref{m1n}, we obtain
\begin{equation*}
\left\{
\begin{aligned}
&\partial_t\delta\varrho
 +u^\varepsilon\cdot\nabla\delta\varrho+\dive\delta u
 =-\delta u\cdot\nabla\varrho
  -\varrho^\varepsilon\dive\delta u
  -\delta\varrho\,\dive u,\\
&\partial_t\delta u
 +u^\varepsilon\cdot\nabla\delta u+\nabla\delta\varrho
 +\delta u-\delta b=-\delta u\cdot\nabla u
 -\delta a\,u^\varepsilon\\
&\qquad-a\,\delta u
 +g^0(\varrho^\varepsilon,u^\varepsilon,f^\varepsilon)
 -g^0(\varrho,u,f)
 +\frac{\varepsilon}{1+\varrho^\varepsilon}
   \bigl(\Delta_xu^\varepsilon+2\nabla_x\dive_xu^\varepsilon\bigr),\\
&\partial_t\delta f+v\cdot\nabla_x\delta f
 -\mathcal L\delta f-\delta u\cdot vM^{1/2}
 =\frac12\delta u\cdot vf^\varepsilon
 -\delta u\cdot\nabla_vf^\varepsilon
 +\frac12u\cdot v\delta f-u\cdot\nabla_v\delta f.
\end{aligned}
\right.
\end{equation*}
Subtracting the equations used in Subsection
\ref{s7}, we obtain 
\begin{equation}\label{e124}
\left\{
\begin{aligned}
&\partial_t\delta\varrho+\dive_x\delta u
 =-\dive_x(\varrho^\varepsilon\delta u+\delta\varrho\,u),\\
&\partial_t\delta u+\nabla_x\delta\varrho+\delta u-\delta V
 =-u^\varepsilon\cdot\nabla_x\delta u-\delta u\cdot\nabla_xu
 +g^0(\varrho^\varepsilon,u^\varepsilon,f^\varepsilon)-g^0(\varrho,u,f)\\
&\qquad\quad
 -\delta a\,(u^\varepsilon-V^\varepsilon)
 -a(\delta u-\delta V)
 +\frac{\varepsilon}{1+\varrho^\varepsilon}
 (\Delta_xu^\varepsilon+2\nabla_x\dive_xu^\varepsilon),\\
&\partial_t\delta a+\dive_x\delta V
 =-\dive_x(a^\varepsilon\delta V+\delta a V),
 \end{aligned}
 \right.
 \end{equation}
 and
\begin{equation}\label{e125}
\left\{
\begin{aligned}
&\partial_t\delta V+\nabla_x\delta a
 -(\delta u-\delta V)
 +\dive_x\Theta(\delta G)\\
&\qquad=
 \frac{a^\varepsilon}{1+a^\varepsilon}\Big(\nabla_xa^\varepsilon
 +\dive_x\Theta(\{\mathbf I-\mathbf P\}f^\varepsilon)\Big)
 -\frac{a}{1+a}\Big(\nabla_xa
 +\dive_x\Theta(\{\mathbf I-\mathbf P\}f)\Big)\\
&\qquad\quad
 +\frac{V^\varepsilon}{1+a^\varepsilon}\dive_xb^\varepsilon
 -\frac{V}{1+a}\dive_xb
 -\dive_x\bigl(\delta b\otimes b^\varepsilon+b\otimes\delta b\bigr),\\
&\partial_t\delta G+\{\mathbf I-\mathbf P\}
 \bigl(v\cdot\nabla_x\delta G\bigr)-\mathcal L\delta G
 +(v\otimes v-\mathrm{Id}):\nabla_x\delta V M^{1/2}\\
&\qquad=h(\varrho^\varepsilon,u^\varepsilon,f^\varepsilon)
 -h(\varrho,u,f)\\
&\qquad\quad -(v\otimes v-\mathrm{Id}):
 \Big(\delta a\,(u^\varepsilon\otimes b^\varepsilon)
 +a\bigl(\delta u\otimes b^\varepsilon+u\otimes\delta b\bigr)\Big)M^{1/2}\\
&\qquad\quad
 +(v\otimes v-\mathrm{Id}):\Big(\bigl(\nabla_x\delta a
 +\dive_x\Theta(\{\mathbf I-\mathbf P\}\delta f)\bigr)\otimes b^\varepsilon\\
&\qquad\quad
 +\bigl(\nabla_xa+\dive_x\Theta(\{\mathbf I-\mathbf P\}f)\bigr)
 \otimes\delta b-\nabla_x\bigl(\delta a\,V^\varepsilon
 +a\delta V\bigr)\Big)M^{1/2}\\
&\qquad\quad
 -\{\mathbf I-\mathbf P\}\left(v\cdot\nabla_x\left(
 \frac12(v\otimes v-\mathrm{Id}):
 \bigl(\delta b\otimes b^\varepsilon+b\otimes\delta b\bigr)
 M^{1/2}\right)\right).
\end{aligned}
\right.
\end{equation}
For the high-frequency analysis, 
the difference system is written as
\begin{equation}\label{e126}
\left\{
\begin{aligned}
&\partial_t\delta\varrho
 +u^\varepsilon\cdot\nabla\delta\varrho
 +(1+\varrho^\varepsilon)\dive\delta u
 =-\delta u\cdot\nabla\varrho
  -\delta\varrho\,\dive u,\\
&\partial_t\delta u
 +u^\varepsilon\cdot\nabla\delta u
 +q(\varrho^\varepsilon)\nabla\delta\varrho
 +\delta u-\delta b
 =-\delta u\cdot\nabla u
  -\bigl(q(\varrho^\varepsilon)-q(\varrho)\bigr)\nabla\varrho\\
&\qquad
 +\frac{\varrho^\varepsilon
  (u^\varepsilon-b^\varepsilon)-a^\varepsilon u^\varepsilon}
 {1+\varrho^\varepsilon}
 -\frac{\varrho(u-b)-au}{1+\varrho}
 +\frac{\varepsilon}{1+\varrho^\varepsilon}
 \bigl(\Delta_xu^\varepsilon
 +2\nabla_x\dive_xu^\varepsilon\bigr),\\
&\partial_t\delta f+v\cdot\nabla_x\delta f
 -\mathcal L\delta f-\delta u\cdot vM^{1/2}
 =\frac12\delta u\cdot vf^\varepsilon
 -\delta u\cdot\nabla_vf^\varepsilon
 +\frac12u\cdot v\delta f-u\cdot\nabla_v\delta f.
\end{aligned}
\right.
\end{equation}
The difference of the two microscopic equations is
\begin{equation}\label{e127}
\begin{aligned}
&\partial_t\{\mathbf I-\mathbf P\}\delta f
+\{\mathbf I-\mathbf P\}
 \bigl(v\cdot\nabla_x\{\mathbf I-\mathbf P\}\delta f\bigr)
-\mathcal L\{\mathbf I-\mathbf P\}\delta f
+(v\otimes v-\mathrm{Id}):\nabla_x\delta b\,M^{1/2}\\
&\qquad=
h(\varrho^\varepsilon,u^\varepsilon,f^\varepsilon)
-h(\varrho,u,f)
+(v\otimes v-\mathrm{Id}):
 \bigl(\delta u\otimes b^\varepsilon
       +u\otimes\delta b\bigr)M^{1/2}.
\end{aligned}
\end{equation}
We will analyze \eqref{e124}--\eqref{e125} at low
frequencies and \eqref{e126}--
\eqref{e127} at high frequencies.

\subsection{Error functional}

For $j\leq1$, we define
\begin{equation}\label{e128}
\begin{aligned}
\mathcal L_{\ell,j}^{\rm e}:={}&
 \frac12\|\dot\Delta_j(\delta\varrho,\delta u,
 \delta a,\delta V)\|_{L^2_x}^2
 +\frac12\|\dot\Delta_j\delta G\|_{L^2_{x,v}}^2+\eta_1(\dot\Delta_j\delta u,
\nabla_x\dot\Delta_j\delta\varrho)_{L^2_x}\\
&+\eta_2\left((\dot\Delta_j\delta V,
\nabla_x\dot\Delta_j\delta a)_{L^2_x}
+(\Theta(\dot\Delta_j\delta G),
\mathbb D_x\dot\Delta_j\delta V)_{L^2_x}\right),
\end{aligned}
\end{equation}
where $0<\eta_1\ll\eta_2\ll1$ are chosen as in Lemma
\ref{l5}. Bernstein's inequality
gives
\begin{equation*}
 \mathcal L_{\ell,j}^{\rm e}\sim
 \|\dot\Delta_j(\delta\varrho,\delta u,
 \delta a,\delta V)\|_{L^2_x}^2
 +\|\dot\Delta_j\delta G\|_{L^2_{x,v}}^2.
\end{equation*}

Applying the energy argument of Lemma
\ref{l5} to
\eqref{e124}--\eqref{e125}, with the functional
\eqref{e128} and the same parameter choices, we obtain
\begin{align}
\frac{\mathrm d}{\mathrm dt}\mathcal L_{\ell,j}^{\rm e}
&+c2^{2j}\mathcal L_{\ell,j}^{\rm e}
+c\|\dot\Delta_j(\delta u-\delta V)\|_{L^2_x}^2
+c\|\dot\Delta_j\delta G\|_{L^2_xL^2_{v,\nu}}^2
\nonumber\\
&\lesssim\left(\mathcal R_{\ell,j}^{\rm e}
+\varepsilon\left\|\dot\Delta_j\left(
\frac{\Delta_xu^\varepsilon+2\nabla_x\dive_xu^\varepsilon}
{1+\varrho^\varepsilon}\right)\right\|_{L^2_x}\right)
\sqrt{\mathcal L_{\ell,j}^{\rm e}}.
\label{e130}
\end{align}
Here
\begin{equation}\label{e131}
\begin{aligned}
\mathcal R_{\ell,j}^{\rm e}:={}&
\|\dot\Delta_j\nabla_x(\varrho^\varepsilon\delta u
+\delta\varrho\,u)\|_{L^2_x}
+\|\dot\Delta_j\nabla_x(a^\varepsilon\delta V+\delta a V)\|_{L^2_x}\\
&+\|\dot\Delta_j\nabla_x(\delta b\otimes b^\varepsilon
+b\otimes\delta b)\|_{L^2_x}
+\|\dot\Delta_j(u^\varepsilon\cdot\nabla_x\delta u
+\delta u\cdot\nabla_xu)\|_{L^2_x}\\
&+\|\dot\Delta_j(g^0(\varrho^\varepsilon,u^\varepsilon,f^\varepsilon)
-g^0(\varrho,u,f))\|_{L^2_x}
+\|\dot\Delta_j(\delta a(u^\varepsilon-V^\varepsilon)
+a(\delta u-\delta V))\|_{L^2_x}\\
&+\left\|\dot\Delta_j\left(
\begin{aligned}
&\frac{a^\varepsilon}{1+a^\varepsilon}\Big(\nabla_xa^\varepsilon
+\dive_x\Theta(\{\mathbf I-\mathbf P\}f^\varepsilon)\Big)\\
&-\frac{a}{1+a}\Big(\nabla_xa
+\dive_x\Theta(\{\mathbf I-\mathbf P\}f)\Big)
\end{aligned}\right)\right\|_{L^2_x}\\
&+\left\|\dot\Delta_j\left(
\frac{V^\varepsilon}{1+a^\varepsilon}\dive_xb^\varepsilon
-\frac{V}{1+a}\dive_xb\right)\right\|_{L^2_x}
+\|\dot\Delta_j(h(\varrho^\varepsilon,u^\varepsilon,f^\varepsilon)
-h(\varrho,u,f))\|_{L^2_{x,v}}\\
&+\left\|\dot\Delta_j\left(
\Big(\delta a(u^\varepsilon\otimes b^\varepsilon)
+a(\delta u\otimes b^\varepsilon+u\otimes\delta b)\Big)
\right)\right\|_{L^2_x}\\
&+\left\|\dot\Delta_j\left(
\big(\nabla_x\delta a
+\dive_x\Theta(\{\mathbf I-\mathbf P\}\delta f)\big)
\otimes b^\varepsilon
\right)\right\|_{L^2_x}
\\
&+\left\|\dot\Delta_j\left(
\big(\nabla_xa+\dive_x\Theta(\{\mathbf I-\mathbf P\}f)\big)
\otimes\delta b\right)\right\|_{L^2_x}
\\
&+\|\dot\Delta_j\nabla_x(\delta a V^\varepsilon+a\delta V)\|_{L^2_x}.
\end{aligned}
\end{equation}

For \(j\geq0\), we apply \(\dot\Delta_j\) to the original
error variables and use the high-frequency energy of the preceding section.
For each error variable
\(z\in\{\delta\varrho,\delta u,\delta a,\delta b,\delta f\}\),
we write \(z_j:=\dot\Delta_jz\) and define
\begin{equation}\label{e132}
\begin{aligned}
\mathcal L_{h,j}^{\rm e}:={}&
\frac12\int_{\mathbb R^d_x}
q(\varrho^\varepsilon)|\delta\varrho_j|^2\,\mathrm dx
+\frac12\int_{\mathbb R^d_x}
(1+\varrho^\varepsilon)|\delta u_j|^2\,\mathrm dx
+\frac12\|\delta f_j\|_{L^2_{x,v}}^2\\
&+\eta_3 2^{-2j}\int_{\mathbb R^d_x}
\delta u_j\cdot\nabla_x\delta\varrho_j\,\mathrm dx\\
&+\eta_4 2^{-2j}\int_{\mathbb R^d_x}\Big(
\nabla_x\delta a_j\cdot\delta b_j
+\Theta\bigl(\{\mathbf I-\mathbf P\}\delta f_j\bigr):
\mathbb D_x\delta b_j\Big)\,\mathrm dx,
\end{aligned}
\end{equation}
where $0<\eta_3\ll\eta_4\ll1$ are chosen in the same way as
$\eta_1$ and $\eta_2$ in the proof of Lemma
\ref{l7}. \eqref{e75} and Bernstein's inequality imply
\begin{equation*}
\mathcal L_{h,j}^{\rm e}\sim
\|(\delta\varrho_j,\delta u_j)\|_{L^2_x}^2
+\|\delta f_j\|_{L^2_{x,v}}^2
\quad\text{for}\quad j\geq0.
\end{equation*}
Applying the argument of Lemma \ref{l7} to
\eqref{e126}--\eqref{e127}, with the functional \eqref{e132} and
the same choice of $\eta_3$ and $\eta_4$, we obtain
\begin{align}
\frac{\mathrm d}{\mathrm dt}\mathcal L_{h,j}^{\rm e}
&+c\mathcal L_{h,j}^{\rm e}
+c\|\dot\Delta_j(\delta u-\delta b)\|_{L^2_x}^2
+c\left\|\dot\Delta_j
\{\mathbf I-\mathbf P\}\delta f\right\|_{L^2_xL^2_{v,\nu}}^2
\nonumber\\
&\lesssim\left(\mathcal R_{h,j}^{\rm e}
+\varepsilon\left\|\dot\Delta_j\left(
\frac{\Delta_xu^\varepsilon+2\nabla_x\dive_xu^\varepsilon}
{1+\varrho^\varepsilon}\right)\right\|_{L^2_x}\right)
\sqrt{\mathcal L_{h,j}^{\rm e}}.
\label{e134}
\end{align}
Here
\begin{align}
\mathcal R_{h,j}^{\rm e}:={}&
\|(\nabla_xu^\varepsilon,\nabla_x\varrho^\varepsilon)\|_{L^\infty_x}
\|(\delta\varrho_j,\delta u_j)\|_{L^2_x}
\nonumber\\
&+\sum_{z\in\{\delta\varrho,\delta u\}}
\|[u^\varepsilon,\dot\Delta_j]\cdot\nabla_xz\|_{L^2_x}
+\|[\dot\Delta_j,\varrho^\varepsilon]\dive_x\delta u\|_{L^2_x}
+\|[\dot\Delta_j,q(\varrho^\varepsilon)]
\nabla_x\delta\varrho\|_{L^2_x}
\nonumber\\
&+\left\|\dot\Delta_j\bigl(
\delta u\cdot\nabla_x\varrho,
\delta\varrho\,\dive_xu,
\delta u\cdot\nabla_xu,
\bigl(q(\varrho^\varepsilon)-q(\varrho)\bigr)
\nabla_x\varrho\bigr)\right\|_{L^2_x}
\nonumber\\
&+\left\|\dot\Delta_j\left(
\frac{\varrho^\varepsilon
(u^\varepsilon-b^\varepsilon)-a^\varepsilon u^\varepsilon}
{1+\varrho^\varepsilon}
-\frac{\varrho(u-b)-au}{1+\varrho}\right)\right\|_{L^2_x}
\nonumber\\
&+\|\dot\Delta_j(\delta a\,u^\varepsilon
+a\,\delta u)\|_{L^2_x}
+\|\dot\Delta_j(\delta u\otimes b^\varepsilon
+u\otimes\delta b)\|_{L^2_x}
\nonumber\\
&+\left\|\dot\Delta_j\left(
\frac12\delta u\cdot vf^\varepsilon
-\delta u\cdot\nabla_vf^\varepsilon
+\frac12u\cdot v\delta f-u\cdot\nabla_v\delta f
\right)\right\|_{L^2_{x,v}}.
\label{e135}
\end{align}

\subsection{Estimate and convergence}

We define
\begin{align}
\mathfrak D_T^\varepsilon:={}&
 \|(\delta\varrho^\ell,\delta u^\ell)\|_
 {\widetilde L^\infty_T(\dot B^{d/2-1}_{2,1})}
 +\|\delta f^\ell\|_
 {\widetilde L^\infty_T(\dot{\mathcal B}^{d/2-1}_{2,1})}
\nonumber\\
&+\|(\delta\varrho^h,\delta u^h)\|_
 {\widetilde L^\infty_T(\dot B^{d/2-1}_{2,1})}
 +\|\delta f^h\|_
 {\widetilde L^\infty_T(\dot{\mathcal B}^{d/2-1}_{2,1})}
\nonumber\\
&+\|(\delta\varrho^\ell,\delta u^\ell)\|_
 {L^1_T(\dot B^{d/2+1}_{2,1})}
 +\|\delta f^\ell\|_
 {L^1_T(\dot{\mathcal B}^{d/2+1}_{2,1})}
\nonumber\\
&+\|(\delta\varrho^h,\delta u^h)\|_
 {L^1_T(\dot B^{d/2-1}_{2,1})}
 +\|\delta f^h\|_
 {L^1_T(\dot{\mathcal B}^{d/2-1}_{2,1})}
\nonumber\\
&+\|(\delta u-\delta b)^\ell\|_
 {\widetilde L^2_T(\dot B^{d/2-1}_{2,1})}
 +\|(\delta u-\delta b)^h\|_
 {\widetilde L^2_T(\dot B^{d/2-1}_{2,1})}
\nonumber\\
&+\|\{\mathbf I-\mathbf P\}\delta f^\ell\|_
 {\widetilde L^2_T(\dot{\mathcal B}^{d/2-1}_{2,1,\nu})}
 +\|\{\mathbf I-\mathbf P\}\delta f^h\|_
 {\widetilde L^2_T(\dot{\mathcal B}^{d/2-1}_{2,1,\nu})},
\label{e136}
\end{align}
and its initial part 
$\mathfrak D_0^\varepsilon
:=\|(\delta\varrho_0,\delta u_0)\|_{\dot B^{d/2-1}_{2,1}}
+\|\delta f_0\|_{\dot{\mathcal B}^{d/2-1}_{2,1}}$.

Applying the same argument as in
\eqref{e46}--\eqref{e47} to
\eqref{e130} and \eqref{e134}, multiplying by
$2^{j(d/2-1)}$ and summing over the corresponding frequencies, we obtain
\begin{align}
\mathfrak D_T^\varepsilon
\lesssim{}& \mathfrak D_0^\varepsilon
+\sum_{j\leq1}2^{j(d/2-1)}
 \|\mathcal R_{\ell,j}^{\rm e}\|_{L^1_T}
+\sum_{j\geq0}2^{j(d/2-1)}
 \|\mathcal R_{h,j}^{\rm e}\|_{L^1_T}
\nonumber\\
&+\varepsilon
\left\|
\frac{\Delta_xu^\varepsilon+2\nabla_x\dive_xu^\varepsilon}
{1+\varrho^\varepsilon}
\right\|_{L^1_T(\dot B^{d/2-1}_{2,1})}.
\label{e137}
\end{align}
Here we used \eqref{e121} and \eqref{e122} to recover the
low-frequency norms of $\delta b$ and
$\{\mathbf I-\mathbf P\}\delta f$ from those of
$\delta a$, $\delta V$ and $\delta G$.

We next estimate the remainder terms in \eqref{e137}. For
\(Q\in C^\infty\) with
\(Q(0)=0\), the product and composition laws yield
\begin{align*}
\|z^\varepsilon w^\varepsilon-zw\|_{\dot B^{d/2-1}_{2,1}}
&\lesssim
 \|z^\varepsilon-z\|_{\dot B^{d/2-1}_{2,1}}
 \|w^\varepsilon\|_{\dot B^{d/2}_{2,1}}
 +\|z\|_{\dot B^{d/2}_{2,1}}
 \|w^\varepsilon-w\|_{\dot B^{d/2-1}_{2,1}},\\
\|Q(z^\varepsilon)-Q(z)\|_{\dot B^{d/2-1}_{2,1}}
&\lesssim
 C\bigl(\|z^\varepsilon\|_{\dot B^{d/2+1}_{2,1}},
        \|z\|_{\dot B^{d/2+1}_{2,1}}\bigr)
 \|z^\varepsilon-z\|_{\dot B^{d/2-1}_{2,1}}.
\end{align*}
Theorems \ref{t2} and \ref{t3} give uniform bounds for the two
solutions, while \eqref{e136} implies
\begin{align}
&\|(\delta\varrho^\ell,\delta u^\ell,
   \delta a^\ell,\delta b^\ell)\|_
 {\widetilde L^2_T(\dot B^{d/2}_{2,1})}
+\|\delta f^\ell\|_
 {\widetilde L^2_T(\dot{\mathcal B}^{d/2}_{2,1})}
\nonumber\\
&\quad
+\|(\delta\varrho^h,\delta u^h,
   \delta a^h,\delta b^h)\|_
 {\widetilde L^2_T(\dot B^{d/2-1}_{2,1})}
+\|\delta f^h\|_
 {\widetilde L^2_T(\dot{\mathcal B}^{d/2-1}_{2,1})}
\lesssim\mathfrak D_T^\varepsilon.
\label{e138}
\end{align}
By \eqref{E35}, \eqref{e121} and
\eqref{e122}, the same low-frequency bounds hold for $\delta V$
and $\delta G$. The product and composition estimates also give
\[
\|\delta V\|_{\widetilde L^\infty_T(\dot B^{d/2-1}_{2,1})}
+\|\delta V^h\|_{L^1_T(\dot B^{d/2-1}_{2,1})}
\lesssim\mathfrak D_T^\varepsilon.
\]
We shall use
H\"older's inequality in time in the forms
$L^\infty_T\times L^1_T$ and $L^2_T\times L^2_T$. Note that
\begin{align}
g^0(\varrho^\varepsilon,u^\varepsilon,f^\varepsilon)
-g^0(\varrho,u,f)&=
-\bigl(q(\varrho^\varepsilon)-q(\varrho)\bigr)
 \nabla\varrho
-\bigl(q(\varrho^\varepsilon)-1\bigr)\nabla\delta\varrho
\nonumber\\
&\quad+\Big(\frac{\varrho^\varepsilon}{1+\varrho^\varepsilon}
-\frac{\varrho}{1+\varrho}\Big)
 \bigl(u^\varepsilon-b^\varepsilon
 +a^\varepsilon u^\varepsilon\bigr)
\nonumber\\
&\quad+\frac{\varrho}{1+\varrho}
 \bigl(\delta u-\delta b+\delta a\,u^\varepsilon
 +a\,\delta u\bigr),
\label{e139}\\
\frac12u^\varepsilon\cdot vf^\varepsilon
-u^\varepsilon\cdot\nabla_vf^\varepsilon
-\frac12u\cdot vf+u\cdot\nabla_vf
&=\frac12\delta u\cdot v f^\varepsilon
-\delta u\cdot\nabla_vf^\varepsilon
\nonumber\\
&\quad+\frac12u\cdot v\delta f-u\cdot\nabla_v\delta f.
\label{e140}
\end{align}
We also use
\begin{align*}
a^\varepsilon V^\varepsilon-aV
&=a^\varepsilon\delta V+\delta a\,V\quad\text{and}\quad b^\varepsilon\otimes b^\varepsilon-b\otimes b
=\delta b\otimes b^\varepsilon
+b\otimes\delta b,\\
a^\varepsilon(u^\varepsilon-V^\varepsilon)-a(u-V)
&=\delta a\,(u^\varepsilon-V^\varepsilon)
+a(\delta u-\delta V),\\
u^\varepsilon\cdot\nabla_xu^\varepsilon-u\cdot\nabla_xu
&=\delta u\cdot\nabla_xu
+u^\varepsilon\cdot\nabla_x\delta u.
\end{align*}
The product and composition estimates in Lemmas \ref{l11} and
\ref{l12}, together with \eqref{e138}, yield
\begin{align}
&\|(\delta u\cdot\nabla\varrho,
\delta u\cdot\nabla u,
\delta\varrho\,\dive u,
\bigl(q(\varrho^\varepsilon)-q(\varrho)\bigr)
\nabla\varrho)\|_
 {L^1_T(\dot B^{d/2-1}_{2,1})}
\nonumber\\
&\qquad\lesssim
\bigl(\|(\varrho^\varepsilon,u^\varepsilon,f^\varepsilon)\|_{\mathcal Z_T}
+\|(\varrho,u,f)\|_{\mathcal Z_T}\bigr)\mathfrak D_T^\varepsilon.
\label{e141}
\end{align}
The commutator estimate in Lemma \ref{l13}, the composition estimate
in Lemma \ref{l12} and \eqref{e138} give
\begin{align}
&\sum_{j\geq0}2^{j(d/2-1)}
\Big(
\|[u^\varepsilon,\dot\Delta_j]\cdot
\nabla_x\delta\varrho\|_{L^1_TL^2_x}
+\|[u^\varepsilon,\dot\Delta_j]\cdot
\nabla_x\delta u\|_{L^1_TL^2_x}
\nonumber\\
&\quad
+\|[\dot\Delta_j,\varrho^\varepsilon]
\dive_x\delta u\|_{L^1_TL^2_x}
+\|[\dot\Delta_j,q(\varrho^\varepsilon)]
\nabla_x\delta\varrho\|_{L^1_TL^2_x}
\Big)
\nonumber\\
&\quad
+\|(\nabla_xu^\varepsilon,\nabla_x\varrho^\varepsilon)\|_
 {L^1_T(L^\infty_x)}
\|(\delta\varrho,\delta u)\|_
 {L^\infty_T(\dot B^{d/2-1}_{2,1})}
\nonumber\\
&\qquad\lesssim
\|(\varrho^\varepsilon,u^\varepsilon,f^\varepsilon)\|_{\mathcal Z_T}
\|(\delta\varrho,\delta u)\|_
 {L^\infty_T(\dot B^{d/2-1}_{2,1})}.
\label{e142}
\end{align}
At high frequencies, the terms
$u^\varepsilon\cdot\nabla_x(\delta\varrho,\delta u)$,
$(1+\varrho^\varepsilon)\dive_x\delta u$ and
$q(\varrho^\varepsilon)\nabla_x\delta\varrho$ are kept on the
left-hand side of \eqref{e126}; the resulting commutators are
estimated by \eqref{e142}. At low frequencies, the product estimate
gives, for example,
\begin{align*}
\|\dive_x(\varrho^\varepsilon\delta u)^\ell\|_
 {L^1_T(\dot B^{d/2-1}_{2,1})}&\lesssim
\|(\varrho^\varepsilon\delta u^\ell)^\ell\|_
 {L^1_T(\dot B^{d/2}_{2,1})}
+\|(\varrho^\varepsilon\delta u^h)^\ell\|_
 {L^1_T(\dot B^{d/2-1}_{2,1})}
\\
&\quad\lesssim
\|\varrho^\varepsilon\|_{L^2_T(\dot B^{d/2}_{2,1})}
\|\delta u^\ell\|_{L^2_T(\dot B^{d/2}_{2,1})}
+\|\varrho^\varepsilon\|_{L^\infty_T(\dot B^{d/2}_{2,1})}
\|\delta u^h\|_{L^1_T(\dot B^{d/2-1}_{2,1})}
\\
&\quad
\lesssim
\|(\varrho^\varepsilon,u^\varepsilon,f^\varepsilon)\|_{\mathcal Z_T}
\mathfrak D_T^\varepsilon.
\end{align*}
For $(\nabla_x\delta a)\otimes b^\varepsilon$, we write
$(\nabla_x\delta a)\otimes b^\varepsilon
=\nabla_x(\delta a\,b^\varepsilon)-\delta a\,\nabla_xb^\varepsilon$
and obtain
\begin{align*}
&\|((\nabla_x\delta a)\otimes b^\varepsilon)^\ell\|_
 {L^1_T(\dot B^{d/2-1}_{2,1})}\\
&\quad\lesssim
\|\delta a^\ell\|_{L^2_T(\dot B^{d/2}_{2,1})}
\|b^\varepsilon\|_{L^2_T(\dot B^{d/2}_{2,1})}
+\|\delta a^h\|_{L^1_T(\dot B^{d/2-1}_{2,1})}
\|b^\varepsilon\|_{L^\infty_T(\dot B^{d/2}_{2,1})}\\
&\qquad
+\|\delta a^h\|_{L^\infty_T(\dot B^{d/2-1}_{2,1})}
\|b^\varepsilon\|_{L^1_T(\dot B^{d/2+1}_{2,1})}
\lesssim
\|(\varrho^\varepsilon,u^\varepsilon,f^\varepsilon)\|_{\mathcal Z_T}
\mathfrak D_T^\varepsilon.
\end{align*}
The identities
\[
u^\varepsilon-V^\varepsilon
=u^\varepsilon-b^\varepsilon+a^\varepsilon V^\varepsilon
\quad\text{and}\quad
\delta u-\delta V
=\delta u-\delta b+a^\varepsilon\delta V+\delta a V,
\]
the moment estimate \eqref{E35} and Lemmas \ref{l11}--\ref{l12}
imply
\[
\|u^\varepsilon-V^\varepsilon\|_
 {L^2_T(\dot B^{d/2-1}_{2,1}\cap\dot B^{d/2}_{2,1})}
\lesssim\|Z^\varepsilon\|_{\mathcal Z_T}
\quad\text{and}\quad
\|\delta u-\delta V\|_{L^2_T(\dot B^{d/2-1}_{2,1})}
\lesssim\mathfrak D_T^\varepsilon.
\]
Together with \eqref{e138} and \eqref{uv2}, these bounds give
\begin{align*}
&\|\delta a\,(u^\varepsilon-V^\varepsilon)\|_
 {L^1_T(\dot B^{d/2-1}_{2,1})}^{\ell}
+\|a(\delta u-\delta V)\|_
 {L^1_T(\dot B^{d/2-1}_{2,1})}^{\ell}
\\
&\quad\lesssim
\|\delta a^\ell\|_{L^2_T(\dot B^{d/2}_{2,1})}
\|u^\varepsilon-V^\varepsilon\|_
 {L^2_T(\dot B^{d/2-1}_{2,1})}
+\|\delta a^h\|_{L^2_T(\dot B^{d/2-1}_{2,1})}
\|u^\varepsilon-V^\varepsilon\|_
 {L^2_T(\dot B^{d/2}_{2,1})}
\\
&\qquad\quad
+\|a\|_{L^2_T(\dot B^{d/2}_{2,1})}
\|\delta u-\delta V\|_
 {L^2_T(\dot B^{d/2-1}_{2,1})}
\\
&\quad
\lesssim
\bigl(\|(\varrho^\varepsilon,u^\varepsilon,f^\varepsilon)\|_{\mathcal Z_T}
+\|(\varrho,u,f)\|_{\mathcal Z_T}\bigr)
\mathfrak D_T^\varepsilon.
\end{align*}
Using \eqref{e141} and estimating the remaining terms in
\eqref{e139} by the same product laws, we obtain
\[
\|g^0(\varrho^\varepsilon,u^\varepsilon,f^\varepsilon)
-g^0(\varrho,u,f)\|_{L^1_T(\dot B^{d/2-1}_{2,1})}^{\ell}
\lesssim
\bigl(\|(\varrho^\varepsilon,u^\varepsilon,f^\varepsilon)\|_{\mathcal Z_T}
+\|(\varrho,u,f)\|_{\mathcal Z_T}\bigr)
\mathfrak D_T^\varepsilon.
\]
Lemmas \ref{l11} and \ref{l12}, Theorems \ref{t2} and
\ref{t3}, \eqref{e138} and H\"older's inequality in time imply
\begin{align}
&\left\|
\frac{\varrho^\varepsilon
(u^\varepsilon-b^\varepsilon)-a^\varepsilon u^\varepsilon}
{1+\varrho^\varepsilon}
-\frac{\varrho(u-b)-au}{1+\varrho}
\right\|_{L^1_T(\dot B^{d/2-1}_{2,1})}^{h}
\nonumber\\
&\quad
+\|\delta a\,u^\varepsilon+a\,\delta u\|_
 {L^1_T(\dot B^{d/2-1}_{2,1})}^{h}
+\|\delta u\otimes b^\varepsilon+u\otimes\delta b\|_
 {L^1_T(\dot B^{d/2-1}_{2,1})}^{h}
\nonumber\\
&\qquad\lesssim
\bigl(\|(\varrho^\varepsilon,u^\varepsilon,f^\varepsilon)\|_{\mathcal Z_T}
+\|(\varrho,u,f)\|_{\mathcal Z_T}\bigr)\mathfrak D_T^\varepsilon.
\label{e145}
\end{align}
The definition \eqref{e4}, the velocity norm \eqref{L12} and
H\"older's inequality in time also give
\begin{align*}
&\|h(\varrho^\varepsilon,u^\varepsilon,f^\varepsilon)
-h(\varrho,u,f)\|_{L^1_T(\dot{\mathcal B}^{d/2-1}_{2,1})}\\
&\quad\lesssim
\|\delta u^\ell\|_{L^2_T(\dot B^{d/2}_{2,1})}
\|\{\mathbf I-\mathbf P\}f^\varepsilon\|_
{L^2_T(\dot{\mathcal B}^{d/2-1}_{2,1,\nu})}+
\|\delta u^h\|_{L^2_T(\dot B^{d/2-1}_{2,1})}
\|\{\mathbf I-\mathbf P\}f^\varepsilon\|_
{L^2_T(\dot{\mathcal B}^{d/2}_{2,1,\nu})}\\
&\qquad+
\|u\|_{L^2_T(\dot B^{d/2}_{2,1})}
\|\{\mathbf I-\mathbf P\}\delta f\|_
{L^2_T(\dot{\mathcal B}^{d/2-1}_{2,1,\nu})}\\
&\quad\lesssim
\bigl(\|(\varrho^\varepsilon,u^\varepsilon,f^\varepsilon)\|_{\mathcal Z_T}
+\|(\varrho,u,f)\|_{\mathcal Z_T}\bigr)\mathfrak D_T^\varepsilon.
\end{align*}
We also write
\begin{align*}
&\frac{a^\varepsilon}{1+a^\varepsilon}
 \Big(\nabla_xa^\varepsilon
 +\dive_x\Theta(\{\mathbf I-\mathbf P\}f^\varepsilon)\Big)
-\frac{a}{1+a}
 \Big(\nabla_xa+\dive_x\Theta(\{\mathbf I-\mathbf P\}f)\Big)\\
&\quad=\left(\frac{a^\varepsilon}{1+a^\varepsilon}
-\frac{a}{1+a}\right)
 \Big(\nabla_xa^\varepsilon
 +\dive_x\Theta(\{\mathbf I-\mathbf P\}f^\varepsilon)\Big)
 +\frac{a}{1+a}
 \Big(\nabla_x\delta a
 +\dive_x\Theta(\{\mathbf I-\mathbf P\}\delta f)\Big),\\
&\frac{V^\varepsilon}{1+a^\varepsilon}\dive_xb^\varepsilon
-\frac{V}{1+a}\dive_xb
=\left(\frac{V^\varepsilon}{1+a^\varepsilon}
-\frac{V}{1+a}\right)\dive_xb^\varepsilon
+\frac{V}{1+a}\dive_x\delta b.
\end{align*}
By \eqref{e138}, \eqref{uv2}, Lemma
\ref{l12} and H\"older's inequality in time, one has
\begin{align*}
&\left\|
\frac{a^\varepsilon}{1+a^\varepsilon}\Big(\nabla_xa^\varepsilon
+\dive_x\Theta(\{\mathbf I-\mathbf P\}f^\varepsilon)\Big)
-\frac{a}{1+a}\Big(\nabla_xa
+\dive_x\Theta(\{\mathbf I-\mathbf P\}f)\Big)\right\|_
{L^1_T(\dot B^{d/2-1}_{2,1})}^{\ell}
\\
&\quad+\left\|\frac{V^\varepsilon}{1+a^\varepsilon}\dive_xb^\varepsilon
-\frac{V}{1+a}\dive_xb\right\|_
{L^1_T(\dot B^{d/2-1}_{2,1})}^{\ell}
\\
&\quad\quad\lesssim
\bigl(\|(\varrho^\varepsilon,u^\varepsilon,f^\varepsilon)\|_{\mathcal Z_T}
+\|(\varrho,u,f)\|_{\mathcal Z_T}\bigr)\mathfrak D_T^\varepsilon.
\end{align*}

For the cubic nonlinear term, we finally use the identity
\[
a^\varepsilon(u^\varepsilon\otimes b^\varepsilon)-a(u\otimes b)
=\delta a\,(u^\varepsilon\otimes b^\varepsilon)
+a(\delta u\otimes b^\varepsilon+u\otimes\delta b).
\]
The product estimate then gives
\[
\begin{aligned}
&\|a^\varepsilon(u^\varepsilon\otimes b^\varepsilon)
-a(u\otimes b)\|_{L^1_T(\dot B^{d/2-1}_{2,1})}\\
&\quad\lesssim
\bigl(\|(\varrho^\varepsilon,u^\varepsilon,f^\varepsilon)\|_{\mathcal Z_T}
+\|(\varrho,u,f)\|_{\mathcal Z_T}\bigr)^2
\mathfrak D_T^\varepsilon
\lesssim
\bigl(\|(\varrho^\varepsilon,u^\varepsilon,f^\varepsilon)\|_{\mathcal Z_T}
+\|(\varrho,u,f)\|_{\mathcal Z_T}\bigr)
\mathfrak D_T^\varepsilon.
\end{aligned}
\]
The last inequality uses the uniform smallness of the two solutions.
The two low-frequency product estimates above control the remaining
terms in \eqref{e131}. Equation \eqref{e140}, together with
\eqref{e145} and the estimate for $h$, controls the last term in
\eqref{e135}. The preceding estimates imply
\begin{equation}\label{e147}
\sum_{j\leq1}2^{j(d/2-1)}
 \|\mathcal R_{\ell,j}^{\rm e}\|_{L^1_T}
 +\sum_{j\geq0}2^{j(d/2-1)}
 \|\mathcal R_{h,j}^{\rm e}\|_{L^1_T}
 \lesssim
 \bigl(\|(\varrho^\varepsilon,u^\varepsilon,f^\varepsilon)\|_{\mathcal Z_T}
 +\|(\varrho,u,f)\|_{\mathcal Z_T}\bigr)
 \mathfrak D_T^\varepsilon.
\end{equation}

The viscous source determines the convergence rate $\mathcal{O}(\varepsilon)$. Indeed, Lemma \ref{l12} and Theorem \ref{t2} give
\begin{equation}
 \varepsilon\left\|
 \frac{\Delta_xu^\varepsilon+2\nabla_x\dive_xu^\varepsilon}{1+\varrho^\varepsilon}
 \right\|_{L^1_T(\dot B^{d/2-1}_{2,1})}
 \lesssim\varepsilon\|(\varrho_0^\varepsilon,u_0^\varepsilon,f_0^\varepsilon)\|_{\mathcal Z_0}.
\label{e148}
\end{equation}

Substituting \eqref{e147} and
\eqref{e148} into \eqref{e137}, we obtain
\begin{equation}\label{e149}
\mathfrak D_T^\varepsilon
\le C\mathfrak D_0^\varepsilon
+C\bigl(\|Z^\varepsilon\|_{\mathcal Z_T}
+\|Z\|_{\mathcal Z_T}\bigr)
\mathfrak D_T^\varepsilon
+C\varepsilon
\|(\varrho_0^\varepsilon,u_0^\varepsilon,f_0^\varepsilon)\|_{\mathcal Z_0}.
\end{equation}
By the uniform estimates of the two solutions, we absorb the
second term on the right-hand side of \eqref{e149}
and obtain
\[
\mathfrak D_T^\varepsilon
\le C(\mathfrak D_0^\varepsilon+\varepsilon).
\]
Here $C$ is independent of $T$ and $\varepsilon$.
Together with \eqref{e11}, this proves
\eqref{e12}. Interpolation with the uniform
$\dot B^{d/2+1}_{2,1}$ bound gives
\[
\|(\delta\varrho,\delta u)(t)\|_{\dot B^{d/2}_{2,1}}
+\|\delta f(t)\|_{\dot{\mathcal B}^{d/2}_{2,1}}
\le C\sqrt{\varepsilon},
\]
and the corresponding $L^\infty$ estimate follows from the critical
Besov embeddings. The proof of Theorem \ref{t4} is complete.

\section{Uniform time-decay estimates}
\label{s10}

In this section, we establish the optimal decay estimates. We only give the proof for $0<\varepsilon\le1$, as the case $\varepsilon=0$ is the same,
with the viscous terms omitted. 

\subsection{Propagation of the lower-order norm}

We first establish the uniform evolution of the lower-order norms. Let
\[
K_0:=\|(\varrho_0,u_0,f_0)\|_{\mathcal Z_0}
+\|(\varrho_0,u_0)\|_{\dot B^{\sigma_1}_{2,\infty}}
+\|f_0\|_{\dot{\mathcal B}^{\sigma_1}_{2,\infty}},
\]
and set
\begin{align}
\mathcal N_{\sigma_1}^\varepsilon(T):={}&
 \|(\varrho^{\varepsilon,\ell},u^{\varepsilon,\ell})\|_
 {\widetilde L^\infty_T(\dot B^{\sigma_1}_{2,\infty})}
 +\|f^{\varepsilon,\ell}\|_
 {\widetilde L^\infty_T
  (\dot{\mathcal B}^{\sigma_1}_{2,\infty})}
\nonumber\\
&+\|(\varrho^{\varepsilon,\ell},u^{\varepsilon,\ell})\|_
 {\widetilde L^1_T(\dot B^{\sigma_1+2}_{2,\infty})}
 +\|f^{\varepsilon,\ell}\|_
 {\widetilde L^1_T(\dot{\mathcal B}^{\sigma_1+2}_{2,\infty})}
\nonumber\\
&+\|(u^\varepsilon-V^\varepsilon)^\ell\|_
 {\widetilde L^2_T(\dot B^{\sigma_1}_{2,\infty})}
 +\|
 G^{\varepsilon,\ell}\|_
 {\widetilde L^2_T
  (\dot{\mathcal B}^{\sigma_1}_{2,\infty,\nu})}.
\label{e150}
\end{align}

\begin{proposition}\label{p3}
Under the assumptions of Theorem \ref{t5}, we have
\begin{align}
&\|(\varrho^\varepsilon,u^\varepsilon)\|_
 {\widetilde L^\infty_T(\dot B^{\sigma_1}_{2,\infty})}
 +\|f^\varepsilon\|_
 {\widetilde L^\infty_T(\dot{\mathcal B}^{\sigma_1}_{2,\infty})}
 +\mathcal N_{\sigma_1}^\varepsilon(T)
 \le C K_0\quad\text{for all}\quad T>0.
\label{e151}
\end{align}
\end{proposition}

\begin{proof}
Multiplying \eqref{e89} by
\(2^{j\sigma_1}\) and taking the supremum over \(j\leq1\), we obtain
\begin{equation}\label{e152}
\mathcal N_{\sigma_1}^\varepsilon(T)
\lesssim
 \|(\varrho_0^{\varepsilon},a_0^{\varepsilon},u_0^{\varepsilon},V_0^\varepsilon)\|_
 {\dot B^{\sigma_1}_{2,\infty}}
+\|G_0^\varepsilon\|_{\dot{\mathcal B}^{\sigma_1}_{2,\infty}}
+\sup_{j\leq1}2^{j\sigma_1}
 \|\mathcal R_{\ell,j}^\varepsilon\|_{L^1_T}.
\end{equation}
Time interpolation for
\((\varrho^{\varepsilon,\ell},u^{\varepsilon,\ell})\) and
\(f^{\varepsilon,\ell}\) in \eqref{e150} yields
\begin{equation}\label{e153}
\|(\varrho^{\varepsilon,\ell},u^{\varepsilon,\ell})\|_{\widetilde L^2_T(\dot B^{\sigma_1+1}_{2,\infty})}
+\|f^{\varepsilon,\ell}\|_{\widetilde L^2_T(\dot{\mathcal B}^{\sigma_1+1}_{2,\infty})}
\lesssim\mathcal N_{\sigma_1}^\varepsilon(T).
\end{equation}
Since \(\sigma_1+1<d/2\), the high-frequency estimates in
\(\mathcal Z_T\) imply
\begin{align}
&\|(\varrho^{\varepsilon,h},u^{\varepsilon,h})\|_
 {\widetilde L^\infty_T(\dot B^{\sigma_1}_{2,\infty})}
 +\|f^{\varepsilon,h}\|_
 {\widetilde L^\infty_T(\dot{\mathcal B}^{\sigma_1}_{2,\infty})}
 +\|(\varrho^{\varepsilon,h},u^{\varepsilon,h})\|_
 {\widetilde L^2_T(\dot B^{\sigma_1+1}_{2,\infty})}
 +\|f^{\varepsilon,h}\|_
 {\widetilde L^2_T(\dot{\mathcal B}^{\sigma_1+1}_{2,\infty})}
\nonumber\\
&\quad+\|(u^\varepsilon-b^\varepsilon)^h\|_
 {\widetilde L^2_T(\dot B^{\sigma_1}_{2,\infty})}
 +\|\{\mathbf I-\mathbf P\}f^{\varepsilon,h}\|_
 {\widetilde L^2_T(\dot{\mathcal B}^{\sigma_1}_{2,\infty,\nu})}
 \lesssim
 \|(\varrho^\varepsilon,u^\varepsilon,f^\varepsilon)\|_{\mathcal Z_T}.
\label{e154}
\end{align}
Interpolation between the two positive regularity indices in
\eqref{Zt} also gives
\begin{equation}\label{e155}
 \|(\varrho^\varepsilon,u^\varepsilon)\|_{L^2_T(\dot B^{d/2}_{2,1})}
 +\|f^\varepsilon\|_{L^2_T(\dot{\mathcal B}^{d/2}_{2,1})}
 \lesssim
 \|(\varrho^\varepsilon,u^\varepsilon,f^\varepsilon)\|_{\mathcal Z_T}.
\end{equation}
Estimate \eqref{e155} also holds for the moments
\((a^\varepsilon,b^\varepsilon)\). Thus, \eqref{e153}--\eqref{e155},
the definitions of \(V^\varepsilon\) and \(G^\varepsilon\), and Lemmas
\ref{l11}--\ref{l12} imply
\begin{align}
&\|(\varrho^\varepsilon,u^\varepsilon)\|_
 {\widetilde L^\infty_T(\dot B^{\sigma_1}_{2,\infty})}
 +\|f^\varepsilon\|_
 {\widetilde L^\infty_T(\dot{\mathcal B}^{\sigma_1}_{2,\infty})}
 +\|(\varrho^\varepsilon,u^\varepsilon)\|_
 {\widetilde L^2_T(\dot B^{\sigma_1+1}_{2,\infty})}
 +\|f^\varepsilon\|_
 {\widetilde L^2_T(\dot{\mathcal B}^{\sigma_1+1}_{2,\infty})}
\nonumber\\
&\qquad\lesssim
 \mathcal N_{\sigma_1}^\varepsilon(T)
 +\|(\varrho^\varepsilon,u^\varepsilon,f^\varepsilon)\|_{\mathcal Z_T},
\label{e1561} \\
&\|u^\varepsilon-V^\varepsilon\|_
 {\widetilde L^2_T(\dot B^{\sigma_1}_{2,\infty})}
 +\|G^\varepsilon\|_
 {\widetilde L^2_T(\dot{\mathcal B}^{\sigma_1}_{2,\infty,\nu})}\lesssim
 \mathcal N_{\sigma_1}^\varepsilon(T)
 +\|(\varrho^\varepsilon,u^\varepsilon,f^\varepsilon)\|_{\mathcal Z_T}.
\label{e156}
\end{align}

Then, we analyze the nonlinear term in \eqref{e152}. By the definitions of $V_0^\varepsilon$ and $G_0^\varepsilon$, the
product estimate \eqref{uv3} and the composition estimate in Lemma
\ref{l12} yield
\[
\begin{aligned}
&\|(a_0^{\varepsilon},V_0^\varepsilon)\|_{\dot B^{\sigma_1}_{2,\infty}}
+\|G_0^\varepsilon\|_{\dot{\mathcal B}^{\sigma_1}_{2,\infty}}\lesssim
\|f_0^\varepsilon\|_{\dot{\mathcal B}^{\sigma_1}_{2,\infty}}
+\|(a_0^\varepsilon,b_0^\varepsilon)\|_{\dot B^{d/2}_{2,1}}
 \|(a_0^\varepsilon,b_0^\varepsilon)\|_{\dot B^{\sigma_1}_{2,\infty}}\leq (1+\delta_1)K_0.
\end{aligned}
\]
By Minkowski's inequality, one has $\|z\|_{\widetilde L^1_T(\dot B^s_{2,\infty})}
\leq
\|z\|_{L^1_T(\dot B^s_{2,\infty})}$, and the same inequality holds for the $L^2_v$-valued Besov spaces.
By \eqref{e155}--\eqref{e156}, \eqref{uv3} and H\"older's inequality
in time applied to each dyadic block, we obtain
\begin{align*}
\|(a^\varepsilon(u^\varepsilon-V^\varepsilon))^\ell\|_
 {\widetilde L^1_T(\dot B^{\sigma_1}_{2,\infty})}
&\lesssim
\|a^\varepsilon\|_{L^2_T
 (\dot B^{d/2}_{2,1})}
\|u^\varepsilon-V^\varepsilon\|_
 {\widetilde L^2_T(\dot B^{\sigma_1}_{2,\infty})}
\\
&\lesssim
\|(\varrho^\varepsilon,u^\varepsilon,f^\varepsilon)\|_{\mathcal Z_T}
\Big(
\mathcal N_{\sigma_1}^\varepsilon(T)
+\|(\varrho^\varepsilon,u^\varepsilon,f^\varepsilon)\|_{\mathcal Z_T}
\Big).
\end{align*}
Similarly, we use the $\widetilde L^2_T\dot B^{\sigma_1+1}_{2,\infty}$
bound for one factor and the
$L^2_T\dot B^{d/2}_{2,1}$ bound for the other to obtain
\begin{align*}
&\|(\varrho^\varepsilon u^\varepsilon)^\ell\|_
 {\widetilde L^1_T(\dot B^{\sigma_1+1}_{2,\infty})}
+\|(a^\varepsilon V^\varepsilon)^\ell\|_
 {\widetilde L^1_T(\dot B^{\sigma_1+1}_{2,\infty})}
\\
&\quad\lesssim
\|\varrho^\varepsilon\|_
 {\widetilde L^2_T(\dot B^{\sigma_1+1}_{2,\infty})}
\|u^\varepsilon\|_{L^2_T(\dot B^{d/2}_{2,1})}
+\|a^\varepsilon\|_
 {\widetilde L^2_T(\dot B^{\sigma_1+1}_{2,\infty})}
\|V^\varepsilon\|_{L^2_T(\dot B^{d/2}_{2,1})}
\\
&\quad\lesssim
\|(\varrho^\varepsilon,u^\varepsilon,f^\varepsilon)\|_{\mathcal Z_T}
\Big(
\mathcal N_{\sigma_1}^\varepsilon(T)
+\|(\varrho^\varepsilon,u^\varepsilon,f^\varepsilon)\|_{\mathcal Z_T}
\Big),\\
&\|(u^\varepsilon\cdot\nabla_xu^\varepsilon)^\ell\|_
 {\widetilde L^1_T(\dot B^{\sigma_1}_{2,\infty})}
+\|(b^\varepsilon\otimes b^\varepsilon)^\ell\|_
 {\widetilde L^1_T(\dot B^{\sigma_1+1}_{2,\infty})}
+\|h(\varrho^\varepsilon,u^\varepsilon,f^\varepsilon)^\ell\|_
 {\widetilde L^1_T(\dot{\mathcal B}^{\sigma_1}_{2,\infty})}
\\
&\quad\lesssim
\|u^\varepsilon\|_{L^2_T(\dot B^{d/2}_{2,1})}
\|u^\varepsilon\|_
 {\widetilde L^2_T(\dot B^{\sigma_1+1}_{2,\infty})}
+\|b^\varepsilon\|_{L^2_T(\dot B^{d/2}_{2,1})}
\|b^\varepsilon\|_
 {\widetilde L^2_T(\dot B^{\sigma_1+1}_{2,\infty})}
\\
&\qquad
+\|u^\varepsilon\|_{L^2_T(\dot B^{d/2}_{2,1})}
\|\{\mathbf I-\mathbf P\}f^\varepsilon\|_
 {\widetilde L^2_T
  (\dot{\mathcal B}^{\sigma_1}_{2,\infty,\nu})}
\\
&\quad\lesssim
\|(\varrho^\varepsilon,u^\varepsilon,f^\varepsilon)\|_{\mathcal Z_T}
\Big(
\mathcal N_{\sigma_1}^\varepsilon(T)
+\|(\varrho^\varepsilon,u^\varepsilon,f^\varepsilon)\|_{\mathcal Z_T}
\Big).
\end{align*}
Estimates \eqref{e155}--\eqref{e156} handle the differentiated
quadratic terms, while the Cauchy--Schwarz inequality in time and the
microscopic dissipation handle
$h(\varrho^\varepsilon,u^\varepsilon,f^\varepsilon)$.

By the composition estimate in Lemma \ref{l12},
\eqref{e155}, \eqref{e156}
and the product estimate \eqref{uv3}, we obtain
\begin{equation*}
\begin{aligned}
&\left\|
 \frac{a^\varepsilon}{1+a^\varepsilon}
 \bigl(\nabla a^\varepsilon
 +\dive\Theta(\{\mathbf I-\mathbf P\}f^\varepsilon)\bigr)
 \right\|_
 {\widetilde L^1_T(\dot B^{\sigma_1}_{2,\infty})}
+\left\|
 \frac{V^\varepsilon}{1+a^\varepsilon}
 \dive b^\varepsilon
 \right\|_
 {\widetilde L^1_T(\dot B^{\sigma_1}_{2,\infty})}\\
&\qquad\lesssim
 \|(\varrho^\varepsilon,u^\varepsilon,f^\varepsilon)\|_{\mathcal Z_T}
 \Big(
 \mathcal N_{\sigma_1}^\varepsilon(T)
 +\|(\varrho^\varepsilon,u^\varepsilon,f^\varepsilon)\|_{\mathcal Z_T}
 \Big).
\end{aligned}
\end{equation*}

Estimates \eqref{e155}--\eqref{e156}, \eqref{uv3} and Lemma
\ref{l12} also control the pressure and drag terms in
$g(\varrho^\varepsilon,u^\varepsilon,f^\varepsilon)$. Indeed, since
$u^\varepsilon-b^\varepsilon
 =u^\varepsilon-V^\varepsilon-a^\varepsilon V^\varepsilon$, it follows from \eqref{e154} and
\eqref{e156} that
\[
 \|u^\varepsilon-b^\varepsilon\|_
 {\widetilde L^2_T(\dot B^{\sigma_1}_{2,\infty})}
 \lesssim
 \mathcal N_{\sigma_1}^\varepsilon(T)
 +\|(\varrho^\varepsilon,u^\varepsilon,f^\varepsilon)\|_{\mathcal Z_T}
 +\|(\varrho^\varepsilon,u^\varepsilon,f^\varepsilon)\|_{\mathcal Z_T}
  \mathcal N_{\sigma_1}^\varepsilon(T).
\]
Moreover, the identities
\[
 b^\varepsilon=(1+a^\varepsilon)V^\varepsilon\quad\text{and}\quad 
 \{\mathbf I-\mathbf P\}f^\varepsilon
 =G^\varepsilon+\frac12(v\otimes v-\mathrm{Id}):
 (b^\varepsilon\otimes b^\varepsilon)M^{1/2}
\]
allow us to recover the low-frequency norms of $b^\varepsilon$ and
$\{\mathbf I-\mathbf P\}f^\varepsilon$ from those of
$V^\varepsilon$ and $G^\varepsilon$. The differentiated and cubic
terms in \eqref{e80} are then controlled by the
same product estimates.

For the viscosity term, by \eqref{e8} and the product
estimate \eqref{uv3}, we have
\begin{equation*}
\begin{aligned}
&\varepsilon\left\|
 \frac{\varrho^\varepsilon}{1+\varrho^\varepsilon}
 \bigl(\Delta_xu^\varepsilon
 +2\nabla_x\dive_xu^\varepsilon\bigr)\right\|_
 {\widetilde L^1_T(\dot B^{\sigma_1}_{2,\infty})}^{\ell}\\
&\qquad\lesssim
 \left\|\frac{\varrho^\varepsilon}
 {1+\varrho^\varepsilon}\right\|_
 {\widetilde L^\infty_T(\dot B^{\sigma_1}_{2,\infty})}
 \varepsilon
 \|\Delta_xu^\varepsilon
 +2\nabla_x\dive_xu^\varepsilon\|_
 {L^1_T(\dot B^{d/2}_{2,1})}\\
&\qquad\lesssim
 \|(\varrho^\varepsilon,u^\varepsilon,f^\varepsilon)\|_{\mathcal Z_T}
 \Big(
 \mathcal N_{\sigma_1}^\varepsilon(T)
 +\|(\varrho^\varepsilon,u^\varepsilon,f^\varepsilon)\|_{\mathcal Z_T}
 \Big).
\end{aligned}
\end{equation*}
The cubic terms contain one additional quantity bounded by
\(\|(\varrho^\varepsilon,u^\varepsilon,f^\varepsilon)\|_{\mathcal Z_T}
\le C\delta_1\le1\). The preceding estimates imply
\begin{equation}\label{e162}
 \sup_{j\leq1}2^{j\sigma_1}
 \|\mathcal R_{\ell,j}^\varepsilon\|_{L^1_T}
 \lesssim
 \|(\varrho^\varepsilon,u^\varepsilon,f^\varepsilon)\|_{\mathcal Z_T}
 \mathcal N_{\sigma_1}^\varepsilon(T)
 +\bigl(
 \|(\varrho^\varepsilon,u^\varepsilon,f^\varepsilon)\|_{\mathcal Z_T}
 \bigr)^2.
\end{equation}
Substituting \eqref{e162} into
\eqref{e152} and using
\(\|(\varrho^\varepsilon,u^\varepsilon,f^\varepsilon)\|_{\mathcal Z_T}
\le C\|(\varrho_0,u_0,f_0)\|_{\mathcal Z_0}\le C\delta_1\), we get
\begin{align*}
\mathcal N_{\sigma_1}^\varepsilon(T)
&\le C K_0
+C\delta_1\mathcal N_{\sigma_1}^\varepsilon(T)
+C\delta_1\|(\varrho_0,u_0,f_0)\|_{\mathcal Z_0},
\end{align*}
Together with \(\delta_1\ll 1\) and \eqref{e1561}, this proves \eqref{e151}.
\end{proof}

\subsection{Time-weighted estimate}

 We will close the estimate for the weighted functional
\begin{align}
\mathcal D_\alpha^\varepsilon(T):={}&
 \|\langle t\rangle^\alpha
  (\varrho^{\varepsilon,\ell},u^{\varepsilon,\ell})\|_
 {\widetilde L^\infty_T(\dot B^{d/2+1}_{2,1})}
 +\|\langle t\rangle^\alpha f^{\varepsilon,\ell}\|_
 {\widetilde L^\infty_T(\dot{\mathcal B}^{d/2+1}_{2,1})}
\nonumber\\
&+\|\langle t\rangle^\alpha
  (\varrho^{\varepsilon,h},u^{\varepsilon,h})\|_
 {\widetilde L^\infty_T(\dot B^{d/2+1}_{2,1})}
 +\|\langle t\rangle^\alpha f^{\varepsilon,h}\|_
 {\widetilde L^\infty_T(\dot{\mathcal B}^{d/2+1}_{2,1})}
\nonumber\\
&+\|\langle t\rangle^\alpha
  (\varrho^{\varepsilon,\ell},u^{\varepsilon,\ell})\|_
 {L^1_T(\dot B^{d/2+3}_{2,1})}
 +\|\langle t\rangle^\alpha f^{\varepsilon,\ell}\|_
 {L^1_T(\dot{\mathcal B}^{d/2+3}_{2,1})}
\nonumber\\
&+\|\langle t\rangle^\alpha
  (\varrho^{\varepsilon,h},u^{\varepsilon,h})\|_
 {L^1_T(\dot B^{d/2+1}_{2,1})}
+\|\langle t\rangle^\alpha f^{\varepsilon,h}\|_
 {L^1_T(\dot{\mathcal B}^{d/2+1}_{2,1})}
\nonumber\\
&+\|\langle t\rangle^\alpha
 (u^\varepsilon-V^\varepsilon)^\ell\|_
 {\widetilde L^2_T(\dot B^{d/2+1}_{2,1})}
 +\|\langle t\rangle^\alpha
 (u^\varepsilon-b^\varepsilon)^h\|_
 {\widetilde L^2_T(\dot B^{d/2+1}_{2,1})}
\nonumber\\
&+\|\langle t\rangle^\alpha
 G^{\varepsilon,\ell}\|_
 {\widetilde L^2_T(\dot{\mathcal B}^{d/2+1}_{2,1,\nu})}
+\|\langle t\rangle^\alpha
 \{\mathbf I-\mathbf P\}f^{\varepsilon,h}\|_
 {\widetilde L^2_T(\dot{\mathcal B}^{d/2+1}_{2,1,\nu})}
\nonumber\\
&+\varepsilon\|\langle t\rangle^\alpha
 u^{\varepsilon,h}\|_
 {L^1_T(\dot B^{d/2+2}_{2,1})},
\label{e163}
\end{align}
with $\alpha>\frac{1}{2}(\frac{d}{2}+1-\sigma_1)>1$ and
$\langle t\rangle:=1+t$.

It should be emphasized that, compared with the low-frequency index $d/2-1$ in 
\eqref{e90} used in the global existence argument, the low-frequency
regularity in \eqref{e163} is raised to $d/2+1$ in order to establish
the highest-order decay estimate. Accordingly, the corresponding
dissipation is measured at regularity $d/2+3$. Because this changes
the regularities used at low and high frequencies, new nonlinear
estimates are needed to close the weighted energy estimate.

\begin{lemma}\label{l9}
For every \(\eta>0\), there exists \(C_\eta>0\) such that
\begin{equation}
\begin{aligned}
\mathcal D_\alpha^\varepsilon(T)
&\lesssim
 C_\eta
 K_0
 \langle T\rangle^{\alpha-\frac{1}{2}(\frac{d}{2}+1-\sigma_1)}+\Bigl(\eta+(\|(\varrho^\varepsilon,u^\varepsilon,f^\varepsilon)\|_{\mathcal Z_T})^{1/2}\Bigr)
 \mathcal D_\alpha^\varepsilon(T).
\end{aligned}
\label{e164}
\end{equation}
\end{lemma}

\begin{proof}
We first consider the low frequencies. We divide
\eqref{e79} by
\(2\sqrt{\mathcal L_{\ell,j}^\varepsilon+\kappa^2}\),
multiply the resulting inequality by \(\langle t\rangle^\alpha\),
and then let \(\kappa\to0\) to obtain
\begin{equation}\label{e165}
\frac d{dt}\Big(
 \langle t\rangle^\alpha
 \sqrt{\mathcal L_{\ell,j}^\varepsilon}\Big)
 +c2^{2j}\langle t\rangle^\alpha
 \sqrt{\mathcal L_{\ell,j}^\varepsilon}
\lesssim
 \alpha\langle t\rangle^{\alpha-1}
 \sqrt{\mathcal L_{\ell,j}^\varepsilon}
+\langle t\rangle^\alpha
 \mathcal R_{\ell,j}^\varepsilon.
\end{equation}
Integrating \eqref{e165} over $(0,T)$, multiplying by
$2^{j(d/2+1)}$ and summing over $j\leq1$, we obtain
\begin{equation}\label{lownon}
\begin{aligned}
&\|\langle t\rangle^\alpha
(\varrho^{\varepsilon,\ell},u^{\varepsilon,\ell})\|_
{\widetilde L^\infty_T(\dot B^{d/2+1}_{2,1})}
+\|\langle t\rangle^\alpha f^{\varepsilon,\ell}\|_
{\widetilde L^\infty_T(\dot{\mathcal B}^{d/2+1}_{2,1})}
\\
&\quad
+\|\langle t\rangle^\alpha
(\varrho^{\varepsilon,\ell},u^{\varepsilon,\ell})\|_
{L^1_T(\dot B^{d/2+3}_{2,1})}
+\|\langle t\rangle^\alpha f^{\varepsilon,\ell}\|_
{L^1_T(\dot{\mathcal B}^{d/2+3}_{2,1})}
\\
&\lesssim
\|(\varrho^\varepsilon_0,u^\varepsilon_0,f^\varepsilon_0)\|_
{\mathcal Z_0}
+\alpha\int_0^T\langle t\rangle^{\alpha-1}
\Big(
\|(\varrho^{\varepsilon,\ell},u^{\varepsilon,\ell})(t)\|_
{\dot B^{d/2+1}_{2,1}}
+\|f^{\varepsilon,\ell}(t)\|_
{\dot{\mathcal B}^{d/2+1}_{2,1}}
\Big)\,dt
\\
&\qquad
+\sum_{j\leq1}2^{j(d/2+1)}
\|\langle t\rangle^\alpha
\mathcal R_{\ell,j}^\varepsilon\|_{L^1_T}.
\end{aligned}
\end{equation}
Here we used \eqref{e78}, the moment bounds and the definitions of
$V^\varepsilon$ and $G^\varepsilon$.
For the term arising from the time weight, applying interpolation in Lemma \ref{l10} yields
\begin{equation}\label{e166}
\|f^{\varepsilon,\ell}(t)\|_{\dot{\mathcal B}^{d/2+1}_{2,1}}\lesssim
\|f^{\varepsilon,\ell}(t)\|_{\dot{\mathcal B}^{\sigma_1}_{2,\infty}}^{\frac{2}{d/2+3-\sigma_1}}
\|f^{\varepsilon,\ell}(t)\|_{\dot{\mathcal B}^{d/2+3}_{2,1}}^{\frac{d/2+1-\sigma_1}{d/2+3-\sigma_1}}.
\end{equation}
Since the first term on the right-hand
side is bounded by \(\mathcal N_{\sigma_1}^\varepsilon(T)\), inserting
the time weight into \eqref{e166}, we
obtain
\begin{align*}
&\langle t\rangle^{\alpha-1}
\|f^{\varepsilon,\ell}(t)\|_
{\dot{\mathcal B}^{d/2+1}_{2,1}}\lesssim
\bigl(
\langle t\rangle^\alpha
\|f^{\varepsilon,\ell}(t)\|_
{\dot{\mathcal B}^{d/2+3}_{2,1}}
\bigr)^{
\frac{d/2+1-\sigma_1}
{d/2+3-\sigma_1}}
\bigl(\mathcal N_{\sigma_1}^\varepsilon(T)\bigr)^{
\frac{2}{d/2+3-\sigma_1}}
\langle t\rangle^{
\frac{2\alpha-\frac{d}{2}-3+\sigma_1}
{\frac{d}{2}+3-\sigma_1}},
\end{align*}
which, together with Young's inequality, implies, for every \(\eta>0\),
\begin{align*}
\langle t\rangle^{\alpha-1}
\|f^{\varepsilon,\ell}(t)\|_
{\dot{\mathcal B}^{d/2+1}_{2,1}}\le
\eta\langle t\rangle^\alpha
\|f^{\varepsilon,\ell}(t)\|_
{\dot{\mathcal B}^{d/2+3}_{2,1}}
+C_\eta
\mathcal N_{\sigma_1}^\varepsilon(T)
\langle t\rangle^{
\alpha-\frac{1}{2}(\frac{d}{2}+1-\sigma_1)-1}.
\end{align*}
Since \(\alpha>\frac{1}{2}(\frac{d}{2}+1-\sigma_1)\), integrating
the above inequality and applying the same argument to
\(\varrho^{\varepsilon,\ell}\) and \(u^{\varepsilon,\ell}\), we obtain
\begin{align}
&\int_0^T\langle t\rangle^{\alpha-1}
\Big(
\|(\varrho^{\varepsilon,\ell},u^{\varepsilon,\ell})(t)\|_
{\dot B^{d/2+1}_{2,1}}
+\|f^{\varepsilon,\ell}(t)\|_
{\dot{\mathcal B}^{d/2+1}_{2,1}}
\Big)\,dt
\nonumber\\
&\quad\lesssim
C_\eta
\mathcal N_{\sigma_1}^\varepsilon(T)
\langle T\rangle^{
\alpha-\frac{1}{2}(\frac{d}{2}+1-\sigma_1)}
+\eta\mathcal D_\alpha^\varepsilon(T).
\label{e169}
\end{align}

It remains to estimate the nonlinear term in \eqref{lownon}.
By \eqref{e155}, the definitions of $\mathcal Z_T$ and
$\mathcal D_\alpha^\varepsilon(T)$, and time interpolation, we have
\begin{align}
&\|(\varrho^{\varepsilon,\ell},u^{\varepsilon,\ell})\|_
{L^2_T(\dot B^{d/2}_{2,1})}
+\|f^{\varepsilon,\ell}\|_
{L^2_T(\dot{\mathcal B}^{d/2}_{2,1})}
\nonumber\\
&\quad
+\|(\varrho^{\varepsilon,h},u^{\varepsilon,h})\|_
{L^2_T(\dot B^{d/2+1}_{2,1})}
+\|f^{\varepsilon,h}\|_
{L^2_T(\dot{\mathcal B}^{d/2+1}_{2,1})}
\lesssim
\|(\varrho^\varepsilon,u^\varepsilon,f^\varepsilon)\|_{\mathcal Z_T},
\label{e171}\\
&\|\langle t\rangle^\alpha
(\varrho^{\varepsilon,\ell},u^{\varepsilon,\ell})\|_
{\widetilde L^2_T(\dot B^{d/2+2}_{2,1})}
+\|\langle t\rangle^\alpha f^{\varepsilon,\ell}\|_
{\widetilde L^2_T(\dot{\mathcal B}^{d/2+2}_{2,1})}
\nonumber\\
&\quad
+\|\langle t\rangle^\alpha
(\varrho^{\varepsilon,h},u^{\varepsilon,h})\|_
{\widetilde L^2_T(\dot B^{d/2+1}_{2,1})}
+\|\langle t\rangle^\alpha f^{\varepsilon,h}\|_
{\widetilde L^2_T(\dot{\mathcal B}^{d/2+1}_{2,1})}
\lesssim\mathcal D_\alpha^\varepsilon(T).
\label{e172}
\end{align}
The moment bounds, Lemmas \ref{l11}--\ref{l12} and the fixed
low-frequency localization show that the estimates in
\eqref{e171}--\eqref{e172} also hold for
$a^\varepsilon$, $b^\varepsilon$ and
$V^\varepsilon=b^\varepsilon/(1+a^\varepsilon)$.

By \eqref{e171}--\eqref{e172} and \eqref{uv1}, we get
\[
\begin{aligned}
\|\langle t\rangle^\alpha(\varrho^\varepsilon u^\varepsilon)^\ell\|_
{L^1_T(\dot B^{d/2+2}_{2,1})}&\lesssim
\|\varrho^\varepsilon\|_{L^2_T(\dot B^{d/2}_{2,1})}
\Big(\|\langle t\rangle^\alpha u^{\varepsilon,\ell}\|_
{L^2_T(\dot B^{d/2+2}_{2,1})}
+\|\langle t\rangle^\alpha u^{\varepsilon,h}\|_
{L^2_T(\dot B^{d/2+1}_{2,1})}\Big)\\
&\qquad+
\|u^\varepsilon\|_{L^2_T(\dot B^{d/2}_{2,1})}
\Big(\|\langle t\rangle^\alpha\varrho^{\varepsilon,\ell}\|_
{L^2_T(\dot B^{d/2+2}_{2,1})}
+\|\langle t\rangle^\alpha\varrho^{\varepsilon,h}\|_
{L^2_T(\dot B^{d/2+1}_{2,1})}\Big).
\end{aligned}
\]
Applying the same estimate to $a^\varepsilon V^\varepsilon$ and
$b^\varepsilon\otimes b^\varepsilon$, we obtain
\begin{align*}
&\|\langle t\rangle^\alpha
 (\varrho^\varepsilon u^\varepsilon)^\ell\|_
 {L^1_T(\dot B^{d/2+2}_{2,1})}
+\|\langle t\rangle^\alpha
 (a^\varepsilon V^\varepsilon)^\ell\|_
 {L^1_T(\dot B^{d/2+2}_{2,1})}
\\
&\quad
+\|\langle t\rangle^\alpha
 (b^\varepsilon\otimes b^\varepsilon)^\ell\|_
 {L^1_T(\dot B^{d/2+2}_{2,1})}
\lesssim
\|(\varrho^\varepsilon,u^\varepsilon,f^\varepsilon)\|_{\mathcal Z_T}\mathcal D_\alpha^\varepsilon(T).
\end{align*}

For the convection term, we first write
\[
(u^\varepsilon\cdot\nabla u^\varepsilon)^\ell
=\big(u^{\varepsilon,\ell}\cdot\nabla u^{\varepsilon,\ell}
+u^{\varepsilon,h}\cdot\nabla u^{\varepsilon,\ell}
+u^{\varepsilon,\ell}\cdot\nabla u^{\varepsilon,h}
+u^{\varepsilon,h}\cdot\nabla u^{\varepsilon,h}\big)^\ell.
\]
By \eqref{uv1} and the fixed low-frequency localization, one has
\begin{align*}
&\|\langle t\rangle^\alpha
(u^{\varepsilon,\ell}\cdot\nabla u^{\varepsilon,\ell})^\ell\|_
{L^1_T(\dot B^{d/2+1}_{2,1})}\\
&\quad\lesssim
\|u^{\varepsilon,\ell}\|_{L^2_T(\dot B^{d/2}_{2,1})}
\|\langle t\rangle^\alpha u^{\varepsilon,\ell}\|_{L^2_T(\dot B^{d/2+2}_{2,1})}
+\|u^{\varepsilon,\ell}\|_{L^1_T(\dot B^{d/2+1}_{2,1})}
\|\langle t\rangle^\alpha u^{\varepsilon,\ell}\|_{L^\infty_T(\dot B^{d/2+1}_{2,1})},\\
&\|\langle t\rangle^\alpha
(u^{\varepsilon,h}\cdot\nabla u^{\varepsilon,\ell})^\ell\|_
{L^1_T(\dot B^{d/2+1}_{2,1})}\\
&\quad\lesssim
\|u^{\varepsilon,h}\|_{L^2_T(\dot B^{d/2+1}_{2,1})}
\|\langle t\rangle^\alpha u^{\varepsilon,\ell}\|_{L^2_T(\dot B^{d/2+2}_{2,1})}
+\|u^{\varepsilon,\ell}\|_{L^1_T(\dot B^{d/2+1}_{2,1})}
\|\langle t\rangle^\alpha u^{\varepsilon,h}\|_{L^\infty_T(\dot B^{d/2+1}_{2,1})},\\
&\|\langle t\rangle^\alpha
\big(u^{\varepsilon,\ell}\cdot\nabla u^{\varepsilon,h}
+u^{\varepsilon,h}\cdot\nabla u^{\varepsilon,h}\big)^\ell\|_
{L^1_T(\dot B^{d/2+1}_{2,1})}\\
&\quad\lesssim
\|\langle t\rangle^\alpha
\big(u^{\varepsilon,\ell}\cdot\nabla u^{\varepsilon,h}
+u^{\varepsilon,h}\cdot\nabla u^{\varepsilon,h}\big)\|_
{L^1_T(\dot B^{d/2}_{2,1})}\\
&\quad\lesssim
\Big(\|u^{\varepsilon,\ell}\|_{L^2_T(\dot B^{d/2}_{2,1})}
+\|u^{\varepsilon,h}\|_{L^2_T(\dot B^{d/2+1}_{2,1})}\Big)
\|\langle t\rangle^\alpha u^{\varepsilon,h}\|_{L^2_T(\dot B^{d/2+1}_{2,1})}.
\end{align*}
The preceding estimates therefore imply
\begin{equation*}
\|\langle t\rangle^\alpha u^\varepsilon\cdot\nabla u^\varepsilon\|_{L^1_T(\dot B^{d/2+1}_{2,1})}^{\ell}\lesssim\|(\varrho^\varepsilon,u^\varepsilon,f^\varepsilon)\|_{\mathcal Z_T}\mathcal D_\alpha^\varepsilon(T).
\end{equation*}
Using
$u^\varepsilon-V^\varepsilon=u^\varepsilon-b^\varepsilon
+a^\varepsilon V^\varepsilon$, the product estimate and
\eqref{e171}--\eqref{e172} imply
\[
\|\langle t\rangle^\alpha(u^\varepsilon-V^\varepsilon)\|_
{L^2_T(\dot B^{d/2+1}_{2,1})}
\lesssim
\bigl(1+\|(\varrho^\varepsilon,u^\varepsilon,f^\varepsilon)\|_{\mathcal Z_T}\bigr)
\mathcal D_\alpha^\varepsilon(T).
\]
Applying \eqref{uv1}, one obtains
\begin{align}
\|\langle t\rangle^\alpha
 (a^\varepsilon(u^\varepsilon-V^\varepsilon))^\ell\|_
 {L^1_T(\dot B^{d/2+1}_{2,1})}&\lesssim
\|a^\varepsilon\|_{L^2_T(\dot B^{d/2}_{2,1})}
\|\langle t\rangle^\alpha(u^\varepsilon-V^\varepsilon)\|_
{L^2_T(\dot B^{d/2+1}_{2,1})}
\nonumber\\
&\qquad+
\|u^\varepsilon-V^\varepsilon\|_
{L^2_T(\dot B^{d/2-1}_{2,1})}
\|\langle t\rangle^\alpha a^{\varepsilon,\ell}\|_
{L^2_T(\dot B^{d/2+2}_{2,1})}
\nonumber\\
&\qquad+
\|u^\varepsilon-V^\varepsilon\|_
{L^2_T(\dot B^{d/2}_{2,1})}
\|\langle t\rangle^\alpha a^{\varepsilon,h}\|_
{L^2_T(\dot B^{d/2+1}_{2,1})}
\nonumber\\
&\quad
\lesssim
\|(\varrho^\varepsilon,u^\varepsilon,f^\varepsilon)\|_{\mathcal Z_T}\mathcal D_\alpha^\varepsilon(T).
\label{e178}
\end{align}
The definition of $G^\varepsilon$ and the product estimate first yield
\[
\|\langle t\rangle^\alpha
\{\mathbf I-\mathbf P\}f^\varepsilon\|_
{L^2_T(\dot{\mathcal B}^{d/2+1}_{2,1,\nu})}
\lesssim
\bigl(1+\|(\varrho^\varepsilon,u^\varepsilon,f^\varepsilon)\|_{\mathcal Z_T}\bigr)
\mathcal D_\alpha^\varepsilon(T).
\]
The definition of $L^2_{v,\nu}$ and \eqref{uv1} then give
\begin{align}
\|\langle t\rangle^\alpha
h(\varrho^\varepsilon,u^\varepsilon,f^\varepsilon)\|_
 {L^1_T(\dot{\mathcal B}^{d/2+1}_{2,1})}^{\ell}&\lesssim
\|u^\varepsilon\|_{L^2_T(\dot B^{d/2}_{2,1})}
\|\langle t\rangle^\alpha\{\mathbf I-\mathbf P\}f^\varepsilon\|_
{L^2_T(\dot{\mathcal B}^{d/2+1}_{2,1,\nu})}
\nonumber\\
&\qquad+
\|\{\mathbf I-\mathbf P\}f^\varepsilon\|_
{L^2_T(\dot{\mathcal B}^{d/2-1}_{2,1,\nu})}
\|\langle t\rangle^\alpha u^{\varepsilon,\ell}\|_
{L^2_T(\dot B^{d/2+2}_{2,1})}
\nonumber\\
&\qquad+
\|\{\mathbf I-\mathbf P\}f^\varepsilon\|_
{L^2_T(\dot{\mathcal B}^{d/2}_{2,1,\nu})}
\|\langle t\rangle^\alpha u^{\varepsilon,h}\|_
{L^2_T(\dot B^{d/2+1}_{2,1})}
\nonumber\\
&\quad
\lesssim
\|(\varrho^\varepsilon,u^\varepsilon,f^\varepsilon)\|_{\mathcal Z_T}\mathcal D_\alpha^\varepsilon(T).
\label{e180}
\end{align}

For the variable coefficients, the composition estimate in Lemma
\ref{l12} yields, at every index used below,
\begin{align*}
\left\|\frac{a^\varepsilon}
 {1+a^\varepsilon}\right\|_
 {\dot B^r_{2,1}}
&\lesssim\|a^\varepsilon\|_{\dot B^r_{2,1}}
\quad\text{and}\quad
\left\|\frac{V^\varepsilon}
 {1+a^\varepsilon}\right\|_
 {\dot B^r_{2,1}}
\lesssim
\|V^\varepsilon\|_{\dot B^r_{2,1}}
+\|V^\varepsilon\|_{L^\infty_x}
\|a^\varepsilon\|_{\dot B^r_{2,1}},
\end{align*}
and hence
\begin{align*}
&\left\|\langle t\rangle^\alpha
 \frac{a^\varepsilon}{1+a^\varepsilon}
 \Big(\nabla a^\varepsilon+
 \dive\Theta(\{\mathbf I-\mathbf P\}f^\varepsilon)\Big)\right\|_
 {L^1_T(\dot B^{d/2+1}_{2,1})}^{\ell}
\\
&\quad+
\left\|\langle t\rangle^\alpha
 \frac{V^\varepsilon}{1+a^\varepsilon}
 \dive b^\varepsilon\right\|_
 {L^1_T(\dot B^{d/2+1}_{2,1})}^{\ell}
\lesssim
\|(\varrho^\varepsilon,u^\varepsilon,f^\varepsilon)\|_{\mathcal Z_T}\mathcal D_\alpha^\varepsilon(T).
\end{align*}
For the terms involving high frequencies, the low-frequency
localization allows us to estimate one order lower. Hence, by
\eqref{uv1} and the weighted
$L^2_T(\dot B^{d/2+2}_{2,1})$ estimate, we obtain
\begin{align*}
&\|\langle t\rangle^\alpha
 (a^\varepsilon u^\varepsilon\otimes b^\varepsilon)^\ell\|_
 {L^1_T(\dot B^{d/2+1}_{2,1})}
\\
&\quad+
\left\|\langle t\rangle^\alpha\left(
 \bigl(\nabla a^\varepsilon+
 \dive\Theta(\{\mathbf I-\mathbf P\}f^\varepsilon)\bigr)
 \otimes b^\varepsilon-\nabla(a^\varepsilon V^\varepsilon)
 \right)^\ell\right\|_{L^1_T(\dot B^{d/2+1}_{2,1})}\\
 &\quad\quad\lesssim
\|(\varrho^\varepsilon,u^\varepsilon,f^\varepsilon)\|_{\mathcal Z_T}
\mathcal D_\alpha^\varepsilon(T).
\end{align*}
Set \(q(z):=P'(1+z)/(1+z)\). For \(s=d/2,d/2+1\), the composition
estimate in Lemma \ref{l12} gives
\begin{equation*}
\|q(\varrho^\varepsilon)-1\|_{\dot B^s_{2,1}}
+\left\|\frac{\varrho^\varepsilon}
{1+\varrho^\varepsilon}\right\|_{\dot B^s_{2,1}}
\lesssim
\|\varrho^\varepsilon\|_{\dot B^s_{2,1}},
\end{equation*}
Together with \eqref{uv1} and the same low-frequency estimate, this
implies
\begin{equation*}
 \|\langle t\rangle^\alpha
 (q(\varrho^\varepsilon)-1)\nabla\varrho^\varepsilon\|_
 {L^1_T(\dot B^{d/2+1}_{2,1})}^{\ell}
 \lesssim\|(\varrho^\varepsilon,u^\varepsilon,f^\varepsilon)\|_{\mathcal Z_T}
 \mathcal D_\alpha^\varepsilon(T).
\end{equation*}
Moreover, the identity
$u^\varepsilon-b^\varepsilon
=u^\varepsilon-V^\varepsilon-a^\varepsilon V^\varepsilon$,
the definition of $\mathcal D_\alpha^\varepsilon$, the high-frequency
relaxation bound and the product estimate yield
\begin{align}
\|\langle t\rangle^\alpha
 (u^\varepsilon-b^\varepsilon)\|_
 {L^2_T(\dot B^{d/2+1}_{2,1})}
&\lesssim
\mathcal D_\alpha^\varepsilon(T)
+\|\langle t\rangle^\alpha
(a^\varepsilon V^\varepsilon)\|_
 {L^2_T(\dot B^{d/2+1}_{2,1})}
\nonumber\\
&\lesssim
\bigl(1+\|(\varrho^\varepsilon,u^\varepsilon,f^\varepsilon)\|_{\mathcal Z_T}\bigr)
\mathcal D_\alpha^\varepsilon(T).
\label{e186}
\end{align}
Applying the same paraproduct decomposition as in
\eqref{e178}, with the coefficient
$\varrho^\varepsilon/(1+\varrho^\varepsilon)$, gives
\begin{equation*}
\left\|\langle t\rangle^\alpha
\frac{\varrho^\varepsilon}{1+\varrho^\varepsilon}
(u^\varepsilon-b^\varepsilon)\right\|_
{L^1_T(\dot B^{d/2+1}_{2,1})}^{\ell}
\lesssim
\|(\varrho^\varepsilon,u^\varepsilon,f^\varepsilon)\|_{\mathcal Z_T}
\mathcal D_\alpha^\varepsilon(T).
\end{equation*}
We also have
\begin{align*}
\left\|\langle t\rangle^\alpha
 \frac{\varrho^\varepsilon}{1+\varrho^\varepsilon}
 a^\varepsilon u^\varepsilon\right\|_
 {L^1_T(\dot B^{d/2+1}_{2,1})}^{\ell}&\lesssim
\|(\varrho^\varepsilon,a^\varepsilon,u^\varepsilon)\|_
{L^\infty_T(\dot B^{d/2-1}_{2,1})}^{2}
\|\langle t\rangle^\alpha
(\varrho^{\varepsilon,\ell},a^{\varepsilon,\ell},u^{\varepsilon,\ell})\|_
{L^1_T(\dot B^{d/2+3}_{2,1})}
\\
&\qquad+
\|(\varrho^\varepsilon,a^\varepsilon,u^\varepsilon)\|_
{L^\infty_T(\dot B^{d/2}_{2,1})}^{2}
\|\langle t\rangle^\alpha
(\varrho^{\varepsilon,h},a^{\varepsilon,h},u^{\varepsilon,h})\|_
{L^1_T(\dot B^{d/2+1}_{2,1})}
\\
&\quad\lesssim
(\|(\varrho^\varepsilon,u^\varepsilon,f^\varepsilon)\|_{\mathcal Z_T})^2
\mathcal D_\alpha^\varepsilon(T).
\end{align*}
The first inequality follows from Bernstein's inequality at low
frequencies and the critical high-frequency bounds; the reciprocal
coefficient is controlled by Lemma \ref{l12}.
For the variable-viscosity term, we split
$u^\varepsilon=u^{\varepsilon,\ell}+u^{\varepsilon,h}$ and obtain
\begin{align*}
&\varepsilon\left\|\langle t\rangle^\alpha
 \frac{\varrho^\varepsilon}{1+\varrho^\varepsilon}
 \bigl(\Delta_xu^\varepsilon+2\nabla_x\dive_xu^\varepsilon\bigr)\right\|_
 {L^1_T(\dot B^{d/2+1}_{2,1})}^{\ell}
\\
&\quad\lesssim
\|(\varrho^\varepsilon,u^\varepsilon,f^\varepsilon)\|_{\mathcal Z_T}
\Big(\|\langle t\rangle^\alpha u^{\varepsilon,\ell}\|_
 {L^1_T(\dot B^{d/2+3}_{2,1})}
+\varepsilon\|\langle t\rangle^\alpha u^{\varepsilon,h}\|_
 {L^1_T(\dot B^{d/2+2}_{2,1})}\Big)
\\
&\quad\lesssim
\|(\varrho^\varepsilon,u^\varepsilon,f^\varepsilon)\|_{\mathcal Z_T}
\mathcal D_\alpha^\varepsilon(T).
\end{align*}
The preceding estimates and \eqref{e80} yield
\begin{equation}\label{e170}
 \sum_{j\leq1}2^{j(d/2+1)}
 \|\langle t\rangle^\alpha
 \mathcal R_{\ell,j}^\varepsilon\|_{L^1_T}
 \lesssim
 \|(\varrho^\varepsilon,u^\varepsilon,f^\varepsilon)\|_{\mathcal Z_T}
 \mathcal D_\alpha^\varepsilon(T).
\end{equation}
Inserting \eqref{e169} and \eqref{e170} into the summed form of
\eqref{e165}, we obtain
\begin{align}
&\|\langle t\rangle^\alpha
 (\varrho^{\varepsilon,\ell},u^{\varepsilon,\ell})\|_
 {\widetilde L^\infty_T(\dot B^{d/2+1}_{2,1})}
+\|\langle t\rangle^\alpha f^{\varepsilon,\ell}\|_
 {\widetilde L^\infty_T(\dot{\mathcal B}^{d/2+1}_{2,1})}
\nonumber\\
&\quad+\|\langle t\rangle^\alpha
 (\varrho^{\varepsilon,\ell},u^{\varepsilon,\ell})\|_
 {L^1_T(\dot B^{d/2+3}_{2,1})}+\|\langle t\rangle^\alpha f^{\varepsilon,\ell}\|_
 {L^1_T(\dot{\mathcal B}^{d/2+3}_{2,1})}
\nonumber\\
&\quad\lesssim
\|(\varrho_0^\varepsilon,u_0^\varepsilon,f_0^\varepsilon)\|_{\mathcal Z_0}
+C_\eta\mathcal N_{\sigma_1}^\varepsilon(T)
 \langle T\rangle^{\alpha-\frac{1}{2}(\frac{d}{2}+1-\sigma_1)}+\bigl(\eta+\|(\varrho^\varepsilon,u^\varepsilon,f^\varepsilon)\|_{\mathcal Z_T}\bigr)
 \mathcal D_\alpha^\varepsilon(T).
\label{e192}
\end{align}

In addition, the weighted \(L^2_T\) terms follow by multiplying
\eqref{e79} by \(\langle t\rangle^{2\alpha}\):
\begin{align}
&\frac d{dt}\bigl(\langle t\rangle^{2\alpha}
 \mathcal L_{\ell,j}^\varepsilon\bigr)
+c\langle t\rangle^{2\alpha}
 \Big(2^{2j}\mathcal L_{\ell,j}^\varepsilon
 +\|\dot\Delta_j(u^\varepsilon-V^\varepsilon)\|_{L^2_x}^2
 +\|\dot\Delta_jG^\varepsilon\|_{L^2_xL^2_{v,\nu}}^2\Big)
\nonumber\\
&\qquad\lesssim
\alpha\langle t\rangle^{2\alpha-1}
 \mathcal L_{\ell,j}^\varepsilon
+\langle t\rangle^{2\alpha}
 \mathcal R_{\ell,j}^\varepsilon
 \sqrt{\mathcal L_{\ell,j}^\varepsilon}.
\label{e190}
\end{align}

Integrating \eqref{e190} and taking the square root, we obtain
\begin{align*}
&\|\langle t\rangle^\alpha
 \dot\Delta_j(u^\varepsilon-V^\varepsilon)\|_{L^2_TL^2_x}
 +\|\langle t\rangle^\alpha
 \dot\Delta_jG^\varepsilon\|_{L^2_TL^2_xL^2_{v,\nu}}
 \\
&\quad\lesssim
 \sqrt{\mathcal L_{\ell,j}^\varepsilon(0)}
 +\Big(
 \|\langle t\rangle^\alpha
 \sqrt{\mathcal L_{\ell,j}^\varepsilon}\|_{L^\infty_T}
 \int_0^T\langle t\rangle^{\alpha-1}
 \sqrt{\mathcal L_{\ell,j}^\varepsilon}\,dt
 \Big)^{1/2}
 \\
&\qquad\quad
 +\left(
 \|\langle t\rangle^\alpha
 \sqrt{\mathcal L_{\ell,j}^\varepsilon}\|_{L^\infty_T}
 \|\langle t\rangle^\alpha
 \mathcal R_{\ell,j}^\varepsilon\|_{L^1_T}
 \right)^{1/2}.
\end{align*}
Multiplying by \(2^{j(d/2+1)}\), summing over \(j\leq1\) and
using the Cauchy--Schwarz inequality in the dyadic variable, we obtain from
\eqref{e78},
\eqref{e169},
\eqref{e170} and
\eqref{e192}
\begin{align}
&\|\langle t\rangle^\alpha
 (u^\varepsilon-V^\varepsilon)^\ell\|_
 {\widetilde L^2_T(\dot B^{d/2+1}_{2,1})}
 +\|\langle t\rangle^\alpha
 G^{\varepsilon,\ell}\|_
 {\widetilde L^2_T(\dot{\mathcal B}^{d/2+1}_{2,1,\nu})}
 \nonumber\\
&\quad\leq
 C_\eta
 \bigl(\|(\varrho^\varepsilon,u^\varepsilon,f^\varepsilon)\|_{\mathcal Z_T}
 +\mathcal N_{\sigma_1}^\varepsilon(T)\bigr)
 \langle T\rangle^{\alpha-\frac{1}{2}(\frac{d}{2}+1-\sigma_1)} \nonumber\\
 &\quad\quad\quad\quad \quad\quad\quad\quad 
 +\bigl(\eta
 +C\|(\varrho^\varepsilon,u^\varepsilon,f^\varepsilon)\|_
 {\mathcal Z_T}^{1/2}\bigr)
 \mathcal D_\alpha^\varepsilon(T).
\label{e194}
\end{align}
Here we also use the moment bounds and the recovery estimates
for \(V^\varepsilon\) and \(G^\varepsilon\).

We next consider the high frequencies. Dividing \eqref{e96} by the
regularized square root of $\mathcal L_{h,j}^\varepsilon$, multiplying
by $\langle t\rangle^\alpha$ and passing to the limit, we obtain
\begin{align}
&\frac d{dt}\Big(
\langle t\rangle^\alpha
\sqrt{\mathcal L_{h,j}^\varepsilon}\Big)
+c\langle t\rangle^\alpha
\sqrt{\mathcal L_{h,j}^\varepsilon}\leq
\alpha\langle t\rangle^{\alpha-1}
\sqrt{\mathcal L_{h,j}^\varepsilon}
+C\langle t\rangle^\alpha
\mathcal R_{h,j}^\varepsilon,
\qquad j\geq0.
\label{e197}
\end{align}
Let $T_\alpha:=4\alpha/c$. For $t\geq T_\alpha$, the first term on
the damping term on the left absorbs the first term on the right.
On $[0,T_\alpha]$, the time weight is bounded, and
\eqref{e112} gives
\begin{align*}
&\sum_{j\geq0}2^{j(d/2+1)}
\int_0^{\min\{T,T_\alpha\}}
\alpha\langle t\rangle^{\alpha-1}
\sqrt{\mathcal L_{h,j}^\varepsilon(t)}\,dt\lesssim
\|(\varrho^\varepsilon,u^\varepsilon,f^\varepsilon)\|_{\mathcal Z_T}.
\end{align*}
Hence, integrating \eqref{e197} separately on
$[0,\min\{T,T_\alpha\}]$ and, if $T>T_\alpha$, on
$[T_\alpha,T]$, we obtain
\begin{align}
&\sum_{j\geq0}2^{j(d/2+1)}
\Big(
\|\langle t\rangle^\alpha
\sqrt{\mathcal L_{h,j}^\varepsilon}\|_{L^\infty_T}
+\|\langle t\rangle^\alpha
\sqrt{\mathcal L_{h,j}^\varepsilon}\|_{L^1_T}
\Big)
\nonumber\\
&\qquad\lesssim
\|(\varrho_0^\varepsilon,u_0^\varepsilon,f_0^\varepsilon)\|_{\mathcal Z_0}
+
\|(\varrho^\varepsilon,u^\varepsilon,f^\varepsilon)\|_{\mathcal Z_T}
+\sum_{j\geq0}2^{j(d/2+1)}
\|\langle t\rangle^\alpha
\mathcal R_{h,j}^\varepsilon\|_{L^1_T}.
\label{e199}
\end{align}
To analyze \(\mathcal R_{h,j}^\varepsilon\), Lemma \ref{l13} implies
\begin{align*}
&\sum_{j\geq0}2^{j(d/2+1)}
\|\langle t\rangle^\alpha
 [\dot\Delta_j,u^\varepsilon]\cdot
 \nabla(\varrho^\varepsilon,u^\varepsilon)\|_{L^1_TL^2_x}
\\
&\qquad\lesssim
\|\nabla u^\varepsilon\|_{L^1_T(\dot B^{d/2}_{2,1})}
\|\langle t\rangle^\alpha
 (\varrho^\varepsilon,u^\varepsilon)\|_
 {L^\infty_T(\dot B^{d/2+1}_{2,1})}
\lesssim
\|(\varrho^\varepsilon,u^\varepsilon,f^\varepsilon)\|_{\mathcal Z_T}\mathcal D_\alpha^\varepsilon(T).
\end{align*}
Similarly, one has
\begin{align*}
&\sum_{j\geq0}2^{j(d/2+1)}
\left\|\langle t\rangle^\alpha
 \Big(
 [\dot\Delta_j,\varrho^\varepsilon]\dive u^\varepsilon,
 [\dot\Delta_j,q(\varrho^\varepsilon)]
 \nabla\varrho^\varepsilon\Big)\right\|_{L^1_TL^2_x}
\\
&\qquad\lesssim
\|\nabla\varrho^\varepsilon\|_{L^1_T
 (\dot B^{d/2}_{2,1})}
\|\langle t\rangle^\alpha
 (\varrho^\varepsilon,u^\varepsilon)\|_
 {L^\infty_T(\dot B^{d/2+1}_{2,1})}
\lesssim
\|(\varrho^\varepsilon,u^\varepsilon,f^\varepsilon)\|_{\mathcal Z_T}\mathcal D_\alpha^\varepsilon(T).
\end{align*}
Moreover, the product and composition laws in Lemmas \ref{l11} and \ref{l12} give
\begin{align*}
&\left\|\langle t\rangle^\alpha
 \frac{a^\varepsilon u^\varepsilon}
 {1+\varrho^\varepsilon}\right\|_
 {L^1_T(\dot B^{d/2+1}_{2,1})}^{h}
+\left\|\langle t\rangle^\alpha
 \frac{\varrho^\varepsilon(b^\varepsilon-u^\varepsilon)}
 {1+\varrho^\varepsilon}\right\|_
 {L^1_T(\dot B^{d/2+1}_{2,1})}^{h}
\\
&\quad
+\|\langle t\rangle^\alpha
 u^\varepsilon\otimes b^\varepsilon\|_
 {L^1_T(\dot B^{d/2+1}_{2,1})}^{h}
\lesssim
\|(\varrho^\varepsilon,u^\varepsilon,f^\varepsilon)\|_{\mathcal Z_T}\mathcal D_\alpha^\varepsilon(T).
\end{align*}
The first and third terms follow from the weighted high-frequency
bounds. For the second one, we use \eqref{e178} and \eqref{e186}.
The proof of \eqref{e180} also controls the full
$\dot{\mathcal B}^{d/2+1}_{2,1}$ norm. In particular,
\begin{equation*}
\|\langle t\rangle^\alpha
h(\varrho^\varepsilon,u^\varepsilon,f^\varepsilon)\|_
{L^1_T(\dot{\mathcal B}^{d/2+1}_{2,1})}^{h}
\lesssim
\|(\varrho^\varepsilon,u^\varepsilon,f^\varepsilon)\|_{\mathcal Z_T}
\mathcal D_\alpha^\varepsilon(T).
\end{equation*}

For the viscous commutator, write
$u^\varepsilon=u^{\varepsilon,\ell}+u^{\varepsilon,h}$.
By Lemmas \ref{l12} and \ref{l13}, we have
\begin{align*}
&\varepsilon\sum_{j\geq0}2^{j(d/2+1)}
\left\|\langle t\rangle^\alpha
\left[\dot\Delta_j,\frac1{1+\varrho^\varepsilon}\right]
\bigl(\Delta_xu^\varepsilon
+2\nabla_x\dive_xu^\varepsilon\bigr)\right\|_
{L^1_TL^2_x}
\\
&\quad\lesssim
\left\|\langle t\rangle^\alpha
\nabla_x\frac1{1+\varrho^\varepsilon}\right\|_
{L^\infty_T(\dot B^{d/2}_{2,1})}
\varepsilon\|u^{\varepsilon,\ell}\|_
{L^1_T(\dot B^{d/2+2}_{2,1})}
\\
&\qquad\quad+
\left\|\nabla_x\frac1{1+\varrho^\varepsilon}\right\|_
{L^\infty_T(\dot B^{d/2}_{2,1})}
\varepsilon\|\langle t\rangle^\alpha
u^{\varepsilon,h}\|_
{L^1_T(\dot B^{d/2+2}_{2,1})}
\\
&\quad\lesssim
\|(\varrho^\varepsilon,u^\varepsilon,f^\varepsilon)\|_{\mathcal Z_T}
\mathcal D_\alpha^\varepsilon(T).
\end{align*}
Together with the preceding bounds, this estimate yields
\begin{equation}\label{e205}
\sum_{j\geq0}2^{j(d/2+1)}
\|\langle t\rangle^\alpha
\mathcal R_{h,j}^\varepsilon\|_{L^1_T}
\lesssim
\|(\varrho^\varepsilon,u^\varepsilon,f^\varepsilon)\|_{\mathcal Z_T}
\mathcal D_\alpha^\varepsilon(T).
\end{equation}
Equations \eqref{e95}, \eqref{e199} and \eqref{e205} then give
\begin{align*}
&\|\langle t\rangle^\alpha
(\varrho^{\varepsilon,h},u^{\varepsilon,h})\|_
{\widetilde L^\infty_T(\dot B^{d/2+1}_{2,1})}
+\|\langle t\rangle^\alpha f^{\varepsilon,h}\|_
{\widetilde L^\infty_T(\dot{\mathcal B}^{d/2+1}_{2,1})}
\\
&\quad
+\|\langle t\rangle^\alpha
(\varrho^{\varepsilon,h},u^{\varepsilon,h})\|_
{L^1_T(\dot B^{d/2+1}_{2,1})}
+\|\langle t\rangle^\alpha f^{\varepsilon,h}\|_
{L^1_T(\dot{\mathcal B}^{d/2+1}_{2,1})}
\\
&\qquad\lesssim
\|(\varrho_0^\varepsilon,u_0^\varepsilon,f_0^\varepsilon)\|_{\mathcal Z_0}
+\|(\varrho^\varepsilon,u^\varepsilon,f^\varepsilon)\|_{\mathcal Z_T}
+\|(\varrho^\varepsilon,u^\varepsilon,f^\varepsilon)\|_{\mathcal Z_T}
\mathcal D_\alpha^\varepsilon(T).
\end{align*}

For the weighted \(L^2_T\) estimates, we multiply
\eqref{e96} by
\(\langle t\rangle^{2\alpha}\) and obtain
\begin{align}
&\frac d{dt}\Big(
\langle t\rangle^{2\alpha}
\mathcal L_{h,j}^\varepsilon\Big)
+\Big(c-\frac{2\alpha}{\langle t\rangle}\Big)
\langle t\rangle^{2\alpha}
\mathcal L_{h,j}^\varepsilon
\nonumber\\
&\quad
+c\langle t\rangle^{2\alpha}\Big(
\varepsilon\|\nabla_x\dot\Delta_ju^\varepsilon\|_{L^2_x}^2
+\|\dot\Delta_j(u^\varepsilon-b^\varepsilon)\|_{L^2_x}^2
+\|\dot\Delta_j
\{\mathbf I-\mathbf P\}f^\varepsilon\|_{L^2_xL^2_{v,\nu}}^2
\Big)
\nonumber\\
&\qquad\lesssim
\langle t\rangle^{2\alpha}
\mathcal R_{h,j}^\varepsilon
\sqrt{\mathcal L_{h,j}^\varepsilon}.
\label{e206}
\end{align}
Choose $T_\alpha>0$ such that
\[
c-\frac{2\alpha}{\langle t\rangle}\geq\frac c2
\quad\text{for }t\geq T_\alpha.
\]
On $[0,\min\{T,T_\alpha\}]$, the desired estimate follows from the
unweighted high-frequency estimate and interpolation. If
$T>T_\alpha$, integrating \eqref{e206} over $[T_\alpha,T]$ and using
\[
\int_0^T\langle t\rangle^{2\alpha}
\mathcal R_{h,j}^\varepsilon
\sqrt{\mathcal L_{h,j}^\varepsilon}\,dt
\leq
\|\langle t\rangle^\alpha
\sqrt{\mathcal L_{h,j}^\varepsilon}\|_{L^\infty_T}
\|\langle t\rangle^\alpha
\mathcal R_{h,j}^\varepsilon\|_{L^1_T},
\]
we obtain, after taking the square root, multiplying by
\(2^{j(d/2+1)}\) and summing over \(j\geq0\),
\begin{align*}
&\varepsilon^{1/2}
\|\langle t\rangle^\alpha u^\varepsilon\|_
{\widetilde L^2_T(\dot B^{d/2+2}_{2,1})}^{h}
+\|\langle t\rangle^\alpha
(u^\varepsilon-b^\varepsilon)\|_
{\widetilde L^2_T(\dot B^{d/2+1}_{2,1})}^{h}
\nonumber\\
&\quad
+\|\langle t\rangle^\alpha
\{\mathbf I-\mathbf P\}f^\varepsilon\|_
{\widetilde L^2_T(\dot{\mathcal B}^{d/2+1}_{2,1,\nu})}^{h}
\nonumber\\
&\qquad\lesssim
\|(\varrho^\varepsilon,u^\varepsilon,f^\varepsilon)\|_{\mathcal Z_T}
+\sum_{j\geq0}2^{j(d/2+1)}
\left(
\|\langle t\rangle^\alpha
\sqrt{\mathcal L_{h,j}^\varepsilon}\|_{L^\infty_T}
\|\langle t\rangle^\alpha
\mathcal R_{h,j}^\varepsilon\|_{L^1_T}
\right)^{1/2}.
\end{align*}
The Cauchy--Schwarz inequality,
\eqref{e199} and
\eqref{e205} bound the last term by
\[
C\bigl(
\|(\varrho^\varepsilon,u^\varepsilon,f^\varepsilon)\|_{\mathcal Z_T}
\bigr)^{1/2}
\mathcal D_\alpha^\varepsilon(T).
\]
Consequently, one has
\begin{align}
&\varepsilon^{1/2}
\|\langle t\rangle^\alpha u^\varepsilon\|_
{\widetilde L^2_T(\dot B^{d/2+2}_{2,1})}^{h}
+\|\langle t\rangle^\alpha
(u^\varepsilon-b^\varepsilon)\|_
{\widetilde L^2_T(\dot B^{d/2+1}_{2,1})}^{h}
+\|\langle t\rangle^\alpha
\{\mathbf I-\mathbf P\}f^\varepsilon\|_
{\widetilde L^2_T(\dot{\mathcal B}^{d/2+1}_{2,1,\nu})}^{h}
\nonumber\\
&\qquad\lesssim
\|(\varrho^\varepsilon,u^\varepsilon,f^\varepsilon)\|_{\mathcal Z_T}
+\bigl(
\|(\varrho^\varepsilon,u^\varepsilon,f^\varepsilon)\|_{\mathcal Z_T}
\bigr)^{1/2}
\mathcal D_\alpha^\varepsilon(T).
\label{e207}
\end{align}

It remains to estimate the last viscous term in
\eqref{e163}. After applying $\dot\Delta_j$ to the momentum
equation and multiplying by $\langle t\rangle^\alpha$, one obtains
\begin{align*}
&\partial_t(\langle t\rangle^\alpha\dot\Delta_ju^\varepsilon)
-\varepsilon\Delta_x
(\langle t\rangle^\alpha\dot\Delta_ju^\varepsilon)
-2\varepsilon\nabla_x\dive_x
(\langle t\rangle^\alpha\dot\Delta_ju^\varepsilon)
+\langle t\rangle^\alpha\dot\Delta_ju^\varepsilon
\\
&\qquad=
\langle t\rangle^\alpha\dot\Delta_j\bigl(
b^\varepsilon-\nabla_x\varrho^\varepsilon
-u^\varepsilon\cdot\nabla_xu^\varepsilon
+g(\varrho^\varepsilon,u^\varepsilon,f^\varepsilon)
-a^\varepsilon u^\varepsilon\bigr)
+\alpha\langle t\rangle^{\alpha-1}
\dot\Delta_ju^\varepsilon.
\end{align*}
Since the Fourier symbol of
$1-\varepsilon\Delta_x-2\varepsilon\nabla_x\dive_x$ is bounded from
below by $c(1+\varepsilon\,2^{2j})$ on the support of
$\dot\Delta_j$, Duhamel's formula gives
\begin{align}
&(1+\varepsilon\,2^{2j})
\|\langle t\rangle^\alpha
\dot\Delta_ju^\varepsilon\|_{L^1_TL^2_x} \nonumber\\
&\quad\lesssim
\|\dot\Delta_ju_0^\varepsilon\|_{L^2_x}
+\|\langle t\rangle^\alpha
\dot\Delta_j(b^\varepsilon-\nabla_x\varrho^\varepsilon)\|_
{L^1_TL^2_x}+\|\langle t\rangle^\alpha
\dot\Delta_j(u^\varepsilon\cdot\nabla_xu^\varepsilon)\|_
{L^1_TL^2_x}
\nonumber\\
&\qquad
+\left\|\langle t\rangle^\alpha
\dot\Delta_j\bigl(
g(\varrho^\varepsilon,u^\varepsilon,f^\varepsilon)
-a^\varepsilon u^\varepsilon\bigr)\right\|_
{L^1_TL^2_x}
+\alpha\|\langle t\rangle^{\alpha-1}
\dot\Delta_ju^\varepsilon\|_{L^1_TL^2_x}.
\label{e210}
\end{align}
For $j\geq0$, multiply \eqref{e210} by $2^{jd/2}$ and sum over
$j\geq0$. The left-hand side controls
\[
\|\langle t\rangle^\alpha u^\varepsilon\|_
{L^1_T(\dot B^{d/2}_{2,1})}^{h}
+\varepsilon
\|\langle t\rangle^\alpha u^\varepsilon\|_
{L^1_T(\dot B^{d/2+2}_{2,1})}^{h}.
\]
For $t\geq2\alpha$, the term arising from the derivative of the
weight is absorbed by the first norm on the left, since
$\alpha/\langle t\rangle\leq1/2$. Its contribution on
$[0,\min\{T,2\alpha\}]$ is bounded by the unweighted
high-frequency estimate. Moreover, \eqref{e199} and \eqref{e205}
control the initial term and
$\langle t\rangle^\alpha(b^\varepsilon-\nabla_x\varrho^\varepsilon)$.
It remains to estimate the high-frequency nonlinear terms in
$L^1_T(\dot B^{d/2}_{2,1})$. To that end, we write
\[
(u^\varepsilon\cdot\nabla u^\varepsilon)^h
=\big(u^{\varepsilon,\ell}\cdot\nabla u^{\varepsilon,\ell}
+u^{\varepsilon,h}\cdot\nabla u^{\varepsilon,\ell}
+u^{\varepsilon,\ell}\cdot\nabla u^{\varepsilon,h}
+u^{\varepsilon,h}\cdot\nabla u^{\varepsilon,h}\big)^h.
\]
For the low--low term, the high-frequency localization and
\eqref{uv1} give
\begin{align*}
\|\langle t\rangle^\alpha
(u^{\varepsilon,\ell}\cdot\nabla u^{\varepsilon,\ell})^h\|_
{L^1_T(\dot B^{d/2}_{2,1})}
&\lesssim
\|\langle t\rangle^\alpha
u^{\varepsilon,\ell}\cdot\nabla u^{\varepsilon,\ell}\|_
{L^1_T(\dot B^{d/2+1}_{2,1})}\\
&\lesssim
\|u^{\varepsilon,\ell}\|_{L^2_T(\dot B^{d/2}_{2,1})}
\|\langle t\rangle^\alpha u^{\varepsilon,\ell}\|_{L^2_T(\dot B^{d/2+2}_{2,1})}\\
&\qquad
+\|u^{\varepsilon,\ell}\|_{L^1_T(\dot B^{d/2+1}_{2,1})}
\|\langle t\rangle^\alpha u^{\varepsilon,\ell}\|_{L^\infty_T(\dot B^{d/2+1}_{2,1})}.
\end{align*}
For the remaining three terms, the product estimate gives
\begin{align*}
&\|\langle t\rangle^\alpha
\big(u^{\varepsilon,h}\cdot\nabla u^{\varepsilon,\ell}
+u^{\varepsilon,\ell}\cdot\nabla u^{\varepsilon,h}
+u^{\varepsilon,h}\cdot\nabla u^{\varepsilon,h}\big)^h\|_
{L^1_T(\dot B^{d/2}_{2,1})}\\
&\quad\lesssim
\|u^{\varepsilon,h}\|_{L^2_T(\dot B^{d/2+1}_{2,1})}
\|\langle t\rangle^\alpha u^{\varepsilon,\ell}\|_{L^2_T(\dot B^{d/2+1}_{2,1})}
+\|u^{\varepsilon,\ell}\|_{L^2_T(\dot B^{d/2}_{2,1})}
\|\langle t\rangle^\alpha u^{\varepsilon,h}\|_{L^2_T(\dot B^{d/2+1}_{2,1})}\\
&\qquad
+\|u^{\varepsilon,h}\|_{L^2_T(\dot B^{d/2+1}_{2,1})}
\|\langle t\rangle^\alpha u^{\varepsilon,h}\|_{L^2_T(\dot B^{d/2+1}_{2,1})}.
\end{align*}
Thus, one has
\[
\|\langle t\rangle^\alpha
(u^\varepsilon\cdot\nabla u^\varepsilon)^h\|_
{L^1_T(\dot B^{d/2}_{2,1})}
\lesssim
\|(\varrho^\varepsilon,u^\varepsilon,f^\varepsilon)\|_{\mathcal Z_T}
\mathcal D_\alpha^\varepsilon(T).
\]
Estimating $a^\varepsilon u^\varepsilon$ and
$g^0(\varrho^\varepsilon,u^\varepsilon,f^\varepsilon)$ in the same way, we obtain
\begin{align*}
&\sum_{j\geq0}2^{j d/2}
\frac{\varepsilon\,2^{2j}}{1+\varepsilon\,2^{2j}}
\Big(
\|\langle t\rangle^\alpha
\dot\Delta_j(u^\varepsilon\cdot\nabla_xu^\varepsilon)\|_
{L^1_TL^2_x}
+\|\langle t\rangle^\alpha
\dot\Delta_j(a^\varepsilon u^\varepsilon)\|_
{L^1_TL^2_x}
\\
&\hspace{34mm}
+\|\langle t\rangle^\alpha
\dot\Delta_jg^0(\varrho^\varepsilon,u^\varepsilon,f^\varepsilon)\|_
{L^1_TL^2_x}
\Big)
\lesssim
\|(\varrho^\varepsilon,u^\varepsilon,f^\varepsilon)\|_{\mathcal Z_T}
\mathcal D_\alpha^\varepsilon(T).
\end{align*}
For the viscous part of $g$, Lemma \ref{l12}, the product estimate
and \eqref{e8} give
\begin{align*}
&\varepsilon\left\|\langle t\rangle^\alpha
\frac{\varrho^\varepsilon}{1+\varrho^\varepsilon}
\bigl(\Delta_xu^\varepsilon
+2\nabla_x\dive_xu^\varepsilon\bigr)\right\|_
{L^1_T(\dot B^{d/2}_{2,1})}^{h}
\\
&\quad\lesssim
\left\|\langle t\rangle^\alpha
\frac{\varrho^\varepsilon}{1+\varrho^\varepsilon}\right\|_
{L^\infty_T(\dot B^{d/2}_{2,1})}
\varepsilon\|u^{\varepsilon,\ell}\|_
{L^1_T(\dot B^{d/2+2}_{2,1})}+
\left\|\frac{\varrho^\varepsilon}
{1+\varrho^\varepsilon}\right\|_
{L^\infty_T(\dot B^{d/2}_{2,1})}
\varepsilon\|\langle t\rangle^\alpha
u^{\varepsilon,h}\|_
{L^1_T(\dot B^{d/2+2}_{2,1})}
\\
&\quad\lesssim
\|(\varrho^\varepsilon,u^\varepsilon,f^\varepsilon)\|_{\mathcal Z_T}
\mathcal D_\alpha^\varepsilon(T)
+\|(\varrho^\varepsilon,u^\varepsilon,f^\varepsilon)\|_{\mathcal Z_T}
\varepsilon\|\langle t\rangle^\alpha
u^{\varepsilon,h}\|_
{L^1_T(\dot B^{d/2+2}_{2,1})}.
\end{align*}
For sufficiently small initial data, we absorb the last term into the
left-hand side.

The preceding estimates yield
\begin{align}
\varepsilon\|\langle t\rangle^\alpha u^{\varepsilon,h}\|_
{L^1_T(\dot B^{d/2+2}_{2,1})}&\lesssim
C_\eta
\bigl(
\|(\varrho^\varepsilon,u^\varepsilon,f^\varepsilon)\|_{\mathcal Z_T}
+\mathcal N_{\sigma_1}^\varepsilon(T)
\bigr)
\langle T\rangle^{\alpha-\frac{1}{2}(\frac{d}{2}+1-\sigma_1)}
\nonumber\\
&\qquad\quad
+\Big(
\eta+
\bigl(
\|(\varrho^\varepsilon,u^\varepsilon,f^\varepsilon)\|_{\mathcal Z_T}
\bigr)^{1/2}
\Big)
\mathcal D_\alpha^\varepsilon(T).
\label{e214}
\end{align}
Finally, by \eqref{e192}, \eqref{e194}, \eqref{e199}, \eqref{e205},
\eqref{e207} and \eqref{e214}, together with Proposition \ref{p3}, we conclude \eqref{e164}.
\end{proof}

\subsection{Proof of the decay estimates}

\begin{proof}[Proof of Theorem \ref{t5}]
We first choose \(\eta>0\) sufficiently small and then reduce the
admissible upper bound for \(\delta_1\), if necessary, to absorb
the last term in \eqref{e164}. By Theorem
\ref{t2} and Proposition \ref{p3}, we obtain
\begin{equation}\label{e215}
 \mathcal D_\alpha^\varepsilon(T)
 \le C K_0 \langle T\rangle^{\alpha-\frac{1}{2}(\frac{d}{2}+1-\sigma_1)},
\end{equation}
from which we infer
\begin{equation}\label{e216}
 \|(\varrho^\varepsilon,u^\varepsilon)(t)\|_
 {\dot B^{d/2+1}_{2,1}}
 +\|f^\varepsilon(t)\|_
 {\dot{\mathcal B}^{d/2+1}_{2,1}}
 \le C K_0
 \langle t\rangle^{-\frac{1}{2}(\frac{d}{2}+1-\sigma_1)}.
\end{equation}
For \(\sigma_1<\sigma<d/2+1\), interpolating
\eqref{e151} and \eqref{e216} by Lemma
\ref{l10} gives \eqref{e13}. The endpoint
\(\sigma=d/2+1\) follows directly from \eqref{e216}.

For the relaxation variables, the moment equation and
\(\eqref{m1n}_2\) imply
\begin{align}
\partial_t(u^\varepsilon-b^\varepsilon)
+2(u^\varepsilon-b^\varepsilon)
={}&-\nabla\varrho^\varepsilon+\nabla a^\varepsilon
+\dive\Theta(\{\mathbf I-\mathbf P\}f^\varepsilon)
-u^\varepsilon\cdot\nabla u^\varepsilon
-2a^\varepsilon u^\varepsilon
\nonumber\\
&-\Big(\frac{P'(1+\varrho^\varepsilon)}
 {1+\varrho^\varepsilon}-1\Big)\nabla\varrho^\varepsilon
+\frac{\varrho^\varepsilon}{1+\varrho^\varepsilon}
 \bigl(u^\varepsilon-b^\varepsilon
 +a^\varepsilon u^\varepsilon\bigr)
\nonumber\\
&\quad
+\frac{\varepsilon}{1+\varrho^\varepsilon}
 \bigl(\Delta_xu^\varepsilon+2\nabla_x\dive_xu^\varepsilon\bigr).
\label{e217}
\end{align}
Fix \(\sigma\in(\sigma_1,d/2]\). The estimate \eqref{e13} leads to
\begin{align}
&\|(\nabla\varrho^\varepsilon,\nabla a^\varepsilon,
\dive\Theta(\{\mathbf I-\mathbf P\}f^\varepsilon))(t)\|_{
 \dot B^\sigma_{2,1}}
\le
C K_0 \langle t\rangle^{-\frac{\sigma+1-\sigma_1}{2}}.
\label{e218}
\end{align}
For \(\sigma_1<\sigma\leq d/2\), the product law yields, for example,
\[
\begin{aligned}
\|a^\varepsilon u^\varepsilon\|_{\dot B^\sigma_{2,1}}
&\lesssim
\|a^\varepsilon\|_{\dot B^\sigma_{2,1}}
\|u^\varepsilon\|_{\dot B^{d/2}_{2,1}}\leq
C K_0 \langle t\rangle^{
-\frac{\sigma-\sigma_1}{2}
-\frac{\frac{d}{2}-\sigma_1}{2}}\leq
C K_0 \langle t\rangle^{
-\frac{\sigma+1-\sigma_1}{2}},
\end{aligned}
\]
where the last inequality follows from
\(d/2-\sigma_1>1\). Together with the composition estimate in Lemma
\ref{l12}, the same product estimate implies
\begin{align}
&\|u^\varepsilon\cdot\nabla u^\varepsilon,
a^\varepsilon u^\varepsilon,
\frac{\varrho^\varepsilon}{1+\varrho^\varepsilon}
a^\varepsilon u^\varepsilon,
(q(\varrho^\varepsilon)-1)\nabla\varrho^\varepsilon\|_
{\dot B^\sigma_{2,1}}
\leq
C K_0 \langle t\rangle^{
-\frac{\sigma+1-\sigma_1}{2}},
\label{e219}\\
&\left\|
\frac{\varrho^\varepsilon}{1+\varrho^\varepsilon}
(u^\varepsilon-b^\varepsilon)\right\|_
{\dot B^\sigma_{2,1}}
\lesssim
\delta_1\|u^\varepsilon-b^\varepsilon\|_
{\dot B^\sigma_{2,1}}.
\label{e220}
\end{align}
For the viscous source, the low frequencies satisfy
\begin{equation}\label{e224}
 \varepsilon\,2^{2j}\|\dot\Delta_ju^\varepsilon\|_{L^2_x}
 \lesssim2^j\|\dot\Delta_ju^\varepsilon\|_{L^2_x}
 \quad\text{for}\quad j\leq1.
\end{equation}
Multiplying \eqref{e224} by \(2^{j\sigma}\) and summing over
\(j\leq1\) bounds the low-frequency viscous source by
\(\|u^\varepsilon\|_{\dot B^{\sigma+1}_{2,1}}^\ell\), which has
the decay rate in \eqref{e218}.

We decompose the viscous source as
\[
\frac{\varepsilon}{1+\varrho^\varepsilon}\bigl(\Delta_xu^\varepsilon+2\nabla_x\dive_xu^\varepsilon\bigr)
=\varepsilon\bigl(\Delta_xu^\varepsilon+2\nabla_x\dive_xu^\varepsilon\bigr)
-\varepsilon\frac{\varrho^\varepsilon}
{1+\varrho^\varepsilon}\bigl(\Delta_xu^\varepsilon+2\nabla_x\dive_xu^\varepsilon\bigr).
\]
We split $u^\varepsilon=u^{\varepsilon,\ell}+u^{\varepsilon,h}$. For
\(\sigma\in(\sigma _1,d/2]\), the part containing
\(u^{\varepsilon,\ell}\) satisfies
\begin{align}
&\varepsilon\left\|
\Big(\frac{\varrho^\varepsilon}
{1+\varrho^\varepsilon}\bigl(\Delta_xu^{\varepsilon,\ell}
+2\nabla_x\dive_xu^{\varepsilon,\ell}\bigr)\Big)^\ell
\right\|_{\dot B^\sigma_{2,1}}\lesssim
\|\varrho^\varepsilon\|_{\dot B^{d/2}_{2,1}}
\|u^{\varepsilon,\ell}\|_{\dot B^{\sigma+1}_{2,1}}
\le C K_0 \langle t\rangle^{-\frac{\sigma+1-\sigma_1}{2}}.
\label{e225}
\end{align}
For the part containing
\(u^{\varepsilon,h}\), we split the convolution at \(t/2\). Using \eqref{e13},
\eqref{e214} and \eqref{e215}, for
\(t\ge2\) we obtain
\begin{align}
&\int_0^te^{-2(t-\tau)}\varepsilon
\left\|\Big(\frac{\varrho^\varepsilon}
{1+\varrho^\varepsilon}\bigl(\Delta_xu^{\varepsilon,h}
+2\nabla_x\dive_xu^{\varepsilon,h}\bigr)\Big)^\ell
(\tau)\right\|_{\dot B^\sigma_{2,1}}\,d\tau
\nonumber\\
&\quad\le
C K_0 e^{-t}
+C K_0 \langle t\rangle^{-\alpha-\frac{\sigma-\sigma_1}{2}}
\varepsilon\|\langle\tau\rangle^\alpha u^{\varepsilon,h}\|_
{L^1_t(\dot B^{d/2+2}_{2,1})}
\nonumber\\
&\quad\le
C K_0 e^{-t}
+C K_0 \langle t\rangle^{-\frac{\sigma-\sigma_1}{2}
-\frac{1}{2}(\frac{d}{2}+1-\sigma_1)} \le
C K_0 \langle t\rangle^{-\frac{\sigma+1-\sigma _1}{2}}.
\label{e226}
\end{align}
At high frequencies, we first write
\[
\left(\frac{\Delta_xu^\varepsilon+2\nabla_x\dive_xu^\varepsilon}
{1+\varrho^\varepsilon}\right)^h
=
\left(\frac{\Delta_xu^{\varepsilon,\ell}+2\nabla_x\dive_xu^{\varepsilon,\ell}}
{1+\varrho^\varepsilon}\right)^h
+
\left(\frac{\Delta_xu^{\varepsilon,h}+2\nabla_x\dive_xu^{\varepsilon,h}}
{1+\varrho^\varepsilon}\right)^h.
\]
Using the weighted low-frequency estimate and
\eqref{e214}, together with the product estimates and the composition
estimate in Lemma \ref{l12}, we obtain
\begin{align}
&\varepsilon\left\|\langle t\rangle^\alpha
\left(\frac{\Delta_xu^\varepsilon
+2\nabla_x\dive_xu^\varepsilon}
{1+\varrho^\varepsilon}\right)^h\right\|_
{L^1_T(\dot B^\sigma_{2,1})}
\nonumber\\
&\quad\lesssim
\Big(1+\|\varrho^\varepsilon\|_
{L^\infty_T(\dot B^{d/2}_{2,1})}\Big)
\|\langle t\rangle^\alpha u^{\varepsilon,\ell}\|_
{L^1_T(\dot B^{d/2+3}_{2,1})}
+\Big(1+\|\varrho^\varepsilon\|_
{L^\infty_T(\dot B^{d/2+1}_{2,1})}\Big)
\varepsilon\|\langle t\rangle^\alpha u^{\varepsilon,h}\|_
{L^1_T(\dot B^{d/2+2}_{2,1})}
\nonumber\\
&\quad\lesssim C\mathcal D_\alpha^\varepsilon(T)
\quad\text{for}\quad \sigma\leq d/2.
\label{e227}
\end{align}
The same estimate without the time weight follows from
\eqref{e8}. For $t\geq2$, we split the convolution
into $[0,t/2]$ and $[t/2,t]$. Since
$e^{-2(t-\tau)}\leq e^{-t}$ on $[0,t/2]$, the early-time integral satisfies
\begin{align*}
&\int_0^{t/2} e^{-2(t-\tau)}\varepsilon
\left\|\left(\frac{\Delta_xu^\varepsilon
+2\nabla_x\dive_xu^\varepsilon}
{1+\varrho^\varepsilon}\right)^h(\tau)\right\|_
{\dot B^\sigma_{2,1}}\,d\tau
\lesssim C K_0 e^{-t}.
\end{align*}
For $\tau\in[t/2,t]$,
$\langle\tau\rangle^{-\alpha}\lesssim
\langle t\rangle^{-\alpha}$. Hence,
\eqref{e227} gives
\begin{align*}
&\int_{t/2}^t e^{-2(t-\tau)}\varepsilon
\left\|\left(\frac{\Delta_xu^\varepsilon
+2\nabla_x\dive_xu^\varepsilon}
{1+\varrho^\varepsilon}\right)^h(\tau)\right\|_
{\dot B^\sigma_{2,1}}\,d\tau\\
&\quad\lesssim
\langle t\rangle^{-\alpha}
\varepsilon\left\|\langle\tau\rangle^\alpha
\left(\frac{\Delta_xu^\varepsilon
+2\nabla_x\dive_xu^\varepsilon}
{1+\varrho^\varepsilon}\right)^h\right\|_
{L^1(t/2,t;\dot B^\sigma_{2,1})}\\
&\quad\lesssim
C\langle t\rangle^{-\alpha}
\mathcal D_\alpha^\varepsilon(t).
\end{align*}
Therefore, \eqref{e215} gives, for \(\sigma\leq d/2\),
\begin{equation}
\begin{aligned}
&\int_0^t e^{-2(t-\tau)}\varepsilon
\left\|\left(\frac{\Delta_xu^\varepsilon
+2\nabla_x\dive_xu^\varepsilon}
{1+\varrho^\varepsilon}\right)^h(\tau)\right\|_
{\dot B^\sigma_{2,1}}\,d\tau\\
&\qquad\lesssim
C K_0 e^{-t}
+C\langle t\rangle^{-\alpha}
\mathcal D_\alpha^\varepsilon(t)\\
&\qquad\leq
C K_0 \langle t\rangle^{-\frac{1}{2}(\frac{d}{2}+1-\sigma_1)}
\leq C K_0
\langle t\rangle^{-\frac{\sigma+1-\sigma_1}{2}}.
\end{aligned}
\label{e229}
\end{equation}
The unweighted estimate covers
\(0\le t\le2\).

For $\sigma_1<\sigma\le d/2$, \eqref{e217} and
the preceding source bounds imply that
$\|u^\varepsilon-b^\varepsilon\|_{\dot B^\sigma_{2,1}}$
is locally absolutely continuous in time, hence differentiable almost everywhere.
We apply Duhamel's formula to \eqref{e217} and use
\begin{equation}\label{e230}
 \int_0^t e^{-c(t-\tau)}\langle\tau\rangle^{-r}\,d\tau
 \le C_r\langle t\rangle^{-r}
 \quad\text{for}\quad r>0.
\end{equation}
This follows by splitting \([0,t]\) into \([0,t/2]\) and
\([t/2,t]\).
Equations \eqref{e218}--
\eqref{e229} imply
\begin{align*}
\|(u^\varepsilon-b^\varepsilon)(t)\|_{\dot B^\sigma_{2,1}}
&\le e^{-2t}\|u_0^\varepsilon-b_0^\varepsilon\|_
 {\dot B^\sigma_{2,1}}
+C K_0\int_0^te^{-2(t-\tau)}
 \langle\tau\rangle^{-\frac{\sigma+1-\sigma_1}{2}}\,d\tau\nonumber\\
&\quad+C\delta_1\int_0^te^{-2(t-\tau)}
 \|(u^\varepsilon-b^\varepsilon)(\tau)\|_
 {\dot B^\sigma_{2,1}}\,d\tau.
\end{align*}
Multiplying by $\langle t\rangle^{\frac{\sigma+1-\sigma_1}{2}}$
and using \eqref{e230}, we obtain
\begin{align*}
&\langle t\rangle^{\frac{\sigma+1-\sigma_1}{2}}
\|(u^\varepsilon-b^\varepsilon)(t)\|_{
 \dot B^\sigma_{2,1}}
\le C K_0+C\delta_1
\sup_{0\le\tau\le t}
\langle\tau\rangle^{\frac{\sigma+1-\sigma_1}{2}}
\|(u^\varepsilon-b^\varepsilon)(\tau)\|_{
 \dot B^\sigma_{2,1}}.
\end{align*}
After reducing \(\delta_1\), we absorb the last term in the preceding
inequality and obtain
\begin{align}
\|(u^\varepsilon-b^\varepsilon)(t)\|_{
 \dot B^\sigma_{2,1}}
&\le C K_0
\langle t\rangle^{-\frac{\sigma+1-\sigma_1}{2}}
\quad\text{for}\quad \sigma\in(\sigma_1,d/2].
\label{e233}
\end{align}

Finally, we estimate $\{\mathbf I-\mathbf P\}f^\varepsilon$.
By \eqref{e23},
\begin{align*}
&\partial_t\{\mathbf I-\mathbf P\}f^\varepsilon
+\{\mathbf I-\mathbf P\}(v\cdot\nabla_x\{\mathbf I-\mathbf P\}f^\varepsilon)
-\mathcal L\{\mathbf I-\mathbf P\}f^\varepsilon
\\
&\qquad=-(v\otimes v-\mathrm{Id}):\nabla b^\varepsilon M^{1/2}
+h(\varrho^\varepsilon,u^\varepsilon,f^\varepsilon)
+(v\otimes v-\mathrm{Id}):
 (u^\varepsilon\otimes b^\varepsilon)M^{1/2}.
\end{align*}
Applying \(\dot\Delta_j\), taking the \(L^2(\mathbb R^d_x\times\mathbb R^d_v)\) inner product with
\(\dot\Delta_j\{\mathbf I-\mathbf P\}f^\varepsilon\) and using \eqref{e3}, we obtain
\begin{align}
&\frac12\frac d{dt}\|\dot\Delta_j\{\mathbf I-\mathbf P\}f^\varepsilon\|_{L^2_{x,v}}^2
+\lambda_0\|\dot\Delta_j\{\mathbf I-\mathbf P\}f^\varepsilon\|_{L^2_xL^2_{v,\nu}}^2
\nonumber\\
&\qquad\le C
\Big(2^j\|\dot\Delta_jb^\varepsilon\|_{L^2_x}
+\|\dot\Delta_j(u^\varepsilon\otimes b^\varepsilon)\|_{L^2_x}
\Big)
\|\dot\Delta_j\{\mathbf I-\mathbf P\}f^\varepsilon\|_{L^2_{x,v}}
\nonumber\\
&\qquad\quad
+\Big|\bigl(\dot\Delta_j h(\varrho^\varepsilon,u^\varepsilon,f^\varepsilon),
 \dot\Delta_j\{\mathbf I-\mathbf P\}f^\varepsilon\bigr)_{L^2_{x,v}}\Big|.
\label{e235}
\end{align}
By \eqref{e4} and the vector-valued product estimate \eqref{uv2},
including its critical case, for $-d/2<\sigma\le d/2$, one has
\[
\|h(\varrho^\varepsilon,u^\varepsilon,f^\varepsilon)\|_
{\dot{\mathcal B}^{\sigma}_{2,1}}
\le C_\sigma
\|u^\varepsilon\|_{\dot B^{d/2}_{2,1}}
\|\{\mathbf I-\mathbf P\}f^\varepsilon\|_
{\dot{\mathcal B}^{\sigma}_{2,1,\nu}}.
\]
For $\sigma_1<\sigma\le d/2$, we divide
\eqref{e235} by the block norm
(with the usual regularization at zero) and sum with weight $2^{j\sigma}$.
By Cauchy--Schwarz,
\[
\|\{\mathbf I-\mathbf P\}f^\varepsilon\|_
{\dot{\mathcal B}^{\sigma}_{2,1,\nu}}^2
\le
\|\{\mathbf I-\mathbf P\}f^\varepsilon\|_
{\dot{\mathcal B}^{\sigma}_{2,1}}
\sum_j2^{j\sigma}
\frac{\|\dot\Delta_j\{\mathbf I-\mathbf P\}f^\varepsilon\|_
{L^2_xL^2_{v,\nu}}^2}
{\|\dot\Delta_j\{\mathbf I-\mathbf P\}f^\varepsilon\|_{L^2_{x,v}}}.
\]
The dyadic energy identities and the time integrability of the source
and dissipation terms ensure that
$\|\{\mathbf I-\mathbf P\}f^\varepsilon\|_{\dot{\mathcal B}^\sigma_{2,1}}$
is locally absolutely continuous, so its time derivative below is understood
almost everywhere.
The preceding inequality and Young's inequality bound the $h$ term
by half of the summed dissipation plus
$C_\sigma\|u^\varepsilon\|_{\dot B^{d/2}_{2,1}}^2
\|\{\mathbf I-\mathbf P\}f^\varepsilon\|_{
\dot{\mathcal B}^{\sigma}_{2,1}}$.
Since $\|\cdot\|_{L^2_{v,\nu}}\ge\|\cdot\|_{L^2_v}$, it follows that
\[
\begin{aligned}
&\frac{d}{dt}\|\{\mathbf I-\mathbf P\}f^\varepsilon\|_
{\dot{\mathcal B}^{\sigma}_{2,1}}
+c\|\{\mathbf I-\mathbf P\}f^\varepsilon\|_
{\dot{\mathcal B}^{\sigma}_{2,1}}\\
&\quad\le
C_\sigma\|u^\varepsilon\|_{\dot B^{d/2}_{2,1}}^2
\|\{\mathbf I-\mathbf P\}f^\varepsilon\|_
{\dot{\mathcal B}^{\sigma}_{2,1}}
+C\bigl(\|\nabla_xb^\varepsilon\|_{\dot B^\sigma_{2,1}}
+\|u^\varepsilon\otimes b^\varepsilon\|_{\dot B^\sigma_{2,1}}\bigr),
\end{aligned}
\]
where $c=\lambda_0/2$. Interpolating the
$L^\infty_t\dot B^{d/2-1}_{2,1}$ and $L^1_t\dot B^{d/2+1}_{2,1}$
bounds in \eqref{e8} gives
$\int_0^\infty\|u^\varepsilon(\tau)\|_{\dot B^{d/2}_{2,1}}^2\,d\tau
\lesssim\delta_1^2$. Hence, the integrating factor is bounded by
$e^{C_\sigma\delta_1^2}$, independently of $t$ and $\varepsilon$.
It follows that
\[
\begin{aligned}
\|\{\mathbf I-\mathbf P\}f^\varepsilon(t)\|_
{\dot{\mathcal B}^\sigma_{2,1}}
\lesssim{}&e^{-ct}\|\{\mathbf I-\mathbf P\}f^\varepsilon_0\|_
{\dot{\mathcal B}^\sigma_{2,1}}\\
&+\int_0^t e^{-c(t-\tau)}\Big(
\|\nabla_xb^\varepsilon(\tau)\|_{\dot B^\sigma_{2,1}}
+\|u^\varepsilon\otimes b^\varepsilon(\tau)\|_{\dot B^\sigma_{2,1}}
\Big)\,d\tau.
\end{aligned}
\]
Equation \eqref{e13} and the product estimate give
\[
\|\nabla_xb^\varepsilon(t)\|_{\dot B^\sigma_{2,1}}
+\|u^\varepsilon\otimes b^\varepsilon(t)\|_{\dot B^\sigma_{2,1}}
\le C K_0 \langle t\rangle^{-\frac{\sigma+1-\sigma_1}{2}}.
\]
Combining the preceding estimate with
\eqref{e230}, we obtain
\begin{equation}
\|\{\mathbf I-\mathbf P\}f^\varepsilon(t)\|_
{\dot{\mathcal B}^\sigma_{2,1}}
\le C K_0
\langle t\rangle^{-\frac{\sigma+1-\sigma_1}{2}}
\quad\text{for}\quad \sigma_1<\sigma\leq d/2.
\label{e236}
\end{equation}
Therefore, \eqref{e233} and
\eqref{e236} imply the improved decay estimate 
\eqref{e16} for $\sigma_1<\sigma\leq d/2$.

If \(\varepsilon=0\), the terms
\eqref{e224}--\eqref{e229} are absent, and the remaining argument is
unchanged.
\end{proof}

\appendix
\section{Besov spaces and some analytic tools}

This appendix contains the notation and the main analytic estimates used
throughout the paper. We refer to \cite[Chapters 2--3]{bahouri1}
for a systematic presentation of homogeneous Besov spaces.

Fix a radial function \(\varphi\in C^\infty_c(\mathbb R^d)\) such that
\[
 \operatorname{supp}\varphi\subset
 \{\xi:3/4\le|\xi|\le8/3\}\quad \text{and}\quad 
 \sum_{j\in\mathbb Z}\varphi(2^{-j}\xi)=1
 \quad(\xi\ne0).
\]
The homogeneous dyadic blocks are defined by
\[
 \dot\Delta_ju
 :=\mathcal F^{-1}\bigl(\varphi(2^{-j}\cdot)\widehat u\bigr)
 =2^{jd}\check\varphi(2^j\cdot)*u
 \quad\text{where}\quad \check\varphi:=\mathcal F^{-1}\varphi.
\]
To avoid the usual ambiguity caused by polynomials, all homogeneous spaces
below are understood in \(\mathcal S'_h(\mathbb R^d)\), the standard
realization of tempered distributions for which the homogeneous
Littlewood--Paley decomposition is valid. In particular, we have
\[
 u=\sum_{j\in\mathbb Z}\dot\Delta_ju\quad\hbox{in }\mathcal S'
 \quad\text{and}\quad 
 \dot\Delta_j\dot\Delta_k u=0\quad\hbox{if }|j-k|\ge2.
\]
Given \(s\in\mathbb R\) and \(1\le p,r\le\infty\), the homogeneous Besov norm is defined by
\[
 \|u\|_{\dot B^s_{p,r}}
 :=\left\|\bigl(2^{js}\|\dot\Delta_ju\|_{L^p_x}\bigr)_
 {j\in\mathbb Z}\right\|_{\ell^r}.
\]
Then \(\dot B^s_{p,r}\) is the set of all \(u\in\mathcal S'_h\) for
which the above quantity is finite. If \(f=f(x,v)\), only the space variable
is localized, and we write
\[
 \|f\|_{\dot{\mathcal B}^s_{p,r}}
 :=\left\|\bigl(2^{js}
 \|\dot\Delta_jf\|_{L^p_xL^2_v}\bigr)_
 {j\in\mathbb Z}\right\|_{\ell^r},
\]
i.e., \(\dot B^s_{p,r}(\mathbb R^d_x;L^2(\mathbb R^d_v))\).
To describe the velocity dissipation norm in \eqref{L12}, set
\[
 \|f\|_{\dot{\mathcal B}^s_{p,r,\nu}}
 :=\left\|\bigl(2^{js}
 \|\dot\Delta_jf\|_{L^p_xL^2_{v,\nu}}\bigr)_
 {j\in\mathbb Z}\right\|_{\ell^r}.
\]

For \(T>0\) and \(1\le q\le\infty\), the corresponding
Chemin--Lerner norm is given by
\[
 \|u\|_{\widetilde L^q_T(\dot B^s_{p,r})}
 :=\left\|\bigl(2^{js}
 \|\dot\Delta_ju\|_{L^q(0,T;L^p_x)}\bigr)_
 {j\in\mathbb Z}\right\|_{\ell^r}.
\]
By Minkowski's inequality, one has
\[
\|u\|_{\widetilde L^q_T(\dot B^s_{p,r})}
\leq\;(\text{resp. }\geq)\;
\|u\|_{L^q_T(\dot B^s_{p,r})}
\quad\text{if }r\geq q
\;(\text{resp. }r\leq q).
\]
The notation \(\widetilde L^q_T(\dot{\mathcal B}^s_{p,r})\) is
defined in the same way, and so is
\(\widetilde L^q_T(\dot{\mathcal B}^s_{p,r,\nu})\).

Consider the low-high frequency decomposition $u=u^\ell+u^h$ with
\[
 u^\ell:=\sum_{j\leq0}\dot\Delta_ju
 \quad \text{and}\quad u^h:=u-u^\ell,
\]
and the low-high frequency Besov norms
\[
 \|u\|_{\dot B^s_{p,r}}^\ell
 :=\left\|\bigl(2^{js}\|\dot\Delta_ju\|_{L^p_x}\bigr)_{j\leq1}\right\|_{\ell^r}
 \quad\text{and}\quad 
 \|u\|_{\dot B^s_{p,r}}^h
 :=\left\|\bigl(2^{js}\|\dot\Delta_ju\|_{L^p_x}\bigr)_{j\geq0}\right\|_{\ell^r}.
\]
 In particular, we have
$\|u^\ell\|_{\dot B^s_{p,1}}
=\|u^\ell\|_{\dot B^s_{p,1}}^\ell$ and
$\|u^h\|_{\dot B^s_{p,1}}
=\|u^h\|_{\dot B^s_{p,1}}^h$. For every \(s'>0\), we also have
\begin{equation*}
 \|u^\ell\|_{\dot B^s_{p,1}}
 \lesssim\|u\|_{\dot B^s_{p,1}}^\ell
 \lesssim\|u\|_{\dot B^{s-s'}_{p,1}}^\ell
 \quad\text{and}\quad 
 \|u^h\|_{\dot B^s_{p,1}}
 \lesssim\|u\|_{\dot B^s_{p,1}}^h
 \lesssim\|u\|_{\dot B^{s+s'}_{p,1}}^h.
\end{equation*}
The same argument applies to the Chemin--Lerner norms and to the mixed
\(L^p_xL^2_v\) norms.

We record below the Besov estimates used in the proof. Their
time-dependent versions follow directly from H\"older's inequality in
time, while the standard Bernstein inequalities will be used without
further comment.

\begin{lemma}\label{l10}
The following properties hold:
\begin{itemize}
\item For \(s\in\mathbb R\), \(1\le p_1\le p_2\le\infty\) and
\(1\le r_1\le r_2\le\infty\), we have
\(\dot B^s_{p_1,r_1}\hookrightarrow
\dot B^{s-d(1/p_1-1/p_2)}_{p_2,r_2}\).

\item For \(1\le p\le q\le\infty\), we have
\(\dot B^0_{p,1}\hookrightarrow L^p(\mathbb R^d_x)\hookrightarrow
\dot B^0_{p,\infty}\hookrightarrow\dot B^\sigma_{q,\infty},
\) where \(\sigma:=-d(1/p-1/q)\).

\item If \(p<\infty\), then
\(\dot B^{d/p}_{p,1}\hookrightarrow C_0(\mathbb R^d)\). Moreover,
for \(s>0\) and \(\delta>0\), we have
\(
H^{s+\delta}(\mathbb R^d_x)\hookrightarrow\dot B^s_{2,1}
\hookrightarrow\dot H^s(\mathbb R^d_x)
\).

\item If \(s_1<s_2\) and \(0<\theta<1\), then we have
\(\|u\|_{\dot B^{\theta s_1+(1-\theta)s_2}_{p,1}}
\lesssim
\|u\|_{\dot B^{s_1}_{p,\infty}}^\theta
\|u\|_{\dot B^{s_2}_{p,\infty}}^{1-\theta}\).

\end{itemize}
\end{lemma}
The same embeddings hold for the \(L^2(\mathbb R^d_v)\)-valued Besov spaces. In
particular, $ \dot{\mathcal B}^{d/2}_{2,1}
 =\dot B^{d/2}_{2,1}(\mathbb R^d_x;L^2(\mathbb R^d_v))
 \hookrightarrow L^\infty(\mathbb R^d_x;L^2(\mathbb R^d_v))$ holds.

We use the following product laws.

\begin{lemma}\label{l11}
The following statements hold:
\begin{itemize}
\item Let \(s>0\) and \(1\le p,r\le\infty\). Then
\(\dot B^s_{p,r}\cap\dot B^{d/2}_{2,1}\) is an algebra, and the
following estimate holds:
\begin{equation}\label{uv1}
\|uv\|_{\dot B^{d/2+1}_{2,1}}
\lesssim
\|u\|_{\dot B^{d/2}_{2,1}}
\|v\|_{\dot B^{d/2+1}_{2,1}}
+\|u\|_{\dot B^{d/2+1}_{2,1}}
\|v\|_{\dot B^{d/2}_{2,1}}.
\end{equation}

\item Let \(2\le p\le\infty\), \(1\le r\le\infty\) and suppose that
\(s_1<d/p\), \(s_2\le d/p\) and \(s_1+s_2>0\).
Then, we have
\begin{equation}\label{uv2}
\|uv\|_{\dot B^{s_1+s_2-d/p}_{p,r}}
\lesssim
\|u\|_{\dot B^{s_1}_{p,r}}
\|v\|_{\dot B^{s_2}_{p,1}}.
\end{equation}

\item Let \(2\le p\le\infty\), \(1\le r\le\infty\) and assume that
\(s_1+s_2\ge0\), \(s_1\le d/p\) and \(s_2<d/p\).
Then, the following estimate holds:
\begin{equation}\label{uv3}
\|uv\|_{\dot B^{s_1+s_2-d/p}_{p,\infty}}
\lesssim
\|u\|_{\dot B^{s_1}_{p,1}}
\|v\|_{\dot B^{s_2}_{p,\infty}}.
\end{equation}
\end{itemize}
\end{lemma}

The nonlinear functions appearing in the system are handled by the following
composition estimate.

\begin{lemma}\label{l12}
Let \(Q\) be smooth on an interval \(I\) containing the range of \(z\),
and assume that \(Q(0)=0\). If \(s>0, 1\le p,r\le\infty\) and
\(z\in\dot B^s_{p,r}\cap L^\infty(\mathbb R^d_x)\), then the
following estimate holds:
\[
\|Q(z)\|_{\dot B^s_{p,r}}
\le C_z\|z\|_{\dot B^s_{p,r}},
\]
where \(C_z\) depends on \(s,p,r,d\), \(\|z\|_{L^\infty_x}\) and finitely
many derivatives of \(Q\) on \(I\).

Let \(2\le p<\infty\), $-d/p<s\le0$ and $1\le r\le\infty$. For
\(z\in\dot B^s_{p,r}\cap\dot B^{d/p}_{p,1}\), the following
estimate holds:
\[
\|Q(z)\|_{\dot B^s_{p,r}}
\le C_z\bigl(1+\|z\|_{\dot B^{d/p}_{p,1}}\bigr)
\|z\|_{\dot B^s_{p,r}}.
\]
The same estimate holds for \(s=-d/p\) if \(r=\infty\).
\end{lemma}

For the high-frequency analysis, we also need a commutator estimate.

\begin{lemma}\label{l13}
Let \(1\le p\le\infty\) and
\(-\min\{d/p,d/p'\}<s\le1+d/p\). Then
\begin{align*}
\sum_{j\in\mathbb Z}2^{js}
\|[w,\dot\Delta_j]\partial_{x_i}z\|_{L^p}
&\lesssim
\|\nabla w\|_{\dot B^{d/p}_{p,1}}
\|z\|_{\dot B^s_{p,1}}\quad\text{for}\quad i=1,\cdots, d,
\end{align*}
where \([A,B]:=AB-BA\) denotes the commutator.
\end{lemma}

We consider the Lam\'e system
\begin{equation}\label{lame}
\left\{
\begin{aligned}
&\partial_tu-\mu\Delta u-(\mu+\lambda)\nabla\dive u=g_2
\quad\text{for}\quad x\in\mathbb R^d
\quad\text{and}\quad t\in(0,T),\\
&u|_{t=0}=u_0.
\end{aligned}
\right.
\end{equation}
Its maximal regularity estimate takes the following form.

\begin{lemma}\label{heat}
Let \(T>0\), \(\mu>0\), \(2\mu+\lambda>0\),
\(s\in\mathbb R\), \(1\le r\le\infty\) and
\(1\le q_2\le q_1\le\infty\). Set
\(\nu_*:=\min\{\mu,2\mu+\lambda\}\). If $u_0\in\dot B^s_{2,r}$ and $
 g_2\in\widetilde L^{q_2}(0,T;\dot B^{s-2+2/q_2}_{2,r})$, then the solution to \eqref{lame} satisfies
\[
 \nu_*^{1/q_1}
 \|u\|_{\widetilde L^{q_1}_T
 (\dot B^{s+2/q_1}_{2,r})}
 \le C\Big(
 \|u_0\|_{\dot B^s_{2,r}}
 +\nu_*^{1/q_2-1}
 \|g_2\|_{\widetilde L^{q_2}_T
 (\dot B^{s-2+2/q_2}_{2,r})}\Big).
\]
The constant \(C\) does not depend on \(T\), \(u_0\) or \(g_2\).
\end{lemma}

We end this appendix with an estimate for the linear
Fokker--Planck equation
\begin{equation}\label{lvfp}
\left\{
\begin{aligned}
&\partial_tf+v\cdot\nabla_xf-\mathcal Lf=g_3,\\
&f|_{t=0}=f_0.
\end{aligned}
\right.
\end{equation}

\begin{lemma}\label{l15}
Let \(T>0\), \(s\in\mathbb R\) and \(1\le r\le\infty\). If $f_0\in\dot{\mathcal B}^s_{2,r}$ and $g_3\in\widetilde L^1(0,T;\dot{\mathcal B}^s_{2,r})$, then \eqref{lvfp} has a unique solution such that
\begin{equation}\label{e238}
\|f\|_{\widetilde L^\infty_T(\dot{\mathcal B}^s_{2,r})}
+\|\{\mathbf I-\mathbf P\}f\|_
 {\widetilde L^2_T(\dot{\mathcal B}^s_{2,r,\nu})}
\le C\Big(
\|f_0\|_{\dot{\mathcal B}^s_{2,r}}
+\|g_3\|_{\widetilde L^1_T(\dot{\mathcal B}^s_{2,r})}\Big).
\end{equation}
The constant \(C\) is independent of \(T\), \(f_0\) and \(g_3\).
\end{lemma}

\begin{proof}
We apply \(\dot\Delta_j\) to \eqref{lvfp} and obtain
\begin{equation*}
\partial_t\dot\Delta_jf+v\cdot\nabla_x\dot\Delta_jf
-\mathcal L\dot\Delta_jf=\dot\Delta_jg_3.
\end{equation*}
Taking the \(L^2(\mathbb R^d_x\times\mathbb R^d_v)\) inner product of this equation with
\(\dot\Delta_jf\) and using \eqref{e3}, one obtains
\[
\frac12\frac d{dt}\|\dot\Delta_jf\|_{L^2_{x,v}}^2
+\lambda_0\|
\dot\Delta_j\{\mathbf I-\mathbf P\}f\|_{L^2_xL^2_{v,\nu}}^2
\le\|\dot\Delta_jg_3\|_{L^2_{x,v}}
\|\dot\Delta_jf\|_{L^2_{x,v}}.
\]
Arguing as in
\eqref{e46}--\eqref{e47}, it follows that
\[
\|\dot\Delta_jf\|_{L^\infty_TL^2_{x,v}}
+\|
\dot\Delta_j\{\mathbf I-\mathbf P\}f\|_{L^2_TL^2_xL^2_{v,\nu}}
\lesssim
\|\dot\Delta_jf_0\|_{L^2_{x,v}}
+\|\dot\Delta_jg_3\|_{L^1_TL^2_{x,v}}.
\]
Multiplying this inequality by \(2^{js}\) and taking the \(\ell^r\)
norm yields \eqref{e238}. Finally, one can regularize the data
and pass to the limit with the preceding uniform estimates to prove the existence. The details are omitted.
\end{proof}

\section{Local well-posedness}
\label{s11}

In this appendix, we prove the local well-posedness result in Theorem
\ref{t6} for the Cauchy problem \eqref{m1n} for the NS--VFP
system with large initial data
in the critical Besov spaces and the propagation of additional
regularity on the same lifespan.

\begin{theorem}\label{t6}
Let \(d\ge2\), and let \(\mu>0\) and \(2\mu+\lambda>0\) be fixed.
Let \(\nu_*:=\min\{\mu,2\mu+\lambda\}\).
\begin{itemize}
\item Assume \(1+\varrho_0\ge c_0>0\), $ M+M^{1/2}f_0\ge0$, $ \varrho_0\in\dot B^{d/2}_{2,1}$, $u_0\in\dot B^{d/2-1}_{2,1}$ and $
 f_0\in\dot{\mathcal B}^{d/2-1}_{2,1}$. There
exists \(T>0\) such that
the Cauchy problem \eqref{m1n} for the NS--VFP system has a unique
solution satisfying
\begin{equation*}
\left\{
\begin{aligned}
&\varrho\in C([0,T];\dot B^{d/2}_{2,1}),\\
&u\in C([0,T];\dot B^{d/2-1}_{2,1})
 \cap L^1(0,T;\dot B^{d/2+1}_{2,1}),\\
&f(t)\in\dot{\mathcal B}^{d/2-1}_{2,1}
 \quad\hbox{for all }t\in[0,T]\quad\text{and}\quad
 t\mapsto\|f(t)\|_{\dot{\mathcal B}^{d/2-1}_{2,1}}
 \in C([0,T]),\\
&
 \{\mathbf I-\mathbf P\}f
 \in L^2(0,T;
 \dot{\mathcal B}^{d/2-1}_{2,1,\nu});
\end{aligned}
\right.
\end{equation*}
\item if in addition \(\varrho_0\in\dot B^{d/2-1}_{2,1}\), then $ \varrho\in C([0,T];\dot B^{d/2-1}_{2,1})$;
\item if $ (\varrho_0,u_0)\in\dot B^{d/2+1}_{2,1}$ and $f_0\in\dot{\mathcal B}^{d/2+1}_{2,1}$, then
\begin{equation*}
\left\{
\begin{aligned}
&(\varrho,u)\in C([0,T];\dot B^{d/2+1}_{2,1}),\\
&f(t)\in\dot{\mathcal B}^{d/2+1}_{2,1}
\quad\hbox{for all }t\in[0,T]\quad\text{and}\quad
t\mapsto\|f(t)\|_{\dot{\mathcal B}^{d/2+1}_{2,1}}
\in C([0,T]),\\
&u\in L^1(0,T;\dot B^{d/2+2}_{2,1}),
\quad u-b\in L^2(0,T;\dot B^{d/2+1}_{2,1})
\quad\text{and}\quad \{\mathbf I-\mathbf P\}f
\in L^2(0,T;
\dot{\mathcal B}^{d/2+1}_{2,1,\nu}).
\end{aligned}
\right.
\end{equation*}
\end{itemize}
\end{theorem}

\begin{proof}
Since \(\mu\) and \(\lambda\) are fixed, the constants below may
depend on these coefficients, \(c_0\) and the size of the initial
data. We let \(u_L\) and \(f_L\) solve
\begin{equation*}
\left\{
\begin{aligned}
&\partial_tu_L-\mu\Delta_xu_L
-(\mu+\lambda)\nabla_x\dive_xu_L=0
\quad\text{and}\quad u_L|_{t=0}=u_0,\\
&\partial_tf_L+v\cdot\nabla_xf_L-\mathcal Lf_L=0
\quad\text{and}\quad f_L|_{t=0}=f_0.
\end{aligned}
\right.
\end{equation*}
By Lemmas \ref{heat} and \ref{l15}, their uniform-in-time norms
are controlled by the initial data, while
\begin{align}
&\|u_L\|_{\widetilde L^2_T
(\dot B^{d/2}_{2,1})}
+\|u_L\|_{L^1_T
(\dot B^{d/2+1}_{2,1})}
+\|\{\mathbf I-\mathbf P\}f_L\|_
{\widetilde L^2_T
(\dot{\mathcal B}^{d/2-1}_{2,1,\nu})}
\longrightarrow0
\quad\hbox{as }T\to0.
\label{e243}
\end{align}
This short-time smallness allows \(u_0\) and \(f_0\) to be
arbitrarily large.

We construct a Picard sequence. We set
$(\varrho^1,u^1,f^1):=(\varrho_0,u_L,f_L)$. For \(n\ge1\), let
\((\varrho^{n+1},u^{n+1},f^{n+1})\) solve
\begin{equation}\label{e244}
\left\{
\begin{aligned}
&\partial_t\varrho^{n+1}
+u^n\cdot\nabla_x\varrho^{n+1}
=-(1+\varrho^{n+1})\dive_xu^n,\\
&\partial_tu^{n+1}
-\frac{1}{1+\varrho^n}\Bigl(\mu\Delta_xu^{n+1}
+(\mu+\lambda)\nabla_x\dive_xu^{n+1}\Bigr)
\\
&\qquad=-u^n\cdot\nabla_xu^n
-\nabla_x\int_0^{\varrho^n}
\frac{P'(1+r)}{1+r}\,dr
+\frac{b^n-u^n-a^n u^n}{1+\varrho^n},\\
&\partial_tf^{n+1}
+v\cdot\nabla_xf^{n+1}
-\mathcal Lf^{n+1}
=u^n\cdot vM^{1/2}
+\frac12u^n\cdot vf^n
-u^n\cdot\nabla_vf^n,\\
&(\varrho^{n+1},u^{n+1},f^{n+1})|_{t=0}
=(\varrho_0,u_0,f_0),
\end{aligned}
\right.
\end{equation}
with
\[
a^n:=\int_{\mathbb R^d_v}M^{1/2}f^n\,dv
\quad\text{and}\quad 
b^n:=\int_{\mathbb R^d_v}vM^{1/2}f^n\,dv.
\]
For a given $(\varrho^{n},u^{n},f^{n})$ in the solution space, the existence of $(\varrho^{n+1},u^{n+1},f^{n+1})$ follows from standard linear theory for transport equations, parabolic
systems with variable coefficients and the linear
Fokker--Planck equation.

We establish a priori estimates on a time interval independent of $n$.
We choose \(M>0\) sufficiently large depending only on the initial
norms, and let \(0<\eta_0\ll1\). We assume
\begin{align}
&\|\varrho^n\|_{\widetilde L^\infty_T
(\dot B^{d/2}_{2,1})}
+\|u^n\|_{\widetilde L^\infty_T
(\dot B^{d/2-1}_{2,1})}
+\|f^n\|_{\widetilde L^\infty_T
(\dot{\mathcal B}^{d/2-1}_{2,1})}
\le M
\quad\text{and}\quad 1+\varrho^n\ge\frac{c_0}{2},
\label{e245}\\
&\|u^n\|_{\widetilde L^2_T
(\dot B^{d/2}_{2,1})}
+\|u^n\|_{L^1_T
(\dot B^{d/2+1}_{2,1})}
+\|\{\mathbf I-\mathbf P\}f^n\|_
{\widetilde L^2_T
(\dot{\mathcal B}^{d/2-1}_{2,1,\nu})}
\le2\eta_0.
\label{e246}
\end{align}
By \eqref{e243}, the bounds \eqref{e245}--\eqref{e246} hold for the first
iterate after decreasing \(T\).

Applying the transport estimate to
$\eqref{e244}_1$, we obtain
\begin{align*}
&\|\varrho^{n+1}\|_{\widetilde L^\infty_T
(\dot B^{d/2}_{2,1})}\\
&\quad\le
\exp\Big(C\|u^n\|_{L^1_T
(\dot B^{d/2+1}_{2,1})}\Big)
\Big(
\|\varrho_0\|_{\dot B^{d/2}_{2,1}}
+C\bigl(1+\|\varrho^{n+1}\|_{\widetilde L^\infty_T
(\dot B^{d/2}_{2,1})}\bigr)
\|u^n\|_{L^1_T
(\dot B^{d/2+1}_{2,1})}
\Big).
\end{align*}
Hence, if \(\eta_0\) is sufficiently small, the last term may be
absorbed, and we have
\begin{equation}\label{e248}
\|\varrho^{n+1}\|_{\widetilde L^\infty_T
(\dot B^{d/2}_{2,1})}
\le
2e^{C\eta_0}
\bigl(\|\varrho_0\|_{\dot B^{d/2}_{2,1}}+C\eta_0\bigr).
\end{equation}
Moreover, the flow \(X^n\) of \(u^n\) satisfies
\[
(1+\varrho^{n+1})(t,X^n(t,x))
=(1+\varrho_0)(x)
\exp\Big(-\int_0^t
\dive_xu^n(\tau,X^n(\tau,x))\,d\tau\Big),
\]
which, together with \eqref{e246}, implies \(1+\varrho^{n+1}\ge c_0e^{-2\eta_0}\ge c_0/2\) for $\eta_0\ll 1$.

We set \(w^{n+1}:=u^{n+1}-u_L\). It follows from
\(\eqref{e244}_2\) that
\begin{equation*}
\left\{
\begin{aligned}
&\partial_tw^{n+1}
-\frac{1}{1+\varrho^n}\Bigl(\mu\Delta_xw^{n+1}
+(\mu+\lambda)\nabla_x\dive_xw^{n+1}\Bigr)\\
&\qquad=-u^n\cdot\nabla_xu^n
-\nabla_x\int_0^{\varrho^n}
\frac{P'(1+r)}{1+r}\,dr
+\frac{b^n-u^n-a^nu^n}{1+\varrho^n}\\
&\qquad\quad-\frac{\varrho^n}{1+\varrho^n}
\Bigl(\mu\Delta_xu_L
+(\mu+\lambda)\nabla_x\dive_xu_L\Bigr),\\
&w^{n+1}|_{t=0}=0.
\end{aligned}
\right.
\end{equation*}
\begin{samepage}
By the variable-coefficient Lam\'e estimate
(cf. \cite{danchin4}), we have
\begin{align}
&\|w^{n+1}\|_{\widetilde L^\infty_T
(\dot B^{d/2-1}_{2,1})}
+\sqrt{\nu_*}\|w^{n+1}\|_{\widetilde L^2_T
(\dot B^{d/2}_{2,1})}
+\nu_*\|w^{n+1}\|_{L^1_T
(\dot B^{d/2+1}_{2,1})}\nonumber\\
&\quad\le C_M\Bigg(
\|u^n\cdot\nabla_xu^n\|_{L^1_T
(\dot B^{d/2-1}_{2,1})}
+\Big\|\nabla_x\int_0^{\varrho^n}
\frac{P'(1+r)}{1+r}\,dr\Big\|_{L^1_T
(\dot B^{d/2-1}_{2,1})}\nonumber\\
&\hspace{27mm}
+\Big\|\frac{b^n-u^n-a^nu^n}{1+\varrho^n}\Big\|_
{L^1_T(\dot B^{d/2-1}_{2,1})}\nonumber\\
&\hspace{27mm}
+\Big\|
\frac{\varrho^n}{1+\varrho^n}
\Bigl(\mu\Delta_xu_L
+(\mu+\lambda)\nabla_x\dive_xu_L\Bigr)
\Big\|_{L^1_T(\dot B^{d/2-1}_{2,1})}
\Bigg),
\label{e250}
\end{align}
\end{samepage}
where \(C_M\) is independent of \(n\). To analyze the right-hand side of \eqref{e250}, we deduce from the product and
composition estimates that
\begin{align*}
&\|u^n\cdot\nabla_xu^n\|_{L^1_T
(\dot B^{d/2-1}_{2,1})}
\lesssim
\|u^n\|_{\widetilde L^2_T
(\dot B^{d/2}_{2,1})}^2
\lesssim\eta_0^2,\\
&\Big\|\nabla_x\int_0^{\varrho^n}
\frac{P'(1+r)}{1+r}\,dr\Big\|_{L^1_T
(\dot B^{d/2-1}_{2,1})}
\le C_MT,\\
&\Big\|\frac{b^n-u^n}{1+\varrho^n}\Big\|_
{L^1_T(\dot B^{d/2-1}_{2,1})}
\le C_MT,\\
&\Big\|\frac{a^nu^n}{1+\varrho^n}\Big\|_
{L^1_T(\dot B^{d/2-1}_{2,1})}
\le C_MT^{1/2}M\eta_0,\\
&\Big\|
\frac{\varrho^n}{1+\varrho^n}
\Bigl(\mu\Delta_xu_L
+(\mu+\lambda)\nabla_x\dive_xu_L\Bigr)
\Big\|_{L^1_T(\dot B^{d/2-1}_{2,1})}
\le C_M\|u_L\|_{L^1_T
(\dot B^{d/2+1}_{2,1})}.
\end{align*}
Consequently, we have
\begin{align}
&\|w^{n+1}\|_{\widetilde L^\infty_T
(\dot B^{d/2-1}_{2,1})}
+\sqrt{\nu_*}\|w^{n+1}\|_{\widetilde L^2_T
(\dot B^{d/2}_{2,1})}
+\nu_*\|w^{n+1}\|_{L^1_T
(\dot B^{d/2+1}_{2,1})}\nonumber\\
&\quad\le C_M\Big(
\|u_L\|_{L^1_T(\dot B^{d/2+1}_{2,1})}
+\eta_0^2+T^{1/2}M\eta_0+T\Big).
\label{e251}
\end{align}

We set \(h^{n+1}:=f^{n+1}-f_L\). By
\(\eqref{e244}_3\), the remainder satisfies
\begin{equation}\label{e252}
\left\{
\begin{aligned}
&\partial_th^{n+1}
+v\cdot\nabla_xh^{n+1}
-\mathcal Lh^{n+1}
=u^n\cdot vM^{1/2}
+\frac12u^n\cdot vf^n
-u^n\cdot\nabla_vf^n,\\
&h^{n+1}|_{t=0}=0.
\end{aligned}
\right.
\end{equation}
We apply Lemma \ref{l15} and use the decomposition
\(f^n=\mathbf Pf^n+\{\mathbf I-\mathbf P\}f^n\) to obtain
\begin{align}
&\|h^{n+1}\|_{\widetilde L^\infty_T
(\dot{\mathcal B}^{d/2-1}_{2,1})}
+\|\{\mathbf I-\mathbf P\}h^{n+1}\|_
{\widetilde L^2_T
(\dot{\mathcal B}^{d/2-1}_{2,1,\nu})}
\nonumber\\
&\qquad\le C\Big(
\|u^n\|_{L^1_T(\dot B^{d/2-1}_{2,1})}+\|u^n\mathbf Pf^n\|_{L^1_T
(\dot{\mathcal B}^{d/2-1}_{2,1})}
+\|u^n
\{\mathbf I-\mathbf P\}f^n\|_{L^1_T
(\dot{\mathcal B}^{d/2-1}_{2,1,\nu})}
\Big).
\label{e253}
\end{align}
The terms on the right-hand side of \eqref{e253} satisfy
\begin{align*}
&\|u^n\|_{L^1_T(\dot B^{d/2-1}_{2,1})}
\le TM,\\
&\|u^n\mathbf Pf^n\|_{L^1_T
(\dot{\mathcal B}^{d/2-1}_{2,1})}
\lesssim
T^{1/2}
\|u^n\|_{\widetilde L^2_T
(\dot B^{d/2}_{2,1})}
\|\mathbf Pf^n\|_{\widetilde L^\infty_T
(\dot{\mathcal B}^{d/2-1}_{2,1})}
\le CT^{1/2}M\eta_0,\\
&\|u^n
\{\mathbf I-\mathbf P\}f^n\|_{L^1_T
(\dot{\mathcal B}^{d/2-1}_{2,1,\nu})}
\lesssim
\|u^n\|_{\widetilde L^2_T
(\dot B^{d/2}_{2,1})}
\|\{\mathbf I-\mathbf P\}f^n\|_
{\widetilde L^2_T
(\dot{\mathcal B}^{d/2-1}_{2,1,\nu})}
\le C\eta_0^2.
\end{align*}
It thus follows that
\begin{align}
&\|h^{n+1}\|_{\widetilde L^\infty_T
(\dot{\mathcal B}^{d/2-1}_{2,1})}
+\|\{\mathbf I-\mathbf P\}h^{n+1}\|_
{\widetilde L^2_T
(\dot{\mathcal B}^{d/2-1}_{2,1,\nu})}
\le C_M\bigl(T+T^{1/2}M\eta_0+\eta_0^2\bigr).
\label{e254}
\end{align}

We first take \(M\) according to the initial norms and then choose
\(\eta_0>0\), depending on \(M\), \(c_0\), \(\mu\) and \(\lambda\),
sufficiently small. By \eqref{e243}, we may finally
decrease \(T\) so that the linear dissipative norms and all the terms
containing \(T\) or \(T^{1/2}\) in
\eqref{e251}--\eqref{e254}
are sufficiently small. Equations \eqref{e248},
\eqref{e251} and
\eqref{e254} then improve
\eqref{e245}--\eqref{e246} by replacing $n$ by $n+1$. This proves
the required bounds uniformly with respect to \(n\).

The convergence of the iteration and uniqueness follow from the
standard lower-order difference estimates;
cf. \cite{danchin4}. We omit the details. The standard theory of transport and parabolic equations implies the 
continuity of \(\varrho\) and \(u\).

It remains to justify the continuity of the norm of \(f\). For each
fixed \(j\in\mathbb Z\), we apply \(\dot\Delta_j\) to the kinetic
equation and take the \(L^2_{x,v}\) inner product with
\(\dot\Delta_jf\), we deduce for
\(0\leq t_1<t_2\leq T\) that
\begin{align*}
&\left|\|\dot\Delta_jf(t_2)\|_{L^2_{x,v}}^2
-\|\dot\Delta_jf(t_1)\|_{L^2_{x,v}}^2\right|\\
&\quad\leq2\int_{t_1}^{t_2}\bigg(
\left|\left\langle\mathcal L\dot\Delta_jf,
\dot\Delta_jf\right\rangle_{L^2_{x,v}}\right|+
\left\|\dot\Delta_j\Big(
u\cdot vM^{1/2}+a\,u\cdot vM^{1/2}\Big)
\right\|_{L^2_{x,v}}\|\dot\Delta_jf\|_{L^2_{x,v}}\\
&\qquad\quad+
\left\|\dot\Delta_j\Big(
(v\otimes v-\mathrm{Id}):(u\otimes b)M^{1/2}
\Big)\right\|_{L^2_{x,v}}
\|\dot\Delta_jf\|_{L^2_{x,v}}+
\|\dot\Delta_jh(\varrho,u,f)\|_{L^2_{x,v}}
\|\dot\Delta_jf\|_{L^2_{x,v}}\bigg)d\tau.
\end{align*}
By \eqref{e245}--\eqref{e246} and the
product estimates in Appendix A, the integrand on the right-hand side
belongs to \(L^1(0,T)\). Hence, the right-hand side tends to zero as
\(t_1\to t_2\), and
\(t\mapsto\|\dot\Delta_jf(t)\|_{L^2_{x,v}}\) is continuous.

On the other hand, the \(\widetilde L^\infty_T
(\dot{\mathcal B}^{d/2-1}_{2,1})\) bound gives
\[
 \sup_{0\leq t\leq T}\sum_{|j|>\bar J}2^{j(d/2-1)}
 \|\dot\Delta_jf(t)\|_{L^2_{x,v}}
 \leq\sum_{|j|>\bar J}2^{j(d/2-1)}
 \|\dot\Delta_jf\|_{L^\infty_TL^2_{x,v}}
 \longrightarrow0
 \quad\hbox{as }\bar J\to\infty.
\]
Consequently, the preceding bound gives
\begin{align*}
&\left|\|f(t_1)\|_{\dot{\mathcal B}^{d/2-1}_{2,1}}
-\|f(t_2)\|_{\dot{\mathcal B}^{d/2-1}_{2,1}}\right|\\
&\quad\leq
\sum_{|j|\leq\bar J}2^{j(d/2-1)}
\left|\|\dot\Delta_jf(t_1)\|_{L^2_{x,v}}
-\|\dot\Delta_jf(t_2)\|_{L^2_{x,v}}\right|+2\sum_{|j|>\bar J}2^{j(d/2-1)}
\|\dot\Delta_jf\|_{L^\infty_TL^2_{x,v}}.
\end{align*}
First letting \(t_1\to t_2\) and then
\(\bar J\to\infty\) proves
\(t\mapsto\|f(t)\|_{\dot{\mathcal B}^{d/2-1}_{2,1}}
\in C([0,T])\). Finally, the standard maximum
principle (cf. \cite{degond1}) for the Fokker--Planck equation preserves
$M+M^{1/2}f\geq0$.

If \(\varrho_0\in\dot B^{d/2-1}_{2,1}\), the transport
estimate at the lower index gives
\begin{align*}
\|\varrho\|_{\widetilde L^\infty_T
(\dot B^{d/2-1}_{2,1})}
\le{}&
\exp\Big(C\|u\|_{L^1_T
(\dot B^{d/2+1}_{2,1})}\Big)
\Big(
\|\varrho_0\|_{\dot B^{d/2-1}_{2,1}}+C\bigl(1+\|\varrho\|_{\widetilde L^\infty_T
(\dot B^{d/2}_{2,1})}\bigr)
\|u\|_{L^1_T(\dot B^{d/2}_{2,1})}
\Big).
\end{align*}
The continuity equation then gives
\(\varrho\in C([0,T];\dot B^{d/2-1}_{2,1})\).

Finally, assume the additional regularity in the third item. We apply
the dyadic energy estimates at the index \(d/2+1\), use
\eqref{e3} and estimate the nonlinear terms by the product
and composition laws in Appendix A. Consequently, a Gronwall
inequality leads to the expected \(d/2+1\)-order bounds on the same lifespan. The same finite-frequency and tail argument as above at the
index \(d/2+1\) proves the asserted continuity of the higher
Besov norm of \(f\). Furthermore,
the maximal-regularity estimate for the Lam\'e system also gives
\(u\in L^1(0,T;\dot B^{d/2+2}_{2,1})\). This completes the proof.
\end{proof}

\vspace{2mm}

\noindent \textbf{Acknowledgments} 
The problem studied in this paper originated during the second
author's short visit to R.~Danchin at LAMA in early 2020.
The second author warmly thanks R.~Danchin for his hospitality. Part of this work was carried out
between 2020 and 2022, during the second author's doctoral studies. H.-L. Li is partially supported by the National Natural Science Foundation of China Grant Nos. 12331007, 12326613 and 11931010, the Beijing Scholar Foundation and the key research project of Academy for Multidisciplinary Studies, Capital Normal University.
L.-Y. Shou is supported by the National Natural Science Foundation of China Grant Nos. 12671258 and 12301275.

\vspace{2mm}

\noindent \textbf{Conflict of interest.} The authors declare that they have no conflict of interest.

\vspace{2mm}

\noindent \textbf{Data availability statement.}
Data sharing is not applicable to this article, as no data sets were generated or analyzed during the current study.

\vspace{3mm}

\bigbreak\bigbreak
\noindent Hai-Liang Li\\
{\scshape School of Mathematical Sciences and Academy for Multidisciplinary Studies,\\
Capital Normal University, Beijing 100048, China\\
{\itshape Email address}: {\tt hailiang.li.math@gmail.com}}

\bigbreak

\noindent Ling-Yun Shou\\
{\scshape School of Mathematical Sciences, Ministry of Education Key Laboratory of NSLSCS,\\
and Key Laboratory of Jiangsu Provincial Universities of FDMTA,\\
Nanjing Normal University, Nanjing 210023, China\\
{\itshape Email address}: {\tt shoulingyun11@gmail.com}}


\begin{thebibliography}{99}

\parskip=0pt
\small

\bibitem{bae1}
H.-O. Bae, Y.-P. Choi, S.-Y. Ha and M.-J. Kang,
Asymptotic flocking dynamics of Cucker--Smale particles immersed in
compressible fluids.
\emph{Discrete Contin. Dyn. Syst.} 34(11), 4419--4458 (2014).

\bibitem{bahouri1}
H. Bahouri, J.-Y. Chemin and R. Danchin,
\emph{Fourier Analysis and Nonlinear Partial Differential Equations}.
Grundlehren der Mathematischen Wissenschaften, vol. 343,
Springer, Heidelberg (2011).

\bibitem{baranger2}
C. Baranger and L. Desvillettes,
Coupling Euler and Vlasov equations in the context of sprays:
the local-in-time, classical solutions.
\emph{J. Hyperbolic Differ. Equ.} 3(1), 1--26 (2006).

\bibitem{boudinNSV1}
L. Boudin, L. Desvillettes, C. Grandmont and A. Moussa,
Global existence of solutions for the coupled Vlasov and
Navier--Stokes equations.
\emph{Differ. Integral Equ.} 22(11--12), 1247--1271 (2009).

\bibitem{boudin2}
L. Boudin, L. Desvillettes and R. Motte,
A modeling of compressible droplets in a fluid.
\emph{Commun. Math. Sci.} 1(4), 657--669 (2003).

\bibitem{brandolese1}
L. Brandolese, L.-Y. Shou, J. Xu and P. Zhang,
Sharp decay characterization for the compressible Navier--Stokes equations.
\emph{Adv. Math.} 456, 109905 (2024).

\bibitem{CaoJiang2021}
W. Cao and P. Jiang,
Global bounded weak entropy solutions to the Euler--Vlasov equations
in fluid-particle system.
\emph{SIAM J. Math. Anal.} 53(4), 3958--3984 (2021).

\bibitem{carrillo2}
J. A. Carrillo, Y.-P. Choi and T. Karper,
On the analysis of a coupled kinetic-fluid model with local alignment forces.
\emph{Ann. Inst. H. Poincar\'e Anal. Non Lin\'eaire}
33(2), 273--307 (2016).

\bibitem{carrillo4}
J. A. Carrillo, R.-J. Duan and A. Moussa,
Global classical solution close to equilibrium to the
Vlasov--Euler--Fokker--Planck system.
\emph{Kinet. Relat. Models} 4(1), 227--258 (2011).

\bibitem{chae1}
M. Chae, K. Kang and J. Lee,
Global existence of weak and classical solutions for the
Navier--Stokes--Vlasov--Fokker--Planck equations.
\emph{J. Differential Equations} 251(9), 2431--2465 (2011).

\bibitem{chae2}
M. Chae, K. Kang and J. Lee,
Global classical solutions for a compressible fluid-particle interaction model.
\emph{J. Hyperbolic Differ. Equ.} 10, 537--562 (2013).

\bibitem{choiNSV1}
Y.-P. Choi,
Large-time behavior for the Vlasov/compressible Navier--Stokes
equations.
\emph{J. Math. Phys.} 57(7), 071501 (2016).

\bibitem{choiblowup}
Y.-P. Choi,
Finite-time blow-up phenomena of Vlasov/Navier--Stokes equations and
related systems.
\emph{J. Math. Pures Appl. (9)} 108(6), 991--1021 (2017).

\bibitem{choi4}
Y.-P. Choi and J. Jung,
Asymptotic analysis for a Vlasov--Fokker--Planck/Navier--Stokes
system in a bounded domain.
\emph{Math. Models Methods Appl. Sci.}
31(11), 2213--2295 (2021).

\bibitem{crin1}
T. Crin-Barat and R. Danchin,
Partially dissipative hyperbolic systems in the critical regularity
setting: the multi-dimensional case.
\emph{J. Math. Pures Appl. (9)} 165, 1--41 (2022).

\bibitem{danchin1}
R. Danchin,
Global existence in critical spaces for compressible Navier--Stokes equations.
\emph{Invent. Math.} 141(3), 579--614 (2000).

\bibitem{danchin4}
R. Danchin,
Well-posedness in critical spaces for barotropic viscous fluids
with truly not constant density.
\emph{Comm. Partial Differential Equations}
32(7--9), 1373--1397 (2007).

\bibitem{DanchinShou2026}
R. Danchin and L.-Y. Shou,
Large-time asymptotics of periodic two-dimensional
Vlasov--Navier--Stokes flows.
\emph{J. Lond. Math. Soc. (2)} 113(1), e70435 (2026).

\bibitem{danchin3}
R. Danchin and J. Xu,
Optimal time-decay estimates for the compressible Navier--Stokes equations
in the critical \(L^p\) framework.
\emph{Arch. Ration. Mech. Anal.} 224, 53--90 (2017).



\bibitem{degond1}
P. Degond,
Global existence of smooth solutions for the Vlasov--Fokker--Planck
equation in \(1\) and \(2\) space dimensions.
\emph{Ann. Sci. \'Ec. Norm. Sup\'er.} 19(4), 519--542 (1986).

\bibitem{duan3}
R.-J. Duan and S. Liu,
Cauchy problem on the Vlasov--Fokker--Planck equation coupled
with the compressible Euler equations through the friction force.
\emph{Kinet. Relat. Models} 6(4), 687--700 (2013).

\bibitem{goudon1}
T. Goudon, L. He, A. Moussa and P. Zhang,
The Navier--Stokes--Vlasov--Fokker--Planck system near equilibrium.
\emph{SIAM J. Math. Anal.} 42(5), 2177--2202 (2010).

\bibitem{goudon2}
T. Goudon, P.-E. Jabin and A. Vasseur,
Hydrodynamic limit for the Vlasov--Navier--Stokes equations. I:
Light particles regime.
\emph{Indiana Univ. Math. J.} 53(6), 1495--1515 (2004).

\bibitem{goudon3}
T. Goudon, P.-E. Jabin and A. Vasseur,
Hydrodynamic limit for the Vlasov--Navier--Stokes equations. II:
Fine particles regime.
\emph{Indiana Univ. Math. J.} 53(6), 1517--1536 (2004).

\bibitem{guo2}
Y. Guo,
The Boltzmann equation in the whole space.
\emph{Indiana Univ. Math. J.} 53, 1081--1094 (2004).

\bibitem{guo1}
Y. Guo and Y. Wang,
Decay of dissipative equations and negative Sobolev spaces.
\emph{Comm. Partial Differential Equations}
37, 2165--2208 (2012).

\bibitem{hamdache1}
K. Hamdache,
Global existence and large time behaviour of solutions for
the Vlasov--Stokes equations.
\emph{Jpn. J. Ind. Appl. Math.} 15(1), 51--74 (1998).

\bibitem{han1}
D. Han-Kwan, A. Moussa and I. Moyano,
Large time behavior of the Vlasov--Navier--Stokes system on the torus.
\emph{Arch. Ration. Mech. Anal.} 236, 1273--1323 (2020).

\bibitem{he1}
L. He,
On the global smooth solution to the two-dimensional fluid-particle system.
\emph{Discrete Contin. Dyn. Syst.} 27(1), 237--263 (2010).

\bibitem{jabin1}
P.-E. Jabin and B. Perthame,
Notes on mathematical problems on the dynamics of dispersed particles
interacting through a fluid.
In: N. Bellomo and M. Pulvirenti (eds.),
\emph{Modeling in Applied Sciences},
pp. 111--147. Birkh\"auser, Boston (2000).

\bibitem{lf1}
F.-C. Li, Y. Mu and D. Wang,
Strong solutions to the compressible
Navier--Stokes--Vlasov--Fokker--Planck equations:
global existence near the equilibrium and large time behavior.
\emph{SIAM J. Math. Anal.} 49(2), 984--1026 (2017).


\bibitem{LiNiShouWang2025}
F.-C. Li, J. Ni, L.-Y. Shou and D. Wang,
The incompressible inhomogeneous
Navier--Stokes--Vlasov--Fokker--Planck equations:
global well-posedness and inviscid limit.
Preprint, arXiv:2512.11220 (2025).

\bibitem{LiNiWang2026}
F.-C. Li, J. Ni and D. Wang,
Global well-posedness and inviscid limit of the compressible
Navier--Stokes--Vlasov--Fokker--Planck system with
density-dependent friction force.
Preprint, arXiv:2603.07411 (2026).





\bibitem{lishouNSV1}
H.-L. Li and L.-Y. Shou,
Global well-posedness of one-dimensional compressible
Navier--Stokes--Vlasov system.
\emph{J. Differential Equations} 280, 841--890 (2021).

\bibitem{lishouNSV2}
H.-L. Li and L.-Y. Shou,
Asymptotical behavior of one-dimensional compressible
Navier--Stokes--Vlasov system.
\emph{Sci. Sin. Math.} 51(6), 985--1002 (2021).

\bibitem{lishou1}
H.-L. Li and L.-Y. Shou,
Global existence and optimal time-decay rates of the compressible
Navier--Stokes--Euler system.
\emph{SIAM J. Math. Anal.} 55(3), 1810--1846 (2023).

\bibitem{LiShouZhang2025}
H.-L. Li, L.-Y. Shou and Y. Zhang,
Exponential stability of the inhomogeneous
Navier--Stokes--Vlasov system in vacuum.
\emph{Kinet. Relat. Models} 18(2), 252--285 (2025).

\bibitem{lhl1}
H.-L. Li, J. Sun, G. Zhang and M. Zhong,
The behaviors of the solutions to the compressible
Navier--Stokes (Euler)/Fokker--Planck equations.
Preprint (2018).


\bibitem{LWW}
H.-L. Li, T. Wang and Y. Wang,
Wave phenomena to the three-dimensional fluid-particle model.
\emph{Arch. Ration. Mech. Anal.} 243, 1019--1089 (2022).

\bibitem{LiuShouXu2025}
J. Liu, L.-Y. Shou and J. Xu,
The Boltzmann equation in the homogeneous critical regularity framework.
\emph{SIAM J. Math. Anal.} 57(4), 4314--4357 (2025).

\bibitem{liu1}
T.-P. Liu and S.-H. Yu,
Boltzmann equation: Micro--macro decompositions and positivity of
shock profiles.
\emph{Comm. Math. Phys.} 246(1), 133--179 (2004).

\bibitem{mellet2}
A. Mellet and A. Vasseur,
Global weak solutions for a
Vlasov--Fokker--Planck/Navier--Stokes system of equations.
\emph{Math. Models Methods Appl. Sci.} 17(7), 1039--1063 (2007).

\bibitem{mellet1}
A. Mellet and A. Vasseur,
Asymptotic analysis for a
Vlasov--Fokker--Planck/compressible Navier--Stokes system of equations.
\emph{Comm. Math. Phys.} 281, 573--596 (2008).

\bibitem{so1}
V. Sohinger and R. M. Strain,
The Boltzmann equation, Besov spaces, and optimal time decay rates
in \(\mathbb R^n_x\).
\emph{Adv. Math.} 261, 274--332 (2014).

\bibitem{villani1}
C. Villani,
\emph{Hypocoercivity}.
Mem. Amer. Math. Soc. 202, no. 950 (2009).

\bibitem{wangd1}
D. Wang and C. Yu,
Global weak solution to the inhomogeneous
Navier--Stokes--Vlasov equations.
\emph{J. Differential Equations} 259, 3976--4008 (2015).

\bibitem{williams1}
F. A. Williams,
Spray combustion and atomization.
\emph{Phys. Fluids} 1(6), 541--545 (1958).

\bibitem{xin1}
Z. Xin and J. Xu,
Optimal decay for the compressible Navier--Stokes equations
without additional smallness assumptions.
\emph{J. Differential Equations} 274, 543--575 (2021).

\bibitem{xu1}
J. Xu and S. Kawashima,
Global classical solutions for partially dissipative hyperbolic
systems of balance laws.
\emph{Arch. Ration. Mech. Anal.} 211, 513--553 (2014).

\bibitem{xu2}
J. Xu and S. Kawashima,
The optimal decay estimates on the framework of Besov spaces
for generally dissipative systems.
\emph{Arch. Ration. Mech. Anal.} 218, 275--315 (2015).

\bibitem{yuNSV1}
C. Yu,
Global weak solutions to the incompressible
Navier--Stokes--Vlasov equations.
\emph{J. Math. Pures Appl. (9)} 100(2), 275--293 (2013).

\end{thebibliography}
\end{document}